\documentclass{amsart}
\usepackage[foot]{amsaddr}
\usepackage{amsmath}
\usepackage{amsthm}
\usepackage{amsfonts}
\usepackage{amssymb}
\usepackage{bbm}
\usepackage{graphicx}
\usepackage{tikz,subfig}
\usepackage{natbib}
\usepackage{enumitem}
\usepackage{tikz-cd}
\usepackage{algorithm2e}
\usepackage[text={440pt,575pt},headheight=9pt,centering]{geometry}

\usepackage{hyperref}
\hypersetup{
	colorlinks=true,
	linkcolor=blue,
	filecolor=magenta,
	urlcolor=cyan,
	citecolor=blue
}

\allowdisplaybreaks

\theoremstyle{plain}
\newtheorem{thm}{Theorem}[section]
\newtheorem{prop}[thm]{Proposition}
\newtheorem{cor}[thm]{Corollary}
\newtheorem{lemma}[thm]{Lemma}

\theoremstyle{remark}
\newtheorem{defi}{Definition}
\newtheorem{ex}{Example}

\newtheorem{remark}[thm]{Remark}

\usepackage{algorithm2e,bbm,tikz-cd}

\renewcommand{\S}{\mathbb{S}}

\newcommand{\MTPtwo}{$ \text{MTP}_2 $ }

\newcommand{\EMTPtwo}{$ \text{EMTP}_2 $ }

\newcommand{\llangle}{\langle\!\langle}
\newcommand{\rrangle}{\rangle\!\rangle}
\newcommand{\bigllangle}{\big\langle\hspace*{-0.5pt}\big\langle}
\newcommand{\bigrrangle}{\big\rangle\hspace*{-0.5pt}\big\rangle}
\newcommand{\perpp}{\perp \!\!\! \perp}
\newcommand{\rrvert}{\vert}
\newcommand{\rrVert}{\Vert}
\newcommand{\llvert}{\vert}
\newcommand{\llVert}{\Vert}

\DeclareMathOperator*{\argmax}{\operatorname{arg max}}

\DeclareFontFamily{U}{mathx}{\hyphenchar\font45}
\DeclareFontShape{U}{mathx}{m}{n}{
	<5> <6> <7> <8> <9> <10>
	<10.95> <12> <14.4> <17.28> <20.74> <24.88>
	mathx10
}{}
\DeclareSymbolFont{mathx}{U}{mathx}{m}{n}
\DeclareFontSubstitution{U}{mathx}{m}{n}
\DeclareMathAccent{\widecheck}{0}{mathx}{"71}

\begin{document}

\title{Total positivity in multivariate extremes}
\author[F.~R\"ottger]{Frank R\"ottger$^1$}
\email{frank.roettger@unige.ch}
\author[S.~Engelke]{Sebastian Engelke$^1$}
\email{sebastian.engelke@unige.ch}
\address{$^1$Universit\'e de Gen\`eve, Switzerland}
\author[P.~Zwiernik]{Piotr Zwiernik$^2$}
\email{piotr.zwiernik@utoronto.ca}
\address{$^2$University of Toronto, ON, Canada}
\date{\today}
\begin{abstract}
Positive dependence is present in many real world data sets and has appealing
stochastic properties that can be exploited in statistical modeling and
in estimation. In particular, the notion of multivariate total positivity
of order 2 ($ \mathrm{MTP}_{2} $) is a convex constraint and acts as
an implicit regularizer in the Gaussian case. We study positive dependence
in multivariate extremes and introduce $ \mathrm{EMTP}_{2} $, an extremal
version of $ \mathrm{MTP}_{2} $. This notion turns out to appear prominently
in extremes, and in fact, it is satisfied by many classical models. For
a H\"usler--Reiss distribution, the analogue of a Gaussian distribution
in extremes, we show that it is $ \mathrm{EMTP}_{2} $ if and only if its precision
matrix is a Laplacian of a connected graph. We propose an estimator for
the parameters of the H\"usler--Reiss distribution under
$ \mathrm{EMTP}_{2} $ as the solution of a convex optimization problem with
Laplacian constraint. We prove that this estimator is consistent and typically
yields a sparse model with possibly nondecomposable extremal graphical
structure. Applying our methods to a data set of Danube River flows, we
illustrate this regularization and the superior performance compared to
existing methods.
\end{abstract}
\maketitle

\section{Introduction}
\label{sec1}

Multivariate dependence modeling for complex data relies on parsimonious
models to avoid overfitting, allows for interpretation and enables inference
in high dimensions. One approach to regularize models is the framework
of conditional independence and sparsity
(e.g., \citet{Lauritzen96, wainwright2008graphical}). While the sparsity
assumption is often justified, fitting typically requires the choice of
tuning parameters, and it may lead to suboptimal models. An alternative
to this approach is the notion of positive dependence, which can also be
seen as an implicit regularizer through a distributional constraint. Positive
dependence has been extensively studied with connections to probability
theory and statistical physics
(\citet{FKG71,newman1983general,newman1984asymptotic}). In applications positive
dependence arises naturally when the variables in the system are driven
by common factors. Such situations occur, for example, in multivariate
financial data, where the common factor can represent the intrinsic market
component (\citet{ARU20}). Another appearance is in evolutionary processes,
where the observed variables evolve from a common ancestor
(\citet{steel2016phylogeny,zwiernik2018latent}).

Various mathematical definitions of positive dependence exist, including
positive association (\citet{esary1967association}) and multivariate total
positivity of order 2 ($\mathrm{MTP}_{2}$) (\citet{KR1980,FLSUWZ2017}). In particular,
the latter notion is attracting a surging interest. The reason is that,
for Gaussian models, it has the intuitive characterization that all correlations
and partial correlations are nonnegative and that its analytical constraints
on the distribution can be implemented elegantly in the estimation of statistical
models (\citet{SH2015,LUZ2019}). In addition, $\mathrm{MTP}_{2}$ models outperform
state-of-the-art methods in finance (\citet{WRU2020,RZ20}), psychometrics
\citep{LUZ2019,LUZ2020}, machine learning
(\citet{ying2021minimax,EPO17}), medical statistics and phylogenetics
\cite{FLSUWZ2017}. There is also a fundamental link between the assumption
of sparsity and the $ \mathrm{MTP}_{2} $ constraint (\citet{LUZ2019}).

When interest is in extreme events, then intuitively one may expect even
stronger positive dependence, as it can be conceived that multivariate
extreme events arise from a common latent factor. For instance, during
a financial crisis a shock may affect many stock prices simultaneously.
Similarly, flooding at different locations is often caused by the same
large-scale precipitation field.

Multivariate extreme value theory provides asymptotically motivated models
for extremal dependence. Traditionally, the focus was on the analysis of
max-stable distributions, which indeed can be shown to be always positively
associated (\citet{MO1983}). Max-stable models arise as the componentwise
maxima of independent copies of a random vector in its domain of attraction
(\citet{deh1977}). This means that the latter can have any dependence structure,
but the most extreme observations in each component eventually become positively
associated. While this illustrates how positive dependence naturally emerges
in multivariate extremes, max-stable distributions may be too rigid for
modeling higher dimensional data. One reason is that their densities cannot
factorize in a nontrivial way on graphs (\citet{pap2016}).

The interest has, therefore, shifted to multivariate Pareto distributions,
a different type of models in multivariate extremes, which are the only
possible limits for multivariate threshold exceedances
(\citet{roo2006}). For this distribution class, extremal graphical models
can be defined (\citet{EH2020}) that allow for sparse statistical models.
In this paper we propose a new notion of positive dependence for multivariate
Pareto distributions that we call extremal $ \mathrm{MTP}_{2} $ ($
\mathrm{EMTP}_{2} $).

As intuition from practice and the max-stable case suggest,
$ \mathrm{EMTP}_{2} $ arises naturally in existing extreme value models. Indeed,
we show in Section~\ref{sec:emtp2} that many classical models, such as
the extremal logistic (\citet{taw1990}) and extremal Dirichlet distributions
(\citet{CT1991}), are $ \mathrm{EMTP}_{2} $ across the whole range of their
parameter values and in any dimension. Within multivariate Pareto distributions,
the class of H\"usler--Reiss models (\citet{HR1989}), parameterized by a
variogram matrix $\Gamma $, can be seen as the counterpart of Gaussian
models in multivariate extremes. An alternative parameterization is given
in terms of the H\"usler--Reiss precision matrix $\Theta $
(\citet{Hentschel2021}). Inside this class we show that a model is
$ \mathrm{EMTP}_{2} $ if its precision matrix is a Laplacian matrix of a connected
graph with positive edge weights, that is, $\Theta _{ij}\leq 0$ for all
$i\neq j$. This implies that any bivariate H\"usler--Reiss distribution
is $ \mathrm{EMTP}_{2} $.

In Section~\ref{sec:graphicalEx} we formalize the connection between
$ \mathrm{EMTP}_{2} $ distributions and graphical models for extremes. The
case of H\"{u}sler--Reiss distributions closely parallels Gaussian graphical
models (\citet{LUZ2019}) but often allows for stronger results. For instance,
all H\"usler--Reiss tree models are $ \mathrm{EMTP}_{2} $, and this even continues
to hold for any latent tree structure. Finally, we study the axiomatization
of extremal conditional independence in the spirit of
\cite{FLSUWZ2017} and \cite{LS18} and show that $ \mathrm{EMTP}_{2} $ graphical
models satisfy an extremal notion of faithfulness.

The methodological part of our paper focuses on the H\"usler--Reiss distribution.
In Section~\ref{s:estimation} we propose an estimator of the H\"usler--Reiss
precision matrix $\Theta $ that takes the empirical version of the variogram
$\overline{\Gamma}$ as input and optimizes the convex problem
\begin{align}
	\label{eq:logdet} \log\operatorname{Det} \Theta +\frac{1}{2}\operatorname{tr}
	(\overline{\Gamma}\Theta ) 
\end{align}
over all positive semidefinite precision matrices and under the
$ \mathrm{EMTP}_{2} $ constraint that $\Theta $ is a Laplacian matrix of a
connected graph with positive edge weights. Here $ \operatorname{Det} $ denotes the
pseudo-determinant since $\Theta $ has one zero eigenvalue. We prove the
consistency of this estimator, and based on the dual formulation, in Section~\ref{s:algorithms}
we design a block coordinate-descent algorithm that efficiently solves
the constrained optimization problem. The $ \mathrm{EMTP}_{2} $ constraint
acts as an implicit regularizer and the estimator can also be applied in
high-dimensional settings. Moreover, since the solution satisfies KKT conditions
for optimality, the estimator $\widehat{\Theta}$ under
$ \mathrm{EMTP}_{2} $ typically contains zeros, which implies that the corresponding
H\"usler--Reiss model is an extremal graphical model. We formalize this
observation and show that the estimated $ \mathrm{EMTP}_{2} $ graph asymptotically
is a super-graph of the true underlying graph. This allows for interpretation,
in particular, when the estimated graph is sparse, as in our application
to river networks in Section~\ref{s:application}. We note that our estimator
is the first method for extremal graphical models that goes beyond trees
or block graphs (\citet{EV2020}).

An important part of our theoretical contribution is the study of strong
$ \mathrm{MTP}_{2} $, also known as LLC in the literature
(\citet{murota2009recent, RSTU20}). In order to characterize
$ \mathrm{EMTP}_{2} $, we establish new additive relations of positive dependence,
which are of independent interest. For a random variable $X_{0}$ that is
independent of a random vector $\mathbf{X}$, we link the probabilistic
dependence properties of
\begin{align}
	\label{Ydef} \mathbf{Z}=(X_{0}, \mathbf{X}+X_{0}
	\boldsymbol 1)
\end{align}
with those of $\mathbf{X}$, where $\boldsymbol 1$ denotes the vector of
ones. We will show that $\mathbf{Z}$ is $ \mathrm{MTP}_{2} $ if and only if
$\mathbf{X}$ is strongly $\mathrm{MTP}_{2}$. Models of the above form have
been proposed under the name of random location mixtures
(e.g., \citet{Hashorva2012}). As we show later in the paper, this formulation
also links to factor models
(\citet{lawley1962factor,krijnen2004positive}). In the context of multivariate
extreme value theory, these constructions produce versatile tail dependence
structures. In the case of multivariate Pareto distributions, the so-called
extremal functions are of the above form with a standard exponential variable
$X_{0}$ and a lighter tailed random vector $\mathbf{X}$. Many models that
bridge between asymptotic dependence and independence have a representation
as in \eqref{Ydef}; see \cite{eng2018a} for a review. In this paper we
focus on asymptotic dependence and $ \mathrm{EMTP}_{2} $, but in Section~\ref{sec:discussion}
we discuss further implications of our theoretical results to asymptotically
independent extreme value models.

To assess the statistical performance of our $ \mathrm{EMTP}_{2} $ estimator,
in Section~\ref{s:application} we apply it to a data set of river discharges
and compare it to methods from spatial statistics (\citet{ADE2015}) and graphical
modeling (\citet{EH2020}).

\section{Background}
\label{s:back}

Our paper is at the intersection of positive dependence modeling and multivariate
extreme value theory. In this section we introduce the basic types of positive
dependence constraints, and we review existing results on multivariate
Pareto distributions and their connections to extremal graphical models.

\subsection{Notions of positive dependence}
\label{s:posdep}

We begin by recalling two notions of positive dependence. The first one
is multivariate total positivity of order 2 ($\mathrm{MTP}_{2}$) treated in
detail in \cite{KR1980}. The second is a stronger notion, which we call
strong $ \mathrm{MTP}_{2} $ and whose study is motivated by the extremal
$ \mathrm{MTP}_{2} $ property.

Let $ \mathbf{x}\vee \mathbf{y}$ and $ \mathbf{x}\wedge \mathbf{y}$ denote
the componentwise maximum and minimum of
$ \mathbf{x},\mathbf{y}\in \mathbb{R}^{d} $, respectively. A~function
$f:\mathbb{R}^{d}\to \mathbb{R}$ is
\textit{multivariate totally positive of order 2} ($ \mathrm{MTP}_{2} $) if
\begin{align}
	f(\mathbf{x}\vee \mathbf{y}) f(\mathbf{x}\wedge \mathbf{y})\ge f( \mathbf{x}) f(
	\mathbf{y})\quad \text{for all }\mathbf{x},\mathbf{y} \in \mathbb{R}^{d}.
	\label{eq:log-supermodular} 
\end{align}
We say that $f$ is strongly $ \mathrm{MTP}_{2} $ if
\begin{align}
	f\bigl(\mathbf{x}\vee (\mathbf{y}-\alpha \boldsymbol 1) \bigr) f\bigl((
	\mathbf{x}+ \alpha \boldsymbol 1) \wedge \mathbf{y}\bigr)\ge f(\mathbf{x}) f(
	\mathbf{y})\quad \text{for all }\mathbf{x},\mathbf{y}\in \mathbb{R}^{d},
	\alpha \geq 0, \label{eq:LLC} 
\end{align}
where $\boldsymbol 1$ denotes the vector of ones. A~multivariate random
vector $ \mathbf{X}$ with density $ f_{\mathbf{X}}$ is
$ \mathrm{MTP}_{2} $ or strongly $ \mathrm{MTP}_{2} $ if the corresponding property
holds for $f_{\mathbf{X}}$.

The concept of strong $ \mathrm{MTP}_{2} $ distributions is relatively new
and not well studied. In the statistical context, \citet{RSTU20} used
\eqref{eq:LLC} under the name of log-$L^{\#}$-concave (LLC) in reference
to work on discrete optimization (e.g., \citet{murota2009recent}), where
$f$ is, in addition, assumed to be log-concave. For a list of further references
for the appearance of strong $ \mathrm{MTP}_{2} $ in applications, see
\citet[pp.~3--4]{RSTU20}. The following important example discusses both
notions of positive dependence for Gaussian distributions.
\begin{ex}
	\label{ex:gauss}
	If $\mathbf{X}$ is Gaussian with mean vector $\mu $ and invertible covariance
	matrix $\Sigma $, then $\mathbf{X}$ is $ \mathrm{MTP}_{2} $ if and only if
	the inverse covariance matrix $K$ is an \textit{M-matrix}, that is, a positive
	definite matrix such that $K_{ij}\leq 0$ for all $ i\neq j $
	(e.g., \citet{LUZ2019}). Moreover, $\mathbf{X}$ is strongly
	$ \mathrm{MTP}_{2} $ if, in addition, $K$ is a diagonally dominant matrix,
	that is, all row sums are nonnegative ($K\boldsymbol 1\geq
	\boldsymbol 0$) (\citet{RSTU20}).
\end{ex}

Importantly, the $ \mathrm{MTP}_{2} $ property is closed under taking margins,
under conditioning and under coordinatewise increasing transformations;
see Corollary~3.3 and Proposition~3.4 in \cite{FLSUWZ2017}. Moreover, univariate
distributions are always $ \mathrm{MTP}_{2} $.

The situation is more complicated for strongly $ \mathrm{MTP}_{2} $ distributions.
In this paper we develop several new results for strong
$ \mathrm{MTP}_{2} $ distributions that are central to the analysis of extremal
$ \mathrm{MTP}_{2} $. First, we derive an alternative characterization of
strong $ \mathrm{MTP}_{2} $ in Lemma~\ref{lem:LLC2}, which we employ in the
proof of Theorem~\ref{t:LLCetal}. This additionally implies that strong
$ \mathrm{MTP}_{2} $ is also closed under taking margins (Proposition~\ref{prop:LLCmargin}).
Furthermore, we show that univariate distributions are strongly
$ \mathrm{MTP}_{2} $ if and only if the underlying density function is log-concave.
This also shows that strong $ \mathrm{MTP}_{2} $ cannot be closed under arbitrary
increasing transformations. Since these results are fairly technical, we
postpone proofs and auxiliary results to Appendix~\ref{s:smtp2}.

\subsection{Extremal dependence and multivariate Pareto distributions}
\label{s:MPD_prop}

Multivariate extreme value theory studies the tail properties of a random
vector $\mathbf{X}=(X_{1},\ldots ,X_{d})$. The dependence between the largest
values of each component, also called extremal dependence, quantifies to
what degree rare events happen simultaneously in several variables. The
full extremal dependence structure can be complex, and summary statistics
are employed to condense this information into easily interpretable numbers. A~popular bivariate summary statistic is the extremal correlation, defined
for $i,j \in [d]:=\{1,\dots , d\}$ as
\begin{align}
	\label{EC} \chi _{ij} := \lim_{q\to 0} \mathbb P
	\bigl\{F_{i}(X_{i}) > 1- q | F_{j}(X_{j})
	>1-q \bigr\} \in [0,1], 
\end{align}
whenever the limit exists and where $F_{j}$ is the distribution function
of $X_{j}$ (e.g., \citet{col1999}). We speak of asymptotic dependence
and independence if $\chi _{ij}>0$ and $\chi _{ij}=0$, respectively. The
theoretical analysis for asymptotic independence is more nuanced, and a
whole line of research exists
(e.g., \citet{HeffernanTawn2004, wadsworth2012dependence}); we discuss
this further in Section~\ref{sec:discussion}. Extremal correlations can
also be extended to higher dimension (\citet{sch2003}).

Since the univariate theory is well studied
(e.g., \citet{emb1997, deh2006a}), it is common to normalize the margins
to concentrate on the extremal dependence. Throughout this paper we assume
that each component of $\mathbf X$ has been normalized to have standard
exponential distribution with distribution function $1 - \exp (-x)$,
$x\geq 0$; we discuss this normalization in the preprocessing steps of
the application in Section~\ref{sec:doa}.

In this paper we focus on the case of asymptotic dependence. To describe
the extremal dependence structure in this setting, the assumption of multivariate
regular variation is widely used (\citet{res2008}). More precisely, we assume
that the distribution of the exceedances of $\mathbf X$ over a high threshold
converges to a so-called multivariate (generalized) Pareto distribution
$\mathbf Y$ (\citet{roo2006}), that is,
\begin{align}
	\label{MPD_conv} \mathbb{P}(\mathbf{Y}\le \mathbf{z})&=\lim_{u\to \infty}
	\mathbb{P} \bigl(\mathbf{X} - u \boldsymbol{1}\le \mathbf{z} | \llVert \mathbf{X}
	\rrVert _{\infty}>u \bigr), \quad \mathbf{z}\in \mathcal L. 
\end{align}
The distribution of $\mathbf{Y}$ is supported on the space
$\mathcal{L}=\{\mathbf{x}\in \mathbb R^{d}: \| \mathbf{x}\|_{\infty}> 0
\}$, and it satisfies the homogeneity
$\mathbb P( \mathbf{Y}\in t + A) = t^{-1} \mathbb P( \mathbf{Y}\in A)$
for any $t>0$ and Borel set $A\subset \mathcal L$. We say that the vector
$\mathbf{X}$ is in the domain of attraction of the multivariate Pareto
distribution $\mathbf{Y}$. Note that assumption \eqref{MPD_conv} is equivalent
to multivariate regular variation of the random vector
$\exp (\mathbf{X})$ (\citet{res2008}, Proposition~5.15); see Remark~\ref{rem_margins}
and Appendix~\ref{exp_measure} for details.

Replacing the vector $\mathbf{X}$ in \eqref{MPD_conv} by the $I$th margin
$\mathbf{X}_{I} = (X_{i} : i\in I)$, $I\subset [d]$, we denote the corresponding
limit by $\mathbf{Y}_{I}$, which is a multivariate Pareto distribution
on the space
$\mathcal{L}_{I}=\{\mathbf{x}\in \mathbb R^{|I|}: \| \mathbf{x}\|_{
	\infty}> 0 \}$. This is a slight abuse of notation since the so-defined
random vector is not equal to the $I$th margin $(Y_{i}: i \in I)$ of
$\mathbf{Y}$ defined on $\mathbb R^{|I|}$. The difference between the two
is only the support and restricting $(Y_{i}: i \in I)$ to
$\mathcal L_{I}$ results in $\mathbf{Y}_{I}$.

Multivariate Pareto distributions $\mathbf{Y}$ are defined on the nonproduct
space $\mathcal L$. In order to define stochastic properties for
$\mathbf{Y}$, it is convenient to work with the conditional random vectors
\begin{equation*}
	\mathbf{Y}^{k} :=\mathbf{Y}| \{Y_{k}>0\},
\end{equation*}
where $k \in [d]$. If $\mathbf{Y}$ admits a Lebesgue density
$f_{\mathbf{Y}}$, then $ \mathbf{Y}^{k} $ has a density proportional to
$f_{\mathbf{Y}}$ supported on the product space
$ \mathcal{L}^{k}=\{\mathbf{x}\in \mathcal{L}:x_{k}>0\} $. Thanks to the
homogeneity of $\mathbf{Y}$, we have the stochastic representation
\begin{align}
	\label{Yk} \mathbf{Y}^{k} \stackrel{d} {=} \mathbf{W}^{k}
	+ E \boldsymbol 1,
\end{align}
where $ E $ has a standard exponential distribution that is independent
of a $d$-variate random vector $\mathbf{W}^{k}$. The latter is called the
$k$th extremal function, and it satisfies $W^{k}_{k}=0$ almost surely.
We refer to \cite{dom2013a} and \cite{EH2020} for more background on extremal
functions.

\begin{remark}
	\label{rem_margins}
	Multivariate Pareto distributions are often introduced by normalizing the
	random vector $\mathbf{X}$ to standard Pareto margins (\citet{EH2020}). As
	in copulas, this changes only the marginal distributions of
	$\mathbf{Y}$ but not the extremal dependence structure.
	\cite{roo2006} denote multivariate Pareto distributions with general margins
	as multivariate generalized Pareto distributions, in analogy to the generalized
	Pareto distributions in the univariate case, which also include exponential
	distributions. In our paper we opt for the exponential scale since it makes
	the results more concise. Nevertheless, we say that $\mathbf{Y}$ follows
	a multivariate Pareto distribution and drop the ``generalized'' for simplicity.
\end{remark}

Many alternative coefficients for extremal dependence have been studied.
This includes the madogram (\citet{coo2006}) and a coefficient by
\cite{lar2012} used for dimension reduction in \cite{coo2018},
\cite{for2020} and \cite{dre2021}; see the review \cite{eng2021} for details.
Another summary statistic introduced in \cite{EV2020} is the extremal variogram
rooted in $ k \in [d]$, which for a multivariate Pareto distribution
$\mathbf{Y}$ is
\begin{align}
	\label{evario} \Gamma ^{(k)}_{ij}=\operatorname{Var} \bigl(
	Y_{i}^{k}- Y_{j}^{k} \bigr), \quad i,j
	\in [d], 
\end{align}
given that the variance exists.

While summary statistics provide a first idea of the strength of dependence,
they are mainly used for exploratory analysis and model assessment. Approaches
that study probabilistic properties of the whole distribution are more
powerful to improve statistical inference. Examples are the notions of
conditional independence or positive dependence. In Section~\ref{sec:emtp2}
we, therefore, discuss how positive dependence and, in particular, can
be exploited for multivariate Pareto distributions.

\subsection{H\"usler--Reiss distributions}
\label{sec:huesler_reiss}

An important example of a multivariate Pareto distribution is the H\"{u}sler--Reiss
distribution, which can be seen as the analogue of the Gaussian distribution
inside the class of multivariate Pareto distributions.

For a fixed $d\in \mathbb{N}$, let $\mathbb{S}^{d}_{0}$ be the set of symmetric
$d\times d$-matrices with zero diagonal. We say that
$\Gamma \in \mathbb{S}^{d}_{0}$ is a conditionally negative definite matrix
if $\boldsymbol x^{T}\Gamma \boldsymbol x\leq 0$ for all
$\boldsymbol x\in \mathbb{R}^{d}$ such that
$\boldsymbol x^{T} \boldsymbol 1=0$. Moreover, $\Gamma $ is strictly conditionally
negative definite if the inequality is always strict, unless
$\boldsymbol x=\boldsymbol 0$. We denote the cone of such matrices by
$\mathcal C^{d}\subset \mathbb{S}^{d}_{0}$. In Appendix~\ref{app:Gamma}
we collect various results on such matrices, which will be useful in the
next sections. Note that, from here on, we will abbreviate singleton set
$ \{k\} $ by $k$ and index sets $ [d]\setminus k $ by
$ \setminus k $.

The $ d $-variate H\"usler--Reiss distribution is a multivariate Pareto
distribution parametrized by $\Gamma \in \mathcal C^{d}$
(\citet{HR1989}). In this case the random vector
$\mathbf{W}^{k}_{\setminus k}$ in \eqref{Yk} has a $(d-1)$-dimensional
normal distribution with mean vector
$- \operatorname{diag}(\Sigma ^{(k)})/2$ and covariance $\Sigma ^{(k)}$ obtained
from $\Gamma $ via the covariance mapping
\begin{align}
	\Sigma _{ij}^{(k)}=\frac{1}{2} (\Gamma _{ik}+
	\Gamma _{jk}- \Gamma _{ij} ), \quad i, j\neq k; \label{eq:gamma2sigma}
\end{align}
see \cite{eng2014} for details and \cite{deza1997geometry} for the importance
of this mapping in the more general context of distance geometry. Note
that \eqref{eq:gamma2sigma} is a linear isomorphism from
$\mathbb{S}^{d}_{0}$ to the space $\mathbb{S}^{d-1}$ of all symmetric
$(d-1)\times (d-1)$ matrices and its inverse is given by
\begin{align}
	\begin{cases} \Gamma _{ik}=\Sigma ^{(k)}_{ii},
		&i\neq k,
		\\[3pt]
		\Gamma _{ij}=\Sigma ^{(k)}_{ii}+\Sigma
		^{(k)}_{jj}-2\Sigma ^{(k)}_{ij}, &i, j\neq k.
	\end{cases} \label{eq:sigma2gamma} 
\end{align}
The image of the cone $\mathcal C^{d}$ under the linear mapping
\eqref{eq:gamma2sigma} is the set of all positive definite matrices in
$\mathbb{S}^{d-1}$; for a proof, see Lemma~3 in \cite{EH2020} or Lemma~\ref{lem:pdC}
in Appendix~\ref{app:Gamma}. Therefore, $\Sigma ^{(k)}$ is positive definite.

Using the standard terminology of exponential families, in the multivariate
Gaussian distribution the covariance matrix is the mean parameter, and
its inverse is the canonical parameter. Working with the inverse is useful,
as the log-likelihood function is a strictly concave function. Analogously,
a useful parameterization for the H\"{u}sler--Reiss distribution is discussed
in \cite{Hentschel2021}. Let $ \Theta ^{(k)} $ denote the inverse of
$ \Sigma ^{(k)} $. Define the matrix $ \Theta \in \mathbb{S}^{d}$ as
\begin{align}
	\label{eq:Theta} \Theta _{ij}&:= \Theta _{ij}^{(k)}
	\quad \text{for some }k\neq i,j.
\end{align}
Note that $ \Theta _{ij}^{(k)}=\Theta _{ij}^{(k')} $ for
$ i,j\neq k,k' $ by \citet[Lemma~1]{EH2020}. We call $\Theta $ the H\"{u}sler--Reiss
precision matrix. An important alternative characterization of
$\Theta $ is obtained as follows. Define the projection matrix
\begin{equation}
	\label{eq:P1} \boldsymbol P:=I_{d}-\frac{1}{d}\boldsymbol 1
	\boldsymbol 1^{T}, 
\end{equation}
and let $\Sigma := \boldsymbol P(-\frac{\Gamma}{2})\boldsymbol P$. By Lemma~\ref{rem:edm}
if $\Gamma \in \mathcal{C}^{d}$, then $\Sigma $ is positive semidefinite.
Moreover, $\operatorname{rank}(\Sigma )=d-1$ and
$\Sigma \boldsymbol 1=\boldsymbol 0$. Denote by $A^{+}$ the Moore--Penrose
pseudoinverse of $A$.
\begin{prop}[\cite{Hentschel2021}]
	\label{prop:manu}
	Consider $\Theta $ defined in \eqref{eq:Theta} and $\Sigma $ as above.
	We have $\Theta =\Sigma ^{+}$. It follows that
	$\operatorname{rank}(\Theta )=d-1$ and $\Theta \boldsymbol 1=\boldsymbol 0$.
\end{prop}
The matrix $\Theta $ plays a particularly important role in connection
with positive dependence. This will be discussed in the Section~\ref{sec:HR}.

\subsection{Graphical models for multivariate Pareto distributions}
\label{sec:ext_CI}

Let $ \mathbf{Y}$ be a multivariate Pareto random vector with support on
the space $ \mathcal{L} $. As mentioned above, the vector
$ \mathbf{Y}^{k} $ is supported on a product space
$ \mathcal{L}^{k} $. The construction of $ \mathbf{Y}^{k} $ allows to define
extremal conditional independence for multivariate Pareto distributions
as follows.
\begin{defi}[{\citet[Definition~5]{EH2020}}]
	\label{def:extremal_indep}
	Let $ A$, $B$, $C $ be disjoint subsets of $ [d] $. $ \mathbf{Y}_{A} $ is extremal
	conditionally independent of $ \mathbf{Y}_{B} $, given
	$ \mathbf{Y}_{C} $ (abbreviated as
	$ \mathbf{Y}_{A} \perp _{e} \mathbf{Y}_{B}| \mathbf{Y}_{C} $), if for
	all $ k\in [d] $, it holds that
	\begin{align}
		\label{CI_k} \mathbf{Y}_{A}^{k}\perpp
		\mathbf{Y}_{B}^{k}| \mathbf{Y}_{C}^{k}.
	\end{align}
\end{defi}
It was shown that the condition in Definition~\ref{def:extremal_indep}
can be weakened \citep[ Proposition~1]{EH2020}, and in fact, extremal conditional
independence
$ \mathbf{Y}_{A} \perp _{e} \mathbf{Y}_{B}| \mathbf{Y}_{C} $ already
holds if there exists a $k \in C$ in the conditioning set such that \eqref{CI_k}
is satisfied.

Probabilistic graphical models encode conditional independence in graph
structures. Let $ G=(V,E) $ be an undirected graph with vertex set
$ V=[d] $ and edge set $ E $. A~random vector $ \mathbf{X}$ satisfies the
pairwise Markov property with respect to $ G $, when
\begin{equation*}
	X_{i}\perpp X_{j}| \mathbf{X}_{\setminus ij},\quad
	\text{if }(i,j)\notin E.
\end{equation*}
In this case we call $ \mathbf{X}$ a
\textit{probabilistic graphical model}.
\begin{ex}
	\label{ex_gaussian_GM}
	For a multivariate Gaussian random vector $\mathbf{X}$ with invertible
	covariance $ \Sigma $ and concentration matrix $ K=\Sigma ^{-1}$, it holds
	that $ X_{i}\perpp  X_{j}| \mathbf{X}_{\setminus ij}$ if
	and only if $K_{ij}=0 $. This means that for Gaussian graphical models,
	the concentration matrix contains the graph structure.
\end{ex}

Definition~\ref{def:extremal_indep} allows us to define graphical models
that encode extremal conditional independence. Let $ G=(V,E) $ be an undirected
graph with vertex set $ V $ and edge set $ E $. A~multivariate Pareto vector
$ \mathbf{Y}$ satisfies the pairwise Markov property on
$ \mathcal{L} $ with respect to $ G $ when
\begin{equation*}
	Y_{i} \perp _{e} Y_{j}| \mathbf{Y}_{\setminus ij}
	\quad \text{if }(i,j)\notin E.
\end{equation*}
This means that $ Y_{i} $ and $ Y_{j} $ are extremal conditionally independent,
given all other variables, if there is no edge between $ i $ and
$ j $ in $ G $. In fact, this resembles probabilistic graphical models,
only with extremal conditional independence instead of standard conditional
independence. In this case $ \mathbf{Y}$ is called an extremal graphical
model on $ G $. For a decomposable graph $G$ and if $\mathbf{Y}$ has a
positive and continuous density $f_{\mathbf{Y}}$,
\citet[Theorem~1]{EH2020} proves a Hammersley--Clifford theorem, showing
the equivalence of the pairwise and global Markov properties, as well as
a factorization of the density $ f_{\mathbf{Y}}$ with respect to
$ G $. Note that if $\mathbf{Y}$ has a density, then extremal graphical
models are only defined for connected graphs \citep[Remark~1]{EH2020},
since marginal independence
$\mathbf{Y}_{A} \perp _{e} \mathbf{Y}_{B}$, $A,B\subset V$, would contradict
the existence of the density. This can be relaxed by dropping the assumption
on existence of densities; see Kirstin Strokorb's discussion of
\cite{EH2020}.

\begin{ex}
	\label{ex:HRgraph}
	It was shown in \citet[Proposition~3]{EH2020} that extremal conditional
	independence for H\"usler--Reiss distributions can be read off from the
	inverse covariance matrix
	$ \Theta ^{(k)}:=  (\Sigma ^{(k)}  )^{-1} $. By construction the
	extremal conditional independence does not depend on $ k $, and this is
	also reflected by the relation between the $\Theta ^{(k)}$ for different
	$k\in V$ \citep[Lemma~1]{EH2020}. \cite{Hentschel2021} uses this to rephrase
	extremal conditional independence in terms of the H\"usler--Reiss precision
	matrix $ \Theta $ in \eqref{eq:Theta} such that
	\begin{align}
		Y_{i} \perp _{e} Y_{j}| \mathbf{Y}_{\setminus ij}
		\quad & \Longleftrightarrow \quad \Theta _{ij}=0. \label{eq:indep_from_theta}
	\end{align}
	This equivalence shows the strong similarity of H\"usler--Reiss distributions
	with multivariate Gaussians, where conditional independences can be read
	off from the precision matrix.
	
	We have seen that the H\"usler--Reiss distribution has many similar properties
	with respect to extremal conditional independence as the Gaussian distribution
	with respect to standard conditional independence. It can, therefore, be
	considered as an analogue of a Gaussian graphical model among extremal
	graphical models.
\end{ex}

\section{Extremal \texorpdfstring{\MTPtwo}{MTP2}~distributions}
\label{sec:emtp2}

In this section we introduce a notion of total positivity in extremes.
For some of the most popular parametric families, such as H\"usler--Reiss
and logistic distributions, we show how this property is characterized.

\subsection{Extremal positivity for multivariate Pareto distributions}
\label{sec3.1}

Total positivity in \eqref{eq:log-supermodular} is defined from an inequality
where a probability density is evaluated at two points and their corresponding
componentwise minimum and maximum. The space $ \mathcal{L} $ in the definition
of multivariate Pareto distributions is not closed under these componentwise
operations. Definition \eqref{eq:log-supermodular} is thus not directly
applicable to multivariate Pareto distributions. Similar to extremal conditional
independence (see Definition~\ref{def:extremal_indep}), we define the extremal
version of $ \mathrm{MTP}_{2} $.
\begin{defi}
	\label{def:emtp}
	Let $ \mathbf{Y}$ be a multivariate Pareto random vector. We say that
	$ \mathbf{Y}$ is extremal multivariate totally positive of order 2 ($
	\mathrm{EMTP}_{2} $) if and only if $ \mathbf{Y}^{k} $ is
	$ \mathrm{MTP}_{2} $ for all $ k\in [d] $.
\end{defi}

Using the stochastic representation \eqref{Yk}, we will rewrite this definition
as an explicit condition on the extremal function $\mathbf{W}^{k}$. This
uses the notion of strong $ \mathrm{MTP}_{2} $ distributions given in
\eqref{eq:LLC} and the following result. Recall that the support of a density
function $f$ is the smallest closed set over which the density integrates
to 1.
\begin{thm}
	\label{t:LLCetal}
	Let $X_{0}$ be a random variable whose density is supported on
	$[c,\infty )$ for some $c\in \mathbb R$, and let
	$\mathbf{X}=(X_{1},\ldots ,X_{d})$ be a random vector such that
	$X_{0}\perpp  \mathbf{X}$. Let
	$\mathbf{Z}=(X_{0}, \mathbf{X}+X_{0}\boldsymbol 1)$. Then:
	\begin{itemize}
		\item[1.] $\mathbf{Z}$ is $ \mathrm{MTP}_{2} $~$\Longleftrightarrow $
		$\mathbf{X}$ is strongly $ \mathrm{MTP}_{2} $.
		\item[2.] $\mathbf{Z}$ is strongly $ \mathrm{MTP}_{2} $~$\Longleftrightarrow $ $X_{0}$ and $\mathbf{X}$ are strongly
		$ \mathrm{MTP}_{2} $.
	\end{itemize}
\end{thm}

The above theorem provides a fundamental result on positive dependence
properties of convolutions of random vectors. We discuss in Section~\ref{sec:discussion}
how it can be used in the analysis of general multivariate extreme value
models.

In the sequel we concentrate on particular application of this theorem
to the representation \eqref{Yk}. This gives us a simple way of verifying
whether a multivariate Pareto distribution $\mathbf{Y}$ is
$ \mathrm{EMTP}_{2} $ and shows how the strong $ \mathrm{MTP}_{2} $ property
becomes important in our setting.
\begin{thm}
	\label{t:mtp2}
	Let $ \mathbf{Y}$ be a multivariate Pareto distribution and
	$ \mathbf{W}^{k} $ the $ k $th extremal function, as defined in
	\eqref{Yk}. Then $ \mathbf{Y}$ is $ \mathrm{EMTP}_{2} $ if and only if the
	distribution of $\mathbf{W}^{k}_{\setminus k}$ is strongly
	$ \mathrm{MTP}_{2} $ for all $ k\in [d] $. Equivalently,
	$\mathbf{W}^{k}_{\setminus k}$ is strongly $ \mathrm{MTP}_{2} $ for some
	$ k\in [d] $.
\end{thm}

The next result gives a useful property of $ \mathrm{EMTP}_{2} $ distributions,
in particular, in connection with latent trees models in Section~\ref{s:tree_models}.
Recall the definition of the margins of a multivariate Pareto distribution
in Section~\ref{s:MPD_prop}.

\begin{prop}
	\label{prop:marginex}
	If a multivariate Pareto distribution $\mathbf{Y}$ is
	$ \mathrm{EMTP}_{2} $, then for any $I \subset [d]$ the margin
	$\mathbf{Y}_{I}$ is also $ \mathrm{EMTP}_{2} $.
\end{prop}

\subsection{H\"usler--Reiss distributions}
\label{sec:HR}

For the H\"usler--Reiss distribution, the extremal function is distributed
according to a degenerate Gaussian distribution. Let $\Theta $ be the H\"usler--Reiss
precision matrix defined in \eqref{eq:Theta}. Denote
$ \mathbb{U}^{d}_{+} \subset \mathbb{S}^{d} $ to be the set of all graph
Laplacians for connected graphs with positive weights on each edge. In
other words, $ \mathbb{U}^{d}_{+} $ is the set of $d\times d$ symmetric
matrices with zero row sums and nonpositive off-diagonal entries whose
support correspond to a connected graph.
\begin{thm}
	\label{t:equivalences_precision_matrices_emtp2}
	Suppose $\mathbf{Y}$ has H\"usler--Reiss distribution with variogram matrix
	$\Gamma $. Then $\mathbf{Y}$ is $ \mathrm{EMTP}_{2} $ if and only if
	$ \Theta $ is the Laplacian matrix of a \emph{connected} graph with positive
	edge weights. Other equivalent conditions are:
	\begin{itemize}
		\item [(i)] $\Theta _{ij}\leq 0$   for all $i\neq j$.
		\item [(ii)] For all $k=1,\ldots ,d$,  $\Theta ^{(k)}$ is a diagonally
		dominant M-matrix.
		\item [(iii)] For all $k=1,\ldots ,d$,  $\Theta ^{(k)}$ is an M-matrix.
		\item [(iv)] For some $k\in [d]$,  $\Theta ^{(k)}$ is a diagonally dominant
		M-matrix.
	\end{itemize}
\end{thm}

\begin{remark}
	\label{rem3.5}
	Note that this theorem and standard results on graph Laplacians imply that
	every support is possible in $\Theta $ as long as it corresponds to a connected
	graph.
\end{remark}

We now discuss the examples of bivariate and trivariate H\"usler--Reiss
distributions with respect to Theorem~\ref{t:equivalences_precision_matrices_emtp2}.

\begin{ex}
	\label{ex:bivariate_HR_mtp2}
	The bivariate H\"usler--Reiss distribution is generated from a Gaussian
	random variable with mean $ -\Gamma _{12}/2 $ and variance
	$ \Gamma _{12} $. Therefore, as
	$ \Theta ^{(1)}=\Theta ^{(2)}=1/\Gamma _{12} $ is always positive by definition
	of $ \Gamma _{12} $, it is always a diagonally dominant M-matrix, and it
	follows that the bivariate H\"usler--Reiss distribution is
	$ \mathrm{EMTP}_{2} $ for any $ \Gamma _{12} $.
\end{ex}

\begin{ex}
	\label{ex:k=3}
	Let $ d=3 $. Then
	\begin{align*}
		\Theta &=\frac{1}{\det (\Sigma ^{(k)})} 
		\begin{pmatrix} 2\Gamma _{23}&
			\Gamma _{12}-\Gamma _{13}-\Gamma _{23}&-\Gamma
			_{12}+ \Gamma _{13}-\Gamma _{23}
			\\
			\Gamma _{12}-\Gamma _{13}-\Gamma _{23}&2\Gamma
			_{13}&-\Gamma _{12}- \Gamma _{13}+\Gamma
			_{23}
			\\
			-\Gamma _{12}+\Gamma _{13}-\Gamma _{23}&-\Gamma
			_{12}-\Gamma _{13}+ \Gamma _{23}&2\Gamma
			_{12} \end{pmatrix}.
	\end{align*}
	
	The conditions in Theorem~\ref{t:equivalences_precision_matrices_emtp2}(ii)
	translate to the triangle inequalities
	\begin{align}
		\label{eq:triangle}
		\nonumber
		\Gamma _{12}&\leq \Gamma _{13}+
		\Gamma _{23},
		\\
		\Gamma _{13}&\leq \Gamma _{12}+\Gamma _{23},
		\\
		\nonumber
		\Gamma _{23}&\leq \Gamma _{12}+\Gamma
		_{13},
	\end{align}
	where the second and the third inequality come from the row sums. Note
	the symmetry in the inequalities, as $ \mathrm{EMTP}_{2} $ does not depend
	on $ k $. It follows that for trivariate H\"usler--Reiss distributions,
	$ \mathrm{EMTP}_{2} $ is equivalent to $\Gamma $ being a metric.
\end{ex}

As we remarked in Appendix~\ref{sec:povar}, as long as $\Gamma $ is a strictly
conditionally negative matrix, $\sqrt{\Gamma _{ij}}$ are always distances
in the sense that the map $(i, j) \mapsto \sqrt{\Gamma _{ij}}$ is a metric
function (satisfies the triangle inequality). In the special case when
$\Theta $ is a Laplacian matrix, as in Theorem~\ref{t:equivalences_precision_matrices_emtp2},
the map $(i, j) \mapsto \Gamma _{ij}$ is also a metric function by Lemma~\ref{lem:eff_res}.
In the electrical network literature, this corresponds to the statement
that if $\Theta $ is a Laplacian of a connected graph then the corresponding
resistances $\Gamma $ define a metric
(\citet{fiedler98,devriendt2020effective,klein1993resistance}). In Example~\ref{ex:k=3}
we showed that these two conditions are equivalent if $d=3$. If
$d>3$, then $\Gamma $ being a metric is a strictly weaker condition. Here
we present a probabilistic interpretation for the case when
$\Gamma $ is a metric. Recall the classical notion of positive association
(\citet{esary1967association}): A random vector $\mathbf{X}$ is positively
associated if $\operatorname{Cov}(f(\mathbf{X}),g(\mathbf{X}))\geq 0$ for any two
nondecreasing functions $f$, $g$ for which this covariance exists. By
\cite{Pitt1982}, a Gaussian $\mathbf{X}$ is positively associated if and
only if its covariance matrix has only nonnegative entries.
\begin{prop}
	\label{prop:extremal_association}
	The parameter matrix $ \Gamma $ in a H\"usler--Reiss random vector satisfies
	the triangle inequality
	$ \Gamma _{ij}\le \Gamma _{ik}+\Gamma _{jk} $ for all
	$ i,j,k\in [d] $ if and only if all extremal functions
	$\mathbf{W}^{k}$, $k\in [d]$, are positively associated.
\end{prop}

\subsection{Other important constructions}
\label{sec:otherconstr}

Another popular construction of multivariate Pareto distributions arises
from extremal functions of the form
\begin{align}
	\mathbf{W}^{k}=(U_{1}-U_{k},\ldots
	,U_{d}-U_{k}), \label{eq:indep_gen} 
\end{align}
for independent $U_{1},\ldots ,U_{d}$. Examples include the extremal logistic
(\citet{taw1990,dom2016}) and the extremal Dirichlet distribution
(\citet{CT1991}), which we will discuss below.

Our next result provides a simple way of checking whether such constructions
are $ \mathrm{EMTP}_{2} $.
\begin{prop}
	\label{prop:indep_gen}
	Consider the multivariate Pareto distribution with stochastic representation
	\eqref{Yk}. Suppose that $ W^{k}_{i}=U_{i}-U_{k}$ for some independent
	$U_{1},\ldots ,U_{d}$ such that $U_{i}$ has a log-concave distribution
	for every $i \in [d]$. Then $\mathbf{Y}$ is $ \mathrm{EMTP}_{2} $.
\end{prop}

From Proposition~\ref{prop:indep_gen} it follows that both the extremal
logistic and extremal Dirichlet distributions are always
$ \mathrm{EMTP}_{2} $.

\begin{ex}[Extremal logistic distribution]
	\label{exmp6}
	The extremal logistic distribution with parameter
	$ \theta \in (0,1) $ is defined by an extremal function
	$ \mathbf{W}^{k} $, as in \eqref{eq:indep_gen}, with
	$ U_{i}\sim \operatorname{Gumbel}(\text{location}=\theta G(1-\theta ),
	\text{scale} = \theta )$ for $ i\neq k $ and
	$ (G(1-\theta )\exp (U_{k}))^{-1/\theta} := Z \sim \operatorname{Gamma}(
	\text{shape} = 1-\theta ,\text{scale} = 1) $, where $ G(x) $ is the Gamma
	function (\citet{dom2016}). For $ i\neq k $, $ U_{i} $ follows a Gumbel distribution,
	which is log-concave. For $ i=k $, observe
	\begin{align*}
		-U_{k}&=\theta \log (Z)+\log \bigl(G(1-\theta )\bigr),
	\end{align*}
	which means that
	$ - U_{k} \sim \operatorname{ExpGamma}[\text{shape}= 1-\theta ,\text{scale}=
	\theta ,\text{location}=\log (G(1-\theta ))] $ follows an exponential Gamma
	distribution. An
	$ \operatorname{ExpGamma}[\text{shape}= \kappa ,\text{scale}=\theta ,
	\text{location}=\mu ]$ distribution has density
	\begin{equation*}
		\frac{1}{\theta G (\kappa )}\exp \biggl( \frac{\kappa (x-\mu )}{\theta }-\exp \biggl(\frac{x-\mu }{\theta }
		\biggr) \biggr),
	\end{equation*}
	which is log-concave. Hence, $ U_{k} $ is log-concave by symmetry. By Proposition~\ref{prop:indep_gen}
	this implies $ \mathrm{EMTP}_{2} $.
\end{ex}

\begin{ex}[Extremal Dirichlet distribution]
	\label{exmp7}
	The extremal Dirichlet distribution with parameters
	$ \alpha _{1},\ldots , \alpha _{d} $ has an extremal function, as in
	\eqref{eq:indep_gen}, where
	$ \exp (U_{i})\sim  \operatorname{Gamma}(\text{shape}=\alpha _{i},\text{scale}=1/
	\alpha _{i}) $ for $ i\neq k $ and
	$ \exp (U_{k})\sim \operatorname{Gamma}(\text{shape}=\alpha _{k}+1,\text{scale}=1/
	\alpha _{k}) $ (\citet{EV2020}). As the exponential Gamma distribution is
	log-concave, this is $ \mathrm{EMTP}_{2} $ by Proposition~\ref{prop:indep_gen}.
\end{ex}

\subsection{Bivariate Pareto distributions and \texorpdfstring{$\mathrm{EMTP}_{2}$}{EMTP2}}
\label{sec:2variate}

A bivariate Pareto distribution $\mathbf{Y}= (Y_{1}, Y_{2})$ is completely
characterized by a univariate distribution. Indeed, the extremal function
then satisfies $\mathbf{W}^{1} = (0, W^{1}_{2})$ with a real-valued random
variable $W_{2}^{1}$ with $\mathbb{E}(\exp W_{2}^{1})=1$. The second extremal
function $\mathbf{W}^{2} = (W^{2}_{1},0)$ is determined by the first one
through the duality
$\mathbb P(W^{2}_{1} \leq z) = \mathbb E(\mathbbm{1}_{\{W^{1}_{2}\geq -z
\}}\exp W^{1}_{2})$, $z\in \mathbb R$ \citep[Example~3]{EH2020}. Conversely,
any random variable $W_{2}^{1}$ with $\mathbb{E}(\exp W_{2}^{1})=1$ defines
a unique bivariate Pareto distribution through the extremal function and
duality.

These results extend to extremal tree models since they are a composition
of bivariate Pareto distributions (\citet{EV2020}); see Section~\ref{s:tree_models}
below.

By Theorem~\ref{t:mtp2} $ \mathrm{EMTP}_{2} $ is equivalent to the univariate
random variable $ W_{2}^{1} $ being strongly $ \mathrm{MTP}_{2} $. This gives
us the following result.
\begin{thm}
	\label{thm:bivariate_emtp2}
	A bivariate Pareto distribution is $ \mathrm{EMTP}_{2} $ if and only if the
	distribution of $ W^{1}_{2}$ is log-concave.
\end{thm}

Log-concave distributions include many known families like Gaussian, exponential,
uniform, beta or Laplace, but also the class of generalized extreme value
distributions such that many bivariate Pareto distributions are indeed
$ \mathrm{EMTP}_{2} $ for any parameter. The construction of an example where
the bivariate Pareto distribution is not $ \mathrm{EMTP}_{2} $ requires a
positive random variable $ W^{1}_{2} $ with
$\mathbb E (\exp W^{1}_{2}) = 1 $ for which $ W_{1}^{2} $ is not log-concave.
One example is when $ \exp (W^{1}_{2})$ is folded Laplace.
\begin{ex}
	\label{exmp8}
	Let $\exp (W^{1}_{2})=|X|$, where $X$ is distributed according to a Laplace
	distribution with mean $\mu $ and scale parameter $\sigma $. The density
	of $\exp (W^{1}_{2})$ equals
	\begin{equation*}
		f(y)=\frac{1}{\sigma} 
		\begin{cases} e^{-\mu /\sigma}\cosh (y/\sigma
			) & \text{for }0\le y< \mu,
			\\
			e^{-y/\sigma}\cosh (\mu /\sigma ) & \text{for }0\le \mu \le y; \end{cases}
	\end{equation*}
	see also \cite{LK2015}. The density of $ W^{1}_{2} $ equals
	$ e^{y}f(e^{y}) $, such that log-concavity of $ W^{1}_{2} $ requires that
	the second derivative of this is nonpositive. We compute for
	$ 0\le y<\mu $
	\begin{align*}
		\frac{\partial ^{2}}{\partial y^{2}} \biggl(y-\log (\sigma )- \frac{\mu}{\sigma}+\log \cosh \biggl(
		\frac{e^{y}}{\sigma} \biggr) \biggr)&=\frac{\partial ^{2}}{\partial y^{2}}\log \cosh \biggl(
		\frac{e^{y}}{\sigma} \biggr)
		\\
		&= \frac{e^{y}   (\sigma \tanh   (\frac{e^{y}}{\sigma}  )+e^{y} \operatorname{sech}^{2}  (\frac{e^{y}}{\sigma}  )  )}{\sigma ^{2}},
	\end{align*}
	which is clearly positive.
\end{ex}

\section{\texorpdfstring{$ \mathrm{EMTP}_{2} $}{EMTP2} in graphical extremes}
\label{sec:graphicalEx}

The previous section introduced the notion of $ \mathrm{EMTP}_{2} $. In this
section we study $ \mathrm{EMTP}_{2} $ in the context of extremal graphical
models. We focus on two aspects that we find particularly important. We
first discuss the case of extremal tree models and their latent counterparts,
which provide another strong theoretical argument for the usefulness of
the $ \mathrm{EMTP}_{2} $ constraint. We then characterize extremal conditional
independence structures that may appear in $ \mathrm{EMTP}_{2} $ distributions.

\subsection{Extremal tree models}
\label{s:tree_models}

For any undirected tree $T=(V,E)$, a multivariate Pareto distribution
$ \mathbf{Y}$ that is Markov to $ T $ is called an extremal tree model
(\citet{EH2020}). Such models also arise as the limits of regularly varying
Markov trees (\citet{Segers2020}). Define a directed tree
$ T^{k}=(V,E^{k}) $ rooted in $ k $ by directing all edges in $ T $ away
from $ k $. By \citet[Proposition~1]{EV2020}, for any $k \in V$, the extremal
function $ \mathbf{W}^{k} $ has the stochastic representation
\begin{align}
	W_{i}^{k}&=\sum_{e\in \operatorname{ph}(ki;T^{k})}W_{e},
	\label{eq:tree_representation} 
\end{align}
where $ \operatorname{ph}(ki;T^{k}) $ is the set of directed edges on the path from
$ k $ to $ i $ in $ T^{k} $ and $ \{W_{e},  e\in E^{k}\} $ is a set of
independent random variables, where $W_{e}$ with $e = (i,j)$ has the distribution
of $W^{i}_{j}$, that is, the $j$th component of the $i$th extremal function
of $\mathbf{Y}$.

For a H\"{us}sler--Reiss tree model on the tree $ T $, it was shown in
\citet[Proposition~4]{EV2020} that the extremal variogram defined in \eqref{evario}
is a tree metric, that is,
\begin{equation*}
	\Gamma ^{(k)}_{ij}=\sum_{mn\in \operatorname{ph}(ij;T)}
	\Gamma ^{(k)}_{mn}.
\end{equation*}
As a consequence, the minimum spanning tree with weights
$ \Gamma ^{(k)}_{ij}>0 $ is unique and equals the underlying tree
$ T $ \citep[Corollary~1]{EV2020}. The link of this model class to Brownian
motion tree models is established in Proposition~\ref{prop:BMTM} in Appendix~\ref{app:Gamma}.

\begin{prop}
	\label{prop:trees_mtp2}
	Extremal tree models are $ \mathrm{EMTP}_{2} $ if and only if all bivariate
	margins are $ \mathrm{EMTP}_{2} $, that is, if all $ W_{e} $ in
	\eqref{eq:tree_representation} have log-concave densities. This implies
	that H\"usler--Reiss tree models are always $ \mathrm{EMTP}_{2} $.
\end{prop}
In comparison, Gaussian tree models are $ \mathrm{MTP}_{2} $ if and only if
their covariance is nonnegative \citep[Proposition~5.3]{LUZ2019}. This
is one way to illustrate why $ \mathrm{EMTP}_{2} $ constraints are more natural
for extreme data than $ \mathrm{MTP}_{2} $ constraints are in the classical
case.

A generalization of an extremal tree model is an extremal
\emph{latent} tree model. The latter is defined as a multivariate Pareto
distribution $\mathbf{Y}$ obtained as the margin
$\tilde{\mathbf{Y}}_{O}$ of a larger extremal tree model
$(\tilde{\mathbf{Y}}_{O}, \tilde{\mathbf{Y}}_{U})$, where
$\tilde{\mathbf{Y}}_{O}$ and $\tilde{\mathbf{Y}}_{U}$ correspond to the
observed and unobserved variables, respectively. Extremal latent tree models
have been used in \cite{asenova2021inference} for modeling floods on a
river network. By Proposition~\ref{prop:marginex} every margin of an
$ \mathrm{EMTP}_{2} $ distribution is $ \mathrm{EMTP}_{2} $. This implies that
every extremal latent tree model is $ \mathrm{EMTP}_{2} $.

We note that the family of latent extremal tree models is much larger than
the family of extremal tree models and contains an extremal version of
the widely used one-factor model; see \cite{zwiernik2018latent} for more
examples and basic overview of latent tree models.
\begin{ex}[H\"{u}sler--Reiss one-factor model]
	\label{exmp9}
	Define an extremal one-factor model as the margin of an extremal tree model
	over a tree with a single inner node and all other vertices connected to
	it. Here the margin is taken over the outer nodes. Consider a $d$-dimensional
	H\"{u}sler--Reiss vector $\mathbf{Y}$ with parameter matrix
	$\Gamma $ with the following form. For a vector
	$\boldsymbol a=(a_{1},\ldots ,a_{d})$ with strictly positive entries, suppose
	that $\Gamma _{ij}=a_{i}+a_{j}$ for all $i\neq j$, $i,j \in [d]$. This
	is an extremal latent tree model since $\mathbf{Y}$ is the margin of a
	$(d+1)$-dimensional H\"usler--Reiss tree model on the star tree, where
	the $d$ leaves correspond to the observed variables and the central node
	is the unobserved variable; this can be seen since the extended
	$\Gamma $ with $\Gamma _{i(d+1)} = a_{i}$ is a tree metric on the star
	tree (\citet{EV2020}, Proposition~4). Using the covariance mapping
	\eqref{eq:gamma2sigma}, we see that, for every $i,j,k \in [d]$ with
	$i,j\neq k$,
	\begin{equation*}
		\Sigma ^{(k)}_{ij} = 
		\begin{cases} a_{k}
			& \text{if } i\neq j,
			\\
			a_{i}+a_{k} & \text{if } i=j. \end{cases} 
	\end{equation*}
	Carefully applying the Sherman--Morrison formula
	(\citet{hornjohnson}, Section~0.7.4), we see that, for all $i\neq j$,
	\begin{equation*}
		\Theta _{ij} = - \frac{\prod_{l\neq i,j}a_{l}}{\sum_{k=1}^{d}  \prod_{l\neq k} a_{l}} < 0,
	\end{equation*}
	which reconfirms that H\"{u}sler--Reiss one-factor models are
	$ \mathrm{EMTP}_{2} $.
\end{ex}

\subsection{Axioms for conditional independence and faithfulness}
\label{sec4.2}

Conditional independence models can be discussed in a purely combinatorial
way. We follow the definitions in \citet[Section~5]{FLSUWZ2017}. Let
$ \langle A,B| C\rangle $ be a ternary relation encoding abstract independence
of $ A $ and $ B $ conditioning on $ C $, where $ A$, $B$, $C $ are disjoint
subsets of $ V $. Here and in the following, unions $ A\cup B $ of two
sets $ A$, $B$ are abbreviated to $ AB $. A~conditional independence model
$ \mathcal{I} $ is a set of such relations. $ \mathcal{I} $ is called a
graphoid if it satisfies the following axioms for disjoint
$ A,B,C,D \subset V $:
\begin{itemize}
	\item[1.]
	$ \langle A,B| C\rangle \in \mathcal{I}$~$\Leftrightarrow$
	$\langle B,A | C\rangle \in \mathcal{I}$   (symmetry),
	\item[2.]
	$ \langle A,BD | C\rangle \in \mathcal{I} $~$ \Rightarrow$
	$\langle A,B| C\rangle \in \mathcal{I} \wedge \langle A,D| C
	\rangle \in \mathcal{I} $   (decomposition),
	\item[3.]
	$ \langle A,BD| C\rangle \in \mathcal{I} $~$ \Rightarrow$
	$\langle A,B| CD\rangle \in \mathcal{I} \wedge \langle A,D| BC
	\rangle \in \mathcal{I} $   (weak union),
	\item[4.]
	$ \langle A,B| CD\rangle \in \mathcal{I} \wedge \langle A,D| C
	\rangle \in \mathcal{I} $~$ \Leftrightarrow $ $  \langle A,BD| C
	\rangle \in \mathcal{I} $   (contraction),
	\item[5.]
	$ \langle A, B| CD \rangle \in \mathcal{I}  \wedge   \langle A, C
	| BD\rangle \in \mathcal{I} $~$ \Rightarrow  $ $ \langle A, BC
	| D \rangle \in \mathcal{I} $   (intersection).
\end{itemize}

A stochastic conditional independence model on a set of distributions is
always a semigraphoid, that is, it satisfies axioms (1)--(4). If in addition
the distributions have positive densities, then it is a graphoid. In the
discussion of \cite{EH2020}, Steffen Lauritzen discusses that extremal
conditional independence models are also semigraphoids. Under the assumption
of a positive density $f_{\mathbf{Y}}$, the intersection axiom for extremal
conditional independence for multivariate Pareto distributions follows
from the fact that then, for any $k\in V$, $\mathbf{Y}^{k} $ satisfies
the intersection axiom for classical stochastic conditional independence
because its density is proportional to $f_{\mathbf{Y}}$; see
\citet[Section~1.1.5]{Pearl2009} or
\citet[Proposition~3.1]{Lauritzen96}.

For classical conditional independence, if the distributions are
$ \mathrm{MTP}_{2} $, then \cite{FLSUWZ2017} show that the following additional
axioms are satisfied:
\begin{itemize}
	\item[(6)]
	$ \langle A,B | C \rangle \in \mathcal{I} \wedge \langle A,D | C
	\rangle \in \mathcal{I} $~$\Rightarrow $ $ \langle A,BD | C
	\rangle \in \mathcal{I} $   (composition),
	\item[(7)]
	$ \langle i,j | C \rangle \in \mathcal{I} \wedge \langle i,j | lC
	\rangle \in \mathcal{I} $~$\Rightarrow $ $ \langle i,l | C
	\rangle \in \mathcal{I}\vee \langle j,l | C \rangle \in
	\mathcal{I} $   (singleton-transitivity),
	\item[(8)]
	$ \langle A,B | C \rangle \in \mathcal{I} \wedge D\subseteq V
	\setminus AB$~$  \Rightarrow$ $   \langle A,B | CD \rangle \in
	\mathcal{I} $   (upward-stability).
\end{itemize}
As a consequence, these axioms also hold for extremal conditional independence
when the multivariate Pareto distribution is $ \mathrm{EMTP}_{2} $. In summary,
we have the following theorem.
\begin{thm}
	\label{thm:us_st_comp}
	Extremal conditional independence for a multivariate Pareto distribution
	with positive density is a graphoid. If in addition the distribution is
	$ \mathrm{EMTP}_{2} $, then it is also upward-stable, singleton-transitive
	and compositional.
\end{thm}
We omit the proof since the statements follows from the corresponding statements
for $\mathbf{Y}^{k}$, $k\in V$. For extremal conditional independence on
multivariate Pareto distributions with positive density, we note that the
following peculiarity arises. For instance, for $D=\emptyset $, the right-hand
side of Axiom (5) would lead to unconditional independence
$A \perp _{e} BC$, which is impossible, as discussed in Section~\ref{sec:ext_CI}.
This is not a contradiction to the validity of Axiom (5), since it can
be shown that in that case also the left-hand side can not arise.

\begin{remark}
	\label{rem:singletonCI}
	Similar to conditional independence, extremal conditional independence
	under $ \mathrm{EMTP}_{2} $ is equivalent to the respective collection of
	singleton conditional independences,
	\begin{align*}
		\mathbf{Y}_{A} \perp _{e} \mathbf{Y}_{B}|
		\mathbf{Y}_{C} \quad \Leftrightarrow \quad Y_{i} \perp
		_{e} Y_{j}| \mathbf{Y}_{C} \quad \forall i\in
		A,j \in B.
	\end{align*}
	This follows as in \citet[Corollary~1]{LS18}.
\end{remark}

Many constraint-based structure learning algorithms, like the PC algorithm
(\citet{spirtes2000causation}) that is used to learn the skeleton in directed
acyclic graphs, rely on the assumption that the dependence structure in
the data-generating distribution reflects faithfully the graph. A~distribution
in a graphical model over a graph $G$ is faithful if and only each conditional
independence corresponds exactly to graph separation. We define extremal
faithfulness analogously.

For an undirected graph $ G=(V,E) $, we can define an independence model
$ \mathcal{I}(G) $ through graph separation with respect to $ G $ by
\begin{equation*}
	\langle A,B| C \rangle \in \mathcal{I}(G) \quad\Longleftrightarrow \quad C
	\text{ separates } A \text{ from }B,
\end{equation*}
where the latter means that all paths on $G$ between $ A $ and $ B $ cross
$ C $. On the other hand, we can introduce an independence model
$ \mathcal{I}_{e}(\mathbb P_{\mathbf Y}) $ for a multivariate Pareto distributions
$ \mathbf{Y}$ through
\begin{equation*}
	\langle A,B| C\rangle \in \mathcal{I}_{e}(\mathbb P_{\mathbf Y})
	\quad \Longleftrightarrow \quad \mathbf{Y}_{A}\perp _{e}
	\mathbf{Y}_{B} | \mathbf{Y}_{C}.
\end{equation*}
We further define the extremal pairwise independence graph
$ G_{e}(\mathbb P_{\mathbf Y}) $ such that
\begin{equation*}
	(i,j)\in E \quad \Longleftrightarrow \quad \bigl\langle i,j| V\setminus \{i,j\}
	\bigr\rangle \in \mathcal{I}_{e}(\mathbb P_{\mathbf Y}).
\end{equation*}
A multivariate Pareto distribution $ \mathbb P_{\mathbf Y} $ is said to
be extremal faithful to a graph $G$, if
$ \mathcal{I}_{e}(\mathbb P_{\mathbf Y})=\mathcal{I}(G)$.

\begin{thm}
	\label{t:faithful}
	Let $\mathbf{Y}$ be a multivariate Pareto distribution with positive and
	continuous density. If $\mathbf{Y}$ is in addition
	$ \mathrm{EMTP}_{2} $, then
	$\mathcal I_{e}(\mathbb P_{\mathbf{Y}}) = \mathcal I(G_{e}(\mathbb P_{
		\mathbf Y}))$; that is, $\mathbb P_{\mathbf{Y}}$ is extremal faithful to
	its pairwise independence graph.
\end{thm}
The proof is similar to \citet[Theorem~6.1]{FLSUWZ2017} who show that
stochastic independence models that are $ \mathrm{MTP}_{2} $ and satisfy the
intersection axiom are always faithful to the corresponding pairwise independence
graph. It is available in Appendix~\ref{pr:faithful}.

\section{Learning totally positive H\"usler--Reiss distributions}
\label{s:estimation}

The work of \cite{SH2015} and \cite{LUZ2019} show that, in the Gaussian
case, the MLE under $ \mathrm{MTP}_{2} $ has many nice properties. For example,
the maximum likelihood estimator exists with probability 1 as long as the
sample size is at least two and the $ \mathrm{MTP}_{2} $ constraint works
as an implicit regularizer. In this section we study the estimation of
H\"usler--Reiss distributions under the $ \mathrm{EMTP}_{2} $ constraint replacing
the likelihood function with a surrogate likelihood.

\subsection{Surrogate likelihood and its dual}
\label{s:surrogate}

In order to use properties of Gaussian maximum likelihood theory, we apply
a transformation to a H\"usler--Reiss Pareto distribution
$\mathbf{Y}$. Recall that $\mathbf{Y}^{k}$ is defined as the conditioned
random vector $\mathbf{Y}| \{Y_{k} > 0\}$ and that from Section~\ref{sec:huesler_reiss}
we have for a H\"usler--Reiss distribution with parameter matrix
$\Gamma \in \mathcal C^{d}$
\begin{align}
	\label{Y_normal} \bigl(Y^{k}_{i} - Y^{k}_{k}
	\bigr)_{i\neq k} = \mathbf{W}^{k}_{
		\setminus k} \sim N\bigl(-
	\operatorname{diag}\bigl(\Sigma ^{(k)}\bigr)/2, \Sigma ^{(k)}
	\bigr). 
\end{align}
Consider a data matrix $y \in \mathbb R^{n \times d}$ of $n$ independent
observations of $\mathbf{Y}$ with $i$th row
$(y_{i1},\ldots ,y_{id})$. Let
$\mathcal I_{k} = \{i\in [n] : y_{ik} > 0\}$ be the index set of observations
where the $k$th coordinate exceeds zero. If $|\mathcal I_{k}|\ge 2$, for
any $i\in \mathcal I_{k}$, $ k\in [d] $, we define independent observations
$\boldsymbol w_{i}$ of $\mathbf{W}^{k}$ by
\begin{equation*}
	w_{ij} = y_{ij} - y_{ik}, \quad j=1,\ldots ,d,
\end{equation*}
and let the corresponding sample covariance matrix be
\begin{equation}
	\label{eq:Omegak} \Omega ^{(k)}=\frac{1}{ \llvert \mathcal{I}_{k} \rrvert }\sum
	_{i\in \mathcal I_{k}}( \boldsymbol w_{i}-\bar{\boldsymbol w}) (
	\boldsymbol w_{i}- \bar{\boldsymbol w})^{T}\quad \text{where
	} \bar{\boldsymbol w}= \frac{1}{ \llvert \mathcal I_{k} \rrvert }\sum_{i\in \mathcal I_{k}}
	\boldsymbol w_{i}. 
\end{equation}
Note that, by construction, $w_{ik}=0$ for all $i\in \mathcal I_{k}$, and
so the $k$th row/column of $\Omega ^{(k)}$ is zero. If
$|\mathcal I_{k}|<2$, set $\Omega ^{(k)}=0$. We obtain the empirical variogram
$ \bar{\Gamma}^{(k)} $ from $\Omega ^{(k)} $ via the inverse covariance
mapping
\begin{equation}
	\label{eq:invFarris} \bar{\Gamma}_{ij}^{(k)}= \Omega
	_{ii}^{(k)}+\Omega _{jj}^{(k)}-2 \Omega
	_{ij}^{(k)}\quad \text{for all }i,j. 
\end{equation}
Because the index set $\mathcal I_{k}$ depends on $k$, the estimator
$ \bar{\Gamma}^{(k)} $ also depends on $k$. In order to obtain an estimate
of $ \Gamma $ that is symmetric and uses all data, we define the combined
empirical variogram as
\begin{equation}
	\label{eq:GS} \overline\Gamma :=\frac{1}{d} \sum
	_{k=1}^{d}\bar{\Gamma}^{(k)};
\end{equation}
see also \citet[Corollary~2]{EV2020}. For each $k=1,\ldots ,d$, via the
covariance mapping \eqref{eq:gamma2sigma}, we obtain an empirical covariance
$ S^{(k)} $ from $\overline\Gamma $.

The inverse covariance matrix of $\mathbf{W}^{k}_{\setminus k}$ can be
estimated by maximizing the surrogate log-likelihood that takes the form
\begin{equation}
	\label{eq:surrogatelike} \ell \bigl(\Theta ^{(k)};S^{(k)}\bigr) :=\log
	\det \Theta ^{(k)}-\operatorname{tr}\bigl(S^{(k)} \Theta
	^{(k)}\bigr), 
\end{equation}
which is derived from \eqref{Y_normal} by dropping the likelihood contribution
of the mean vector $-\operatorname{diag}(\Sigma ^{(k)})/2$. Maximizing this would
result in an estimate $\widehat\Theta ^{(k)}$ of $\Theta ^{(k)}$ that is
close to the maximum likelihood estimator since the mean vector only contains
information on the diagonal of $\Sigma ^{(k)}$. Note that the function
$ \ell (\Theta ^{(k)};S^{(k)}) $ in \eqref{eq:surrogatelike} is directly
related to the log-determinantal Bregman divergence
(e.g., \citet{ravikumar2011high}), so its use can be justified outside
of the Gaussian setting.

A more elegant formulation of the surrogate log-likelihood that is independent
of $k$ is given next. For a square matrix $A$, denote by
$\operatorname{Det}(A)$ its pseudo-determinant, that is, the product of all nonzero
eigenvalues. For $\Theta \in \mathbb S^{d}$ with
$\Theta \boldsymbol 1=\boldsymbol 0$ and
$ Q_{ij}:=-\Theta _{ij}$, $i\neq j $, the weighted matrix-tree theorem
(\citet{duval2009simplicial}) yields for any $k\in [d]$ that
\begin{align}
	\label{eq:det_spanning_trees} \operatorname{Det}(\Theta ) &= d\cdot \det \bigl(\Theta
	^{(k)}\bigr) = d\cdot \sum_{T\in \mathcal T}\prod
	_{ij\in T} Q_{ij},
\end{align}
where $\mathcal T$ is the set of all spanning trees over the complete graph
with vertices $\{1,\ldots ,d\}$.

As in Appendix~\ref{app:Gamma}, we equip $\mathbb{S}^{d}_{0}$ with the
inner product
$\llangle A,B\rrangle   :=\sum_{i<j}A_{ij}B_{ij}$.

\begin{lemma}
	\label{lem:tolapl}
	The right-hand side of \eqref{eq:surrogatelike} can be rewritten in terms
	of $\Theta $ as
	\begin{equation}
		\label{eq:surrogatelike2} \ell (\Theta ;S)=\log\operatorname{Det} \Theta -\langle S,\Theta
		\rangle - \log (d), 
	\end{equation}
	where $S=\boldsymbol P(-\frac{1}{2}\overline{\Gamma}) \boldsymbol P$ with
	$\boldsymbol P$ defined in \eqref{eq:P1} or, equivalently, in terms of
	$Q\in \mathbb{S}^{d}_{0}$ as
	\begin{equation}
		\label{eq:likeinQ} \ell (Q;\overline{\Gamma}) = \log \biggl(\sum
		_{T\in \mathcal T}\prod_{ij
			\in T} Q_{ij}
		\biggr)-\llangle \overline{\Gamma},Q\rrangle . 
	\end{equation}
\end{lemma}

Note that it follows from \eqref{eq:likeinQ} that a proportional representation
of the log-likelihood in terms of $\Theta $ and $\overline{\Gamma}$ is
given by \eqref{eq:logdet}.

In order to enforce the $ \mathrm{EMTP}_{2} $ constraint for the H\"usler--Reiss
distribution, we propose to solve a restricted optimization problem using
the characterization in Theorem~\ref{t:equivalences_precision_matrices_emtp2}.
Recall that $ \mathbb{U}^{d}_{+} \subset \mathbb{S}^{d} $ denotes the set
of all graph Laplacians for connected graphs with positive weights on each
edge. Slightly abusing notation, we also denote by $\mathbb{U}_{+}$ its
image in $\mathbb{S}_{0}$, that is, the points in the nonnegative orthant
of $\mathbb{S}_{0}$ whose support is a connected graph. Thus, for any fixed
$\overline \Gamma $ we consider the problem of maximizing
$\ell (Q;\overline\Gamma )$ in \eqref{eq:likeinQ} over
$\mathbb{U}_{+}$, that is,
\begin{equation}
	\label{eq:emtp2_opt_theta} \widehat{Q} :=\arg \max_{Q\in \mathbb U_{+}} \ell (Q;
	\overline\Gamma ). 
\end{equation}
This is a convex optimization problem because
$\ell (Q;\overline\Gamma )$ is a strictly concave function over the convex
set $\mathbb U_{+}$.

We call $\widehat Q$ a surrogate maximum likelihood estimator for
$ Q$ under $ \mathrm{EMTP}_{2} $. To obtain an estimator
$ \widehat{\Gamma} $ for the variogram under $ \mathrm{EMTP}_{2} $, we first
take $\widehat\Theta $, given by
$\widehat\Theta _{ij}=-\widehat Q_{ij}$, to define
$\widehat\Sigma =\widehat\Theta ^{+}$ and then map
\begin{equation*}
	\widehat \Gamma _{ij} = \widehat\Sigma _{ii}+\widehat\Sigma
	_{jj}-2 \widehat\Sigma _{ij},
\end{equation*}
as explained in Appendix~\ref{app:Gamma}. This is not the maximum likelihood
estimator of $\Gamma $ under the $ \mathrm{EMTP}_{2} $ constraint because
we have dropped the contribution of the mean vector as in \eqref{eq:surrogatelike}.
Nevertheless, $ \widehat{\Gamma} $ is a very natural estimator since it
has a simple interpretation in terms of the input matrix
$\overline \Gamma $; see Theorem~\ref{th:kktdd} below. Moreover, we show
in Proposition~\ref{t:consistent} that it is a consistent estimator.

To analyze this optimization problem in more detail, we study it from the
perspective of convex analysis. We first derive its dual problem.
\begin{prop}
	\label{prop:dual}
	The dual problem of \eqref{eq:emtp2_opt_theta} is
	\begin{equation}
		\label{eq:dual} \text{maximize } \log \det \left( 
		\begin{bmatrix} 0 &
			-\boldsymbol 1^{T}
			\\
			\boldsymbol 1 & -\frac {1}{2}\Gamma \end{bmatrix} 
		\right)+(d-1) \quad \text{subject to } \Gamma \in \mathcal{C}^{d} \text{
			and } \Gamma \leq \overline\Gamma . 
	\end{equation}
\end{prop}
\begin{proof}
	Let $\mathcal K^{d}$ be the set of all $Q\in \mathbb{S}^{d}_{0}$ such that
	the corresponding $\Gamma $ lies in $\mathcal C^{d}$. For given
	$\overline{\Gamma}$ we define the extended-real-valued function
	\begin{equation}
		\label{eq:extf} f(Q) = 
		\begin{cases} -\ell (Q; \overline{\Gamma}) &
			\text{if } Q\in \mathcal K^{d},
			\\
			+\infty & \text{otherwise}. \end{cases} 
	\end{equation}
	The problem in (\ref{eq:emtp2_opt_theta}) can be therefore reformulated
	as follows:
	\begin{equation}
		\label{eq:mtp2optproblem2} \text{minimize } f(Q)\quad \text{subject to } Q\geq 0.
	\end{equation}
	The Lagrangian for this problem is
	$f(Q)-\llangle \Lambda ,Q\rrangle  $, where
	$\Lambda \in \mathbb{S}^{d}_{0}$ and $\Lambda _{ij}\geq 0$ for all
	$1\leq i< j\leq d$ (Lagrange multipliers of the nonnegative constraints).
	Clearly,
	\begin{equation*}
		\sup_{\Lambda \geq 0} \bigl\{f(Q)-\llangle \Lambda ,Q \rrangle \bigr\} =
		\begin{cases} f(Q) & \text{if }Q\geq 0,
			\\
			+\infty & \text{otherwise}. \end{cases} 
	\end{equation*}
	This implies that Problem \eqref{eq:mtp2optproblem2} is equivalent to
	\begin{equation*}
		\inf_{Q} \sup_{\Lambda \geq 0} \bigl\{f(Q)-\llangle
		\Lambda ,Q\rrangle \bigr\},
	\end{equation*}
	where the infimum is unrestricted. By duality theory (Slater's conditions),
	we obtain the same value by swapping $\inf $ and $\sup $. We obtain the
	Lagrange dual function
	\begin{equation*}
		\inf_{Q} \bigl\{f(Q)-\llangle \Lambda ,Q\rrangle \bigr\}.
	\end{equation*}
	If the infimum exists, it is obtained at the unique
	$Q\in \mathcal K^{d}$ for which the gradient of
	$f(Q)-\llangle \Lambda ,Q\rrangle  $ vanishes. By Proposition~\ref{prop:Gamma2Q}
	\begin{equation*}
		\nabla _{Q} \log \biggl(\sum_{T\in \mathcal T}\prod
		_{ij\in T}Q_{ij}\biggr)= \Gamma ,
	\end{equation*}
	and so
	\begin{equation*}
		\nabla \bigl\{f(Q)-\llangle \Lambda ,Q\rrangle \bigr\} = -\Gamma +\overline{
			\Gamma}-\Lambda ,
	\end{equation*}
	showing that the optimal point must satisfy
	$\Gamma \leq \overline{\Gamma}$ and
	$\Lambda =\overline{\Gamma}-\Gamma $ and so optimizing the dual function
	is equivalent to optimizing a function of $\Gamma $ of the form
	\begin{equation}
		\label{eq:hdual} h(\Gamma ) = 
		\begin{cases} \log \det \bigl(\Sigma
			^{(k)}(\Gamma )\bigr)+(d-1) & \text{if }\Gamma \leq \overline{\Gamma},
			\Gamma \in \mathcal C^{d},
			\\
			-\infty & \text{otherwise}. \end{cases} 
	\end{equation}
	Finally, we use the Cayley--Menger formula that states that, for every
	$k=1,\ldots ,d$,
	\begin{align}
		\label{eq:menger} \det \bigl(\Sigma ^{(k)}\bigr)& = \det \left(
		\begin{bmatrix} 0 & -\boldsymbol 1^{T}
			\\
			\boldsymbol 1 & -\frac {1}{2}\Gamma \end{bmatrix} \right). 
	\end{align}
\end{proof}

The proof of the previous result and the KKT conditions imply the following
theorem.
\begin{thm}
	\label{th:kktdd}
	The point $(\widehat Q, \widehat \Gamma )$ is the unique optimal point
	of $f(Q)$ over $Q\in \mathbb U_{+}$ if and only if:
	\begin{itemize}
		\item [(i)] $\widehat Q_{ij}\geq 0$ for all $1\leq i<j\leq d$,
		\item [(ii)] $\overline{\Gamma}_{ij}\geq \widehat \Gamma _{ij}$ for all
		$1\leq i<j\leq d$,\vspace*{1pt}
		\item [(iii)]
		$(\overline{\Gamma}_{ij}- \widehat \Gamma _{ij})\widehat Q_{ij}=0$ for
		all $1\leq i<j\leq d$.
	\end{itemize}
\end{thm}

The condition in (iii) implies that the $ \mathrm{EMTP}_{2} $ estimator acts
as an implicit regularizer since some of the entries of $\widehat Q$ will
be set to zero. We, therefore, define the $ \mathrm{EMTP}_{2} $ graph
$\widehat G= (V, \widehat E)$ as the graph with edges
\begin{equation*}
	(i,j) \notin \widehat E \quad \Longleftrightarrow \quad \widehat
	Q_{ij} = 0,
\end{equation*}
which corresponds to the extremal pairwise independence graph of the H\"usler--Reiss
distribution with parameter matrix $\widehat \Gamma $.

We find as a simple corollary of Theorem~\ref{th:kktdd} that the
$ \mathrm{EMTP}_{2} $ estimator equals the surrogate maximum likelihood estimator
for the graphical model with respect to the estimated
$ \mathrm{EMTP}_{2} $ graph $\widehat G$.

\begin{cor}
	\label{cor5.4}
	Let $ \widehat{G}=(V,\widehat{E}) $ be the $ \mathrm{EMTP}_{2} $ graph corresponding
	to $ \widehat{Q} $. It follows that the surrogate maximum likelihood estimator
	of the extremal graphical model with respect to $ \widehat{G} $, that is,
	\begin{equation*}
		\widecheck{Q}=\argmax \ell (Q;\overline{\Gamma})\quad \text{ subject to }
		Q_{ij}=0 \text{ for all } (i,j) \notin \widehat{E}
	\end{equation*}
	equals the $ \mathrm{EMTP}_{2} $ estimator $ \widehat{Q} $.
\end{cor}
\begin{proof}
	It holds from simple derivation and the graphical model constraints that
	\begin{align}
		(\widecheck{\Gamma}_{ij}-\overline{\Gamma}_{ij})
		\widecheck{Q}_{ij}=0 \label{eq:emtp2_graphical} 
	\end{align}
	for all $ 1\le i<j\le d $. As conditionally negative definite matrix completion
	is unique (\citet{Hentschel2021}) and \eqref{eq:emtp2_graphical} is identical
	to Theorem~\ref{th:kktdd}(iii), the corollary follows.
\end{proof}

\subsection{Existence of the optimum and its consistency}
\label{sec5.2}

Theorem~\ref{t:equivalences_precision_matrices_emtp2} and Lemma~\ref{lem:tolapl}
show that optimizing \eqref{eq:surrogatelike} with respect to all diagonally
dominant M-matrices $\Theta ^{(k)}$ is equivalent to optimizing
$-\log\operatorname{Det} \Theta +\langle S,\Theta \rangle $ over all Laplacian
matrices $\Theta $ of connected graphs, as described in
\eqref{eq:emtp2_opt_theta}. This is precisely the optimization problem
considered in equation (3) in \citet{ying2021minimax}. They show in Theorem~1
that the optimum in \eqref{eq:emtp2_opt_theta} exists almost surely. The
proof of this result in the supplement of \citet{ying2021minimax} actually
reveals a more detailed statement, which is useful for our purposes:
\begin{thm}
	\label{th:gammapos}
	The optimum of the problem \eqref{eq:emtp2_opt_theta} exists if and only
	if $\overline{\Gamma}_{ij}>0$ for all $1\leq i<j\leq d$.
\end{thm}

Note that $\overline{\Gamma}_{ij}=0$ if and only if
$\bar\Gamma ^{(k)}_{ij}=0$ for all $k$; see (\ref{eq:GS}). Moreover,
$\bar\Gamma ^{(k)}_{ij}=0$ if and only if
$\Omega ^{(k)}_{ii}=\Omega ^{(k)}_{ij}=\Omega ^{(k)}_{jj}$. This happens
with probability zero with respect to the underlying sample, unless the
corresponding index set $\mathcal I_{k}$ satisfies
$|\mathcal I_{k}|< 2$. Thus, with probability one,
$\overline{\Gamma}_{ij}>0$, unless for each $k\in [d]$ the event
$\{Y_{k}>0\}$ is observed at most once in the sample.

We finish this section providing a consistency result that uses consistency
of $\overline{\Gamma}$ and the Berge's maximum theorem
(see \citet{berge}, Section VI.3).

\begin{prop}
	\label{prop:maxth}
	The function
	\begin{equation*}
		\overline{\Gamma} \mapsto \widehat \Gamma =\argmax _{\Gamma \in
			\mathcal{C}^{d}\cap \{\Gamma \leq \overline{\Gamma}\}} \log \det
		\left( 
		\begin{bmatrix} 0 & -\boldsymbol 1^{T}
			\\
			\boldsymbol 1 & -\frac {1}{2}\Gamma \end{bmatrix} 
		\right)
	\end{equation*}
	is a continuous function over all $\overline{\Gamma}$ such that
	$\overline{\Gamma}_{ij}>0$ for all $i\neq j$.
\end{prop}
\begin{proof}
	Consider the function $f:\mathcal{C}^{d}\to \mathbb{R}$ given by
	\begin{equation*}
		f(\Gamma ):=\log \det \left( 
		\begin{bmatrix} 0 & -\boldsymbol
			1^{T}
			\\
			\boldsymbol 1 & -\frac {1}{2}\Gamma \end{bmatrix}
		\right).
	\end{equation*}
	This function is indeed well defined by the Cayley--Menger formula in
	\eqref{eq:menger} and Lemma~\ref{lem:pdC} in Appendix~\ref{app:Gamma}.
	If $\overline{\Gamma}_{ij}>0$ for all $i\neq j$, then the set of all
	$\Gamma \in \mathcal C^{d}$ satisfying
	$\Gamma \leq \overline{\Gamma}$ (for a fixed $\overline{\Gamma}$) is a
	bounded nonempty set. Since $\Gamma _{ij}>0$ for all $i<j$, by Proposition~\ref{th:gammapos}
	there is a unique point in this set that maximizes $f(\Gamma )$, which
	shows that the $\arg \max $ mapping in the statement is indeed a well-defined
	function. Moreover, by strict concavity of $f(\Gamma )$, the same holds
	if we maximize $f$ over the closure of
	$\mathcal C^{d}\cap \{\Gamma \leq \overline{\Gamma}\}$. This is a compact
	set, which we denote by
	$g(\overline{\Gamma}):=\overline{\mathcal C^{d}}\cap \{\Gamma \leq
	\overline{\Gamma}\}$, where $\overline{\mathcal C^{d}}$ is the closure
	of ${\mathcal C^{d}}$. Consider the set $\mathcal K$ of all compact subsets
	in $\mathbb{S}^{d}_{0}$. This set forms a metric space with the Hausdorff
	distance
	\begin{equation*}
		\mathbb D(C,D):= \max \Bigl\{\max_{\Gamma \in C} d_{D}(
		\Gamma ),\max_{
			\Gamma \in D} d_{C}(\Gamma ) \Bigr\},\quad C,D
		\in \mathcal K,
	\end{equation*}
	where $d_{C}(\Gamma )$ denotes the Euclidean distance of
	$\Gamma \in \mathbb{S}^{d}_{0}$ to the set $C\in \mathcal K$. The mapping
	$g:\mathbb{S}^{d}_{0}\to \mathcal K$, defined as above, is a mapping between
	two metric spaces. This map is continuous if and only if for every sequence
	if $\Gamma _{n}\to \overline{\Gamma}$, then
	$g(\Gamma _{n})\to g(\overline{\Gamma})$. Equivalently, we want to show
	that
	\begin{equation}
		\label{eq:convDD} \llVert \Gamma _{n}-\overline{\Gamma} \rrVert \to 0
		\quad \Longrightarrow \quad \mathbb D\bigl(g(\Gamma _{n}),g(\overline{
			\Gamma})\bigr)\to 0.
	\end{equation}
	Let $C_{n}=g(\Gamma _{n})$ and $D=g(\overline{\Gamma})$. Since
	$\overline{\Gamma}$ is an interior point of $\mathcal{C}^{d}$, we can assume
	that $\Gamma _{n}\in \mathcal{C}^{d}$ as well. But for every
	$\Gamma \in \overline{\mathcal{C}^{d}}$,
	\begin{equation*}
		d_{D}(\Gamma ) = \sqrt{\sum_{i<j} \bigl(
			\max \bigl\{(\Gamma - \overline{\Gamma})_{ij},0\bigr\}
			\bigr)^{2}}\leq \sqrt{\sum_{i<j}(
			\Gamma - \overline{\Gamma})_{ij}^{2}}= \llVert \Gamma -
		\overline{\Gamma} \rrVert .
	\end{equation*}
	The same argument shows that, for every $\Gamma \in C_{n}$, we have
	$d_{D}(\Gamma ) \leq  d_{D}(\Gamma _{n})$, and so
	\begin{equation*}
		\max_{\Gamma \in C_{n}} d_{D}(\Gamma ) = d_{D}(
		\Gamma _{n}) \leq \llVert \Gamma _{n}-\overline{\Gamma}
		\rrVert .
	\end{equation*}
	By symmetry we can also show that
	$\max_{\Gamma \in D} d_{C_{n}}(\Gamma )\leq \|\Gamma _{n}-
	\overline{\Gamma}\|$, which implies
	\begin{equation*}
		\mathbb D\bigl(g(\Gamma _{n}),g(\overline{\Gamma})\bigr) \leq
		\llVert \Gamma _{n}- \overline{\Gamma} \rrVert
	\end{equation*}
	and thus also \eqref{eq:convDD}. We have established continuity of
	$g$. By the maximum theorem in \citet[Section VI.3]{berge}, the function
	$\overline{\Gamma}\mapsto \argmax _{\Gamma \in g(\overline{\Gamma})}f(
	\Gamma )$ is also continuous.
\end{proof}
As a consequence, we establish the consistency of the estimator
$\widehat{\Theta}$ under $ \mathrm{EMTP}_{2} $.

\begin{thm}
	\label{t:consistent}
	Let $\mathbf{Y}$ be an $ \mathrm{EMTP}_{2} $ H\"usler--Reiss distribution
	with parameter matrix $\Gamma $. Let $\overline \Gamma $ be a consistent
	estimator of $\Gamma $ as the sample size $n \to \infty $. Then the
	$ \mathrm{EMTP}_{2} $ estimator $\widehat{\Gamma}$ based on
	$\overline \Gamma $ is consistent, that is, for any $\varepsilon >0$,
	\begin{equation*}
		\mathbb P \Bigl( \max_{i,j \in V} \llvert \widehat \Gamma
		_{ij} - \Gamma _{ij} \rrvert > \varepsilon \Bigr) \to 0,
		\quad n\to \infty .
	\end{equation*}
\end{thm}
\begin{proof}
	Since $\widehat \Gamma $ is a continuous function of
	$\overline{\Gamma}$ by Proposition~\ref{prop:maxth}, it follows that
	$\widehat\Gamma $ converges in probability to the true $\Gamma $ by the
	continuous mapping theorem.
\end{proof}
\begin{remark}
	\label{rem5.8}
	\citet[Theorem~1]{EV2020} show that, under certain assumptions, the empirical
	variogram $\overline{\Gamma}$ is consistent, that is, it converges in probability
	to the true $\Gamma $; see also Section~\ref{sec:doa}.
\end{remark}

The previous theorem does not imply a consistent recovery of the graph,
and in fact, the $ \mathrm{EMTP}_{2} $ algorithm does not directly enforce
sparsity. Sparsity is, however, often induced indirectly by the KKT conditions.
While it is not expected that the graph structure is recovered in general,
one can show that the estimated $ \mathrm{EMTP}_{2} $ graph is with high probability
a super-graph of the true underlying structure. In applications this is
particularly useful in cases where the estimated $ \mathrm{EMTP}_{2} $ graph,
and, therefore, the true underlying graph, is very sparse; see Section~\ref{danube}
for an example.

\begin{thm}
	\label{t:supergraph}
	Let $\mathbf{Y}$ be an $ \mathrm{EMTP}_{2} $ H\"usler--Reiss distribution
	that is an extremal graphical model on its extremal pairwise independence
	graph $G = (V,E)$, that is, $\Theta _{ij} = 0$ if and only if
	$(i,j) \notin E$. Suppose that $\overline \Gamma $ is a consistent estimator
	of $\Gamma $ as the sample size $n \to \infty $, and let
	$\widehat{\Gamma}$ be the corresponding $ \mathrm{EMTP}_{2} $ estimator. Then
	the estimated $ \mathrm{EMTP}_{2} $ graph
	$\widehat G = (V, \widehat E)$ is asymptotically a super-graph of the true
	underlying graph $G$. More precisely,
	\begin{equation*}
		\mathbb P ( E \subseteq \widehat E ) \to 1, \quad n\to \infty .
	\end{equation*}
\end{thm}
\begin{proof}
	By Theorem~\ref{th:kktdd} (i), it holds that $ \widehat{Q}\ge 0 $. It follows
	that
	\begin{align*}
		\mathbb{P} ( E \subseteq \widehat E )&=\mathbb{P} \bigl( \forall (i,j)\in E:
		\widehat{Q}_{ij}>0 \bigr)
		\\
		&\ge 1-\sum_{(i,j)\in E}\mathbb{P}(\widehat{Q}_{ij}=0).
	\end{align*}
	Because $ \widehat{\Gamma} $ is consistent by Theorem~\ref{t:consistent},
	it follows that $ \widehat{Q} $ is consistent by the continuous mapping
	theorem, as the maps from $ \widehat{\Gamma} $ to
	$ \hat{\Sigma}^{(k)} $ and $ \hat{\Theta}^{(k)} $ to $ \widehat{Q} $ are
	linear and matrix inversion is continuous. Hence, there exists some
	$\varepsilon >0$ with
	\begin{equation*}
		\mathbb{P}(\widehat{Q}_{ij}>\varepsilon )\rightarrow 1\quad
	\end{equation*}
	for all $ (i,j)\in E $. This implies that the probabilities
	$ \mathbb{P}(\widehat{Q}_{ij}=0) $ tend to zero, and consequently,
	$ \mathbb{P}  ( E \subseteq \widehat E   ) $ tends to one as
	$n\to \infty $.
\end{proof}

\begin{remark}
	\label{rem5.10}
	Even if the distribution of $\overline\Gamma $ is asymptotically normal,
	the distribution of $\widehat \Theta $ will be typically intractable. It
	will be equal to a mixture of projections of the Gaussian distribution
	on various faces of the polyhedral cone defined by nonnegativity of
	$Q$. Even if it was possible to understand this distribution, it would
	be still hard to handle, as the number of mixture components is exponential
	in $d$.
\end{remark}

\section{An optimization algorithm}
\label{s:algorithms}

Our aim in this section is to develop a numerical algorithm to optimize
the surrogate likelihood in \eqref{eq:surrogatelike2} in terms of
$\Theta $ (equiv. \eqref{eq:likeinQ} in terms of $Q$). A~natural first
idea is a projected coordinate descent algorithm, as both the gradient
of this function has a simple form and the projection on the set
$\Theta \leq 0$ is straightforward. This is precisely the algorithm proposed
in \cite{ying2021minimax}. We note, however, that ensuring that at each
iteration $\Theta $ is a Laplacian of a \emph{connected} graph is harder
and it occasionally leads to numerical issues.

In what follows, we develop a block coordinate descent algorithm that optimizes
the dual problem updating $\Gamma $ row by row. This algorithm carefully
exploits the structure of the problem and relies on quadratic programming.
Although in our setting $S$ that appears in \eqref{eq:surrogatelike2} is
a positive semidefinite matrix satisfying
$S\boldsymbol 1=\boldsymbol 0$, our algorithm takes as input
\emph{any} positive semidefinite matrix satisfying $S_{ii}>0$ for all
$i\in [d]$ and $S_{ij}<\sqrt{S_{ii}S_{jj}}$ for all $i\neq j$. We observe
that our algorithm is more stable than the projected gradient descent algorithm
in the case when $S$ is rank deficient.

\subsection{General description of the algorithm}
\label{sec6.1}

Our algorithm is a block coordinate descent algorithm that optimizes the
dual problem \eqref{eq:dual}. We refer to all
$\Gamma \in \mathcal{C}^{d}$ satisfying
$\Gamma \leq \overline{\Gamma}$ as the dually feasible points. The algorithm
starts at some given dually feasible point, and it updates the
$\Gamma $ matrix row by row. At each step the value of the function increases,
and the corresponding point is dually feasible. Updating a row requires
solving a quadratic problem. This is similar to the algorithms used for
the graphical LASSO (\citet{banerjee2008model,LZ2020}) but with important
twists.

Denote $A=-\frac{1}{2}\Gamma $, and assume $d\geq 3$. After suitably reordering
the rows/columns of $A$, for any $i=1,\ldots ,d$, we can rewrite the determinant
in \eqref{eq:dual} as
\begin{align}
	-\det 
	\begin{bmatrix} 0 &1 & \boldsymbol 1^{T}
		\\
		1 & 0 & A_{\setminus i,i}^{T}
		\\
		\boldsymbol 1 & A_{\setminus i,i} & A_{\setminus i,\setminus i} \end{bmatrix}.
	\label{eq:menger2} 
\end{align}
The goal in the dual problem \eqref{eq:dual} is to optimize this expression
subject to $\Gamma \leq \overline{\Gamma}$. Instead, we optimize this expression
only with respect to $\boldsymbol y=A_{\setminus i,i}$. This will lead
to a quadratic optimization problem that we can easily solve.

Let $B=(A_{\setminus i,\setminus i})^{-1}$. Since
$\Gamma \in \mathcal C^{d}$, also
$ \Gamma _{\setminus i,\setminus i}\in \mathcal C^{d-1} $ and, in particular,
$\boldsymbol 1^{T} \Gamma _{\setminus i,\setminus i}\boldsymbol 1>0$. In
consequence, by \citet[Lemma~3.2]{micchelli1986}
$\Gamma _{\setminus i,\setminus i}$ has $ d-2 $ negative eigenvalues and
one positive eigenvalue. Hence,
\begin{equation*}
	{\det (A_{\setminus i,\setminus i})} = \frac{1}{(-2)^{d-1}}{ \det (\Gamma
		_{\setminus i,\setminus i})}<0.
\end{equation*}
Using the standard Schur complement arguments, \eqref{eq:menger2} can be
written as
\begin{equation*}
	-\det (A_{\setminus i,\setminus i})\cdot \bigl(\boldsymbol y^{T} \bigl(
	\boldsymbol 1^{T}B\boldsymbol 1B-B\boldsymbol 1\boldsymbol
	1^{T} B \bigr)\boldsymbol y+2\boldsymbol 1^{T} B \boldsymbol
	y-1 \bigr),
\end{equation*}
which has to be maximized with respect to $\boldsymbol y$. Thus, equivalently,
to maximize the expression in \eqref{eq:menger2} with respect to
$\boldsymbol y$, we minimize the quadratic function
\begin{equation}
	\label{eq:quadr} \boldsymbol y^{T} \bigl(B\boldsymbol 1\boldsymbol
	1^{T} B-\boldsymbol 1^{T}B \boldsymbol 1B \bigr)\boldsymbol
	y-2\boldsymbol 1^{T} B \boldsymbol y,
\end{equation}
subject to
$\boldsymbol y\geq -\frac {1}{2}\overline{\Gamma}_{\setminus i,i}$. This
is a simple quadratic optimization problem. However, an important complication
comes from the fact that the corresponding quadratic form is not positive
definite (it contains the vector of ones in its kernel), and so many of
the popular quadratic programming algorithms cannot be used. In our calculations
we have used the \texttt{OSQP} package in \texttt{R} (\citet{osqp}).

In summary, our algorithm relies on a sequence of simple quadratic optimization
problems, and it is outlined below. An implementation of this algorithm
is available as the \texttt{emtp2} function of the R package
\texttt{graphicalExtremes} (\citet{graphicalExtremes}).

\begin{algorithm}[b]
	\SetAlgoLined
	\KwData{Conditionally negative definite $\overline{\Gamma}$.}
	\KwResult{A maximizer of (\ref{eq:dual}).}
	Initialize: $\Gamma =\Gamma ^{0}$ (a dually feasible point, see Section~\ref{sec:start})\;
	\While{there is no convergence}{
		\For{$i=1,\ldots ,d$.}{
			Update $\Gamma _{i,\setminus  i} \leftarrow -2\hat {\boldsymbol y}$, where $\hat{\boldsymbol y}$ is the minimizer of \eqref{eq:quadr} subject to $\boldsymbol y\geq -\frac {1}{2}\overline{\Gamma}_{\setminus i,i}$
	}}
	\caption{The block coordinate descent algorithm for H\"{u}sler--Reiss distributions under $ \mathrm{EMTP}_{2} $}
	\label{alg:HR}
\end{algorithm}

The fact that each iteration gives a dually feasible point will be now
proven formally.
\begin{prop}
	\label{prop:duallyfeas}
	Each iteration of Algorithm~\ref{alg:HR} is a dually feasible point.
\end{prop}
\begin{proof}
	Since the starting point $\Gamma ^{0}$ is an arbitrary dually feasible
	point, it is enough to show that updating its $i$th row/column gives a
	dually feasible point. The constraint
	$\Gamma \leq \overline{\Gamma}$ is embedded explicitly in the optimization
	problem, so it is clearly satisfied. To argue that
	$\Gamma \in \mathcal{C}^{d}$ (which is not explicitly imposed), note that
	$\Gamma $ is obtained by maximizing
	\begin{equation}
		\label{eq:maxdet} \det \left( 
		\begin{bmatrix} 0 & -\boldsymbol
			1^{T}
			\\
			\boldsymbol 1 & -\frac {1}{2}\Gamma \end{bmatrix} 
		\right),
	\end{equation}
	which by the Cayley--Menger formula in \eqref{eq:menger} is equal to the
	determinant of $\Sigma ^{(k)}$ for every $k\in [d]$. Suppose that the algorithm
	updates the $i$th row/column of $\Gamma $, and fix any $k\neq i$. By Lemma~\ref{lem:pdC},
	$\Gamma \in \mathcal{C}^{d}$ if and only if $\Sigma ^{(k)}$ is positive
	definite. Using Sylvester's criterion, equivalently,
	$\det (\Sigma ^{(k)}_{B})>0$ for every nonempty
	$B\subseteq [d]\setminus \{i,k\}$ and $B=[d]\setminus \{k\}$ (enough to
	check the leading principal minors when the rows of
	$\Sigma ^{(k)}_{B}$ are arranged so that the $i$th row/column comes last).
	By \eqref{eq:gamma2sigma}, $\Sigma ^{(k)}_{B}$ is an explicit linear function
	of $\Gamma _{B\cup \{k\}}$. Using the Cayley--Menger formula again, we
	get
	\begin{align}
		\label{eq:mengeraux} \det \bigl(\Sigma ^{(k)}_{B}\bigr)& = \det
		\left( 
		\begin{bmatrix} 0 & -\boldsymbol 1^{T}
			\\
			\boldsymbol 1 & -\frac {1}{2}\Gamma _{B\cup \{k\}} \end{bmatrix}
		\right). 
	\end{align}
	If $B\subseteq [d]\setminus \{i,k\}$, then the update of the algorithm
	does not affect this quantity, and so $\det (\Sigma ^{(k)}_{B})>0$ by the
	fact that the current estimate was dually feasible. If
	$B=[d]\setminus \{k\}$, then the right-hand side of
	\eqref{eq:mengeraux} becomes \eqref{eq:maxdet}. This quantity must then
	be strictly positive after the update because it is at least as big as
	for the current estimate, which was strictly positive.
\end{proof}

\subsection{Convergence criteria}
\label{sec6.2}

Recall that, by strong duality, we can guarantee that, at the optimal point
$(\Gamma ^{*},Q^{*})$, the value of the primal and the dual functions are
equal and for any other point the value of the dual problem is lower. Thus,
to obtain a convergence criterion it is natural to track the duality gap
\begin{align}
	\label{dual_gap} -\log \det \Theta ^{(k)}+\llangle \overline{\Gamma},Q
	\rrangle -\bigl(\log \det \Sigma ^{(k)}+(d-1)\bigr) = \llangle \overline{
		\Gamma},Q\rrangle -(d-1), 
\end{align}
which is guaranteed to be always nonnegative and zero precisely at the
optimal point. The algorithm may be stopped when the duality gap is lower
than some fixed threshold. Optimality of the obtained point can be verified
using the KKT conditions in Theorem~\ref{th:kktdd}. Note, however, that,
to compute the duality gap, the current estimate $\Gamma $ needs to be
mapped to $Q$. This operation involves pseudo-inversion, and so it may
be expensive in high-dimensional situations, as the computational complexity
of pseudo-inversion is cubic in dimension. In this case we can simply track
the absolute change between the updates of $\Gamma $, checking the duality
gap only in the end to decide if more iterations are needed.

\subsection{A starting point}
\label{sec:start}

For our coordinate descent algorithm to work, we require a feasible starting
point. By Proposition~\ref{prop:duallyfeas} every subsequent point in our
procedure will be dually feasible. Our construction relies on ideas that
were used in the context of Gaussian distributions. Let $S$ be a positive
semidefinite matrix. By Proposition~3.4 in \cite{LUZ2019}, as long as
$S_{ii}>0$ for all $i=1,\ldots ,d$ and $S_{ij}<\sqrt{S_{ii}S_{jj}}$ for
all $i\neq j$, there exists a positive definite matrix $Z$ such that
$Z\geq S$ and $Z$ coincides with $S$ on the diagonal. The construction
of such $Z$ links to single-linkage clustering and ultrametrics. Section~3
in \cite{LUZ2019} also describes an efficient method for computing
$Z$, which is implemented as function \texttt{Zmatrix} in the R package
\texttt{golazo} (\citet{golazo}). Let $\Gamma ^{Z}$ be obtained from
$Z$ via the inverse covariance mapping. Note that, by construction,
$\Gamma ^{Z}$ is strictly conditionally negative definite and
\begin{equation*}
	\Gamma ^{Z}_{ij}=Z_{ii}+Z_{jj}-2Z_{ij}=S_{ii}+S_{jj}-2Z_{ij}
	\leq S_{ii}+S_{jj}-2S_{ij}= \overline{
		\Gamma}_{ij}.
\end{equation*}
As a consequence, $\Gamma ^{Z}$ is a valid starting point for our block
coordinate descent algorithm.

\subsection{Performance}
\label{sec6.4}

In our setup the optimization of \eqref{eq:emtp2_opt_theta} arises naturally
as the $ \mathrm{EMTP}_{2} $ constraint maximization of the surrogate likelihood
of the H\"usler--Reiss distribution. The same optimization problem appears
in the literature on graph learning under Laplacian constraints
(\citet{EPO17}). While the optimization problem is the same, the way that
the input for the algorithm is obtained differs. In our case we estimate
the combined empirical variogram $\overline \Gamma $ in \eqref{eq:GS} from
samples of the H\"usler--Reiss distribution and derive the matrix
$S$ as in Lemma~\ref{lem:tolapl}. In the graph Laplacian learning literature,
typically, the matrix $S$ is estimated directly from Gaussian data.

We compare our block coordinate descent algorithm, described in Algorithm~\ref{alg:HR},
with existing methods for numerical optimization of
\eqref{eq:emtp2_opt_theta}. The first method by \cite{EPO17} is the combinatorial
graph Laplacian (CGL) algorithm. For the same problem, \cite{zha2019} propose
an alternating direction method of multipliers (ADMM) and a majorization-minimization
(MM) algorithm, whereas \cite{ying2021minimax} use an adaptive Laplacian
constrained precision matrix estimation (ALPE). For the CGL, ADMM and MM
algorithms we use the implementations in the R package
\texttt{spectralGraphTopology} (\citet{spectralGraphTopology}), and for the
ALPE method, we use the code from the R package \texttt{sparseGraph}
(\citet{sparseGraph}). Since the CGL algorithm did note converge in any of
our settings, we do not consider it further.

In order to compare the computation times of the different algorithms and
the corresponding precision of the numerical solution, we conduct the following
study. We first generate a random variogram matrix $\Gamma $ as the Euclidean
distance matrix of $d$ randomly sampled points from the $ (d-1) $-dimensional
unit sphere. For a given tolerance, we run each algorithm with input given
by this matrix $\Gamma $ (or the corresponding matrix $S$). In the first
version of this paper, we observed convergence problems for the ALPE and
ADMM algorithms in this setting. After contacting the authors of
\texttt{spectralGraphTopology} and \texttt{sparseGraph} and reporting our
observations, they kindly provided us with improved versions of their algorithms,
adapted to variograms sampled from the Euclidean distances on the
$ (d-1) $-dimensional ball. Since implementation of the tolerances of the
algorithms are not directly comparable, we repeat this procedure several
times with different variograms and different tolerances each time.

\begin{figure}
	\includegraphics[trim=0em 3em 0em 2em,scale=0.75]{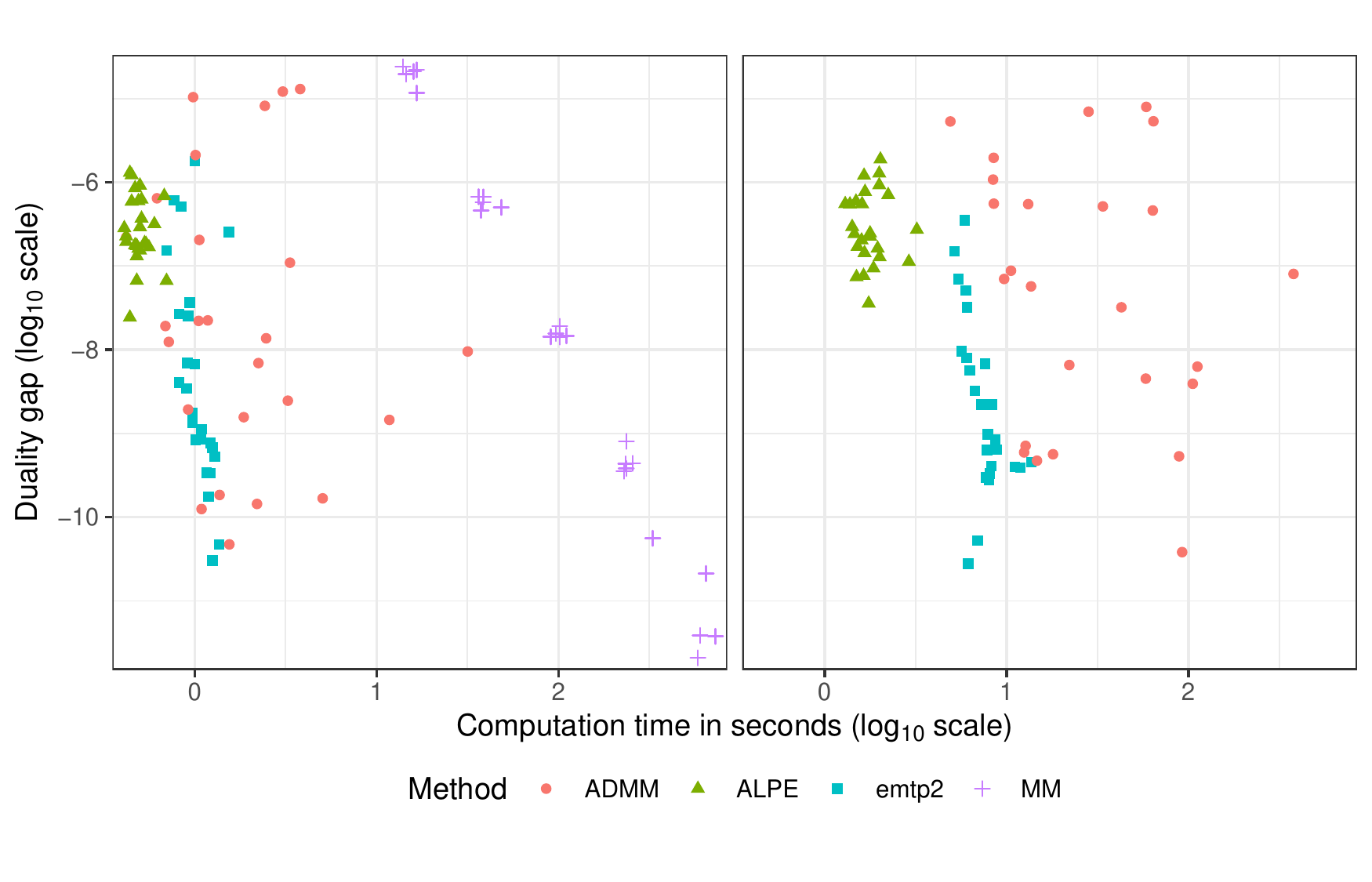}
	\caption{Comparison of the algorithms \texttt{emtp2}, ADMM, MM and ALPE for $ d=50 $ (left) and of the algorithms \texttt{emtp2}, ADMM and ALPE for $ d=100 $ (right) in terms of computation time and duality gap.}\label{fig:performance}
\end{figure}

In the left-hand side of Figure~\ref{fig:performance}, we show the results
for $d=50$ by plotting the duality gap in \eqref{dual_gap} and the corresponding
computation times of the different methods.

We first observe that the adapted ALPE algorithm converges fast, but the
duality gap remains between $ 10^{-6} $ and $ 10^{-8} $, even when we specify
a small tolerance. The ADMM, the MM algorithm and our \texttt{emtp2} algorithm
achieve similar levels of accuracy, but we see that our algorithm is faster
than the ADMM and much faster than the MM algorithm. The right-hand side
of Figure~\ref{fig:performance} shows the same simulation for dimension
$d=100$, where we had to exclude the MM algorithm because of its huge computation
times. Again, we observe that our \texttt{emtp2} algorithm is faster than
the ADMM algorithm and more accurate than the ALPE algorithm.

We further investigate in Table~\ref{fig:comp_times} the computation times
of our $ \mathrm{EMTP}_{2} $ algorithm for a range of dimensions $d$. We observe
that, even for higher dimension, the algorithm can be applied in a reasonable
time. This may be of interest in applications in high-dimensional statistics
where regularization is needed.

All computations in this section were made on a laptop with an Intel Core
i5 processor with $ 1.6\text{ GHz}$. We note that our vanilla implementation
could certainly be largely improved by more efficient programming.

\begin{table}
	\centering
	\begin{tabular}{rrrrr}
		\hline
		$d$ & 50 & 100 & 200 & 400 \\ 
		\hline
		computation time & 1.00 & 6.95 & 70.55 & 910.76 \\ 
		\hline
	\end{tabular}
	\caption{Average computation times (in seconds) of our \texttt{emtp2} algorithm from 10 simulations for a tolerance of $ 10^{-5}$ and different dimensions $d$.}\label{fig:comp_times}
\end{table}

\section{Application}
\label{s:application}

In this section we illustrate the effectiveness of our method by applying
it to the extremes of a data set from the Danube River Basin related to
flood risk assessment. We also discuss the preprocessing of the data prior
to applying our methodology.

\subsection{Data in the domain of attraction}
\label{sec:doa}

While in Section~\ref{s:surrogate} we assumed to have data points directly
from the H\"usler--Reiss distribution $\mathbf{Y}$, in practice, we usually
observe data from a nonextreme random vector $\tilde {\mathbf{X}}$ to which
we apply a preliminary normalization and thresholding step to select the
relevant extremes. Following the theory in Section~\ref{s:MPD_prop}, we
assume that $\tilde {\mathbf{X}}$ has continuous marginal distribution
functions $F_{j}$, $j\in [d]$ and define a normalized random vector
$\mathbf{X}$ with components
\begin{align}
	\label{mar_trans} X_{j} = - \log \bigl\{ 1 - F_{j}(\tilde
	X_{j}) \bigr\}, \quad j \in [d]
\end{align}
with standard exponential margins. We assume that it is multivariate regularly
varying and in the domain of attraction of $\mathbf{Y}$ in the sense of \eqref{MPD_conv}.
For a data matrix $\tilde x \in \mathbb R^{m\times d}$ containing
$m$ observations of $\tilde {\mathbf{X}}$ in the rows, we obtain a data
matrix $x \in \mathbb R^{m\times d}$ by applying the transformation \eqref{mar_trans},
with $F_{j}$ replaced by the empirical distribution functions
$\widehat F_{j}$, to the columns of the matrix $\tilde x$. The rows of
$x$, denoted by $\mathbf x_{i}$, $i\in [m]$, are approximate observations
of $\mathbf{X}$. In a second step, we define the exceedances over some
high threshold $u$ as all observations
\begin{equation*}
	\mathbf y_{i} = \mathbf x_{i} - u\boldsymbol 1 \quad
	\text{for all } i \in \mathcal I = \bigl\{ l \in [m] : \llVert \mathbf
	x_{l} \rrVert _{\infty }> u\bigr\},
\end{equation*}
where the number of exceedances $n= |\mathcal I|$ depends on the threshold
$u$. If $u$ is sufficiently large, by \eqref{MPD_conv} the vectors
$\mathbf y_{i}$, $i\in \mathcal I$, are approximate observations of
$\mathbf{Y}$. We may now follow the steps in Section~\ref{s:surrogate}
to compute the combined empirical variogram $\overline \Gamma $ based on
these data.

Under some regularity conditions, the approximations described above can
be made precise to show that the estimator $\overline \Gamma $ converges
to the true $\Gamma $ if the number of exceedances satisfies
$n \to \infty $ and $n/m \to 0$ (\citet{EV2020}, Theorem~1). The use of
empirical distribution functions for the normalization is standard in multivariate
extreme value theory when the focus is on the dependence structure
(e.g., \citet{ein2009, ein2016}). Similarly to Proposition~\ref{t:consistent},
it then directly follows from the continuity of the
$ \mathrm{EMTP}_{2} $ algorithm, proved in Proposition~\ref{prop:maxth}, that
the $ \mathrm{EMTP}_{2} $ estimator $\widehat \Gamma $ is also consistent
for $\overline \Gamma $ based on data in the domain of attraction of
$\mathbf{Y}$.

\subsection{Danube data}
\label{danube}

For an application that is relevant in terms of risk assessment, we consider
river discharge data from the Upper Danube Basin, which were originally
used in \citet{ADE2015}. The data set consists of daily measurements collected
at $ d=31 $ gauging stations over 50 years from 1960 to 2009 by the Bavarian
Environmental Agency (\href{http://www.gkd.bayern.de}{http://www.gkd.bayern.de}).
After declustering and selecting only the summer months,
\citet{ADE2015} obtain $ m=428 $ observations that are assumed independent.
The Danube data are available in the R package
\texttt{graphicalExtremes} and have been studied in a number of papers
with focus on the modeling of extremal dependence
(\citet{ADE2015,EH2020}) and detecting the extremal causal structure
(\citet{TBK2021,MCD2020,GMPE2021}). For more details on the data and the
preprocessing, we refer to \citet{ADE2015}. We normalize the data as described
in Section~\ref{sec:doa}, and, following \cite{EH2020}, we chose the
$ p=0.9 $ quantile of the marginal Pareto distribution as threshold
$u$, which results in a dataset of $ n=116$ observations.

We begin with an exploratory analysis of the data. From the empirical variogram
$\overline \Gamma $, we obtain an empirical estimate
$\overline \Theta $ of the precision matrix. The respective submatrices
for the stations $I=\{1,2,3,4,5\}$ are
\begin{align*}
	\overline{\Gamma}_{II}&=\begin{pmatrix} 0.00 & 0.53
		& 0.65 & 0.73 & 0.82
		\\
		0.53 & 0.00 & 0.09 & 0.11 & 0.18
		\\
		0.65 & 0.09 & 0.00 & 0.04 & 0.17
		\\
		0.73 & 0.11 & 0.04 & 0.00 & 0.15
		\\
		0.82 & 0.18 & 0.17 & 0.15 & 0.00 \end{pmatrix},\\
	\overline{\Theta}_{II}&=
	\begin{pmatrix} 27.24 & -4.65 & 1.69 & 1.01 & -6.32
		\\
		-4.65 & 50.91 & -18.23 & -8.49 & 0.91
		\\
		1.69 & -18.23 & 62.88 & -39.48 & -3.36
		\\
		1.01 & -8.49 & -39.48 & 61.91 & -19.37
		\\
		-6.32 & 0.91 & -3.36 & -19.37 & 78.67 \end{pmatrix}.
\end{align*}
Considering the full precision matrix, only 250 out of
$30\times 31 /2=465$ free parameters of $\overline \Theta $ are nonpositive,
which at first sight seems not to be in line with the assumption of
$ \mathrm{EMTP}_{2} $. However, in cases where sparsity is present in data,
the true underlying precision matrix $\Theta $ contains many zeros, and
the corresponding empirical estimates fluctuate around zero. Approximately
half of them would, therefore, be positive. If the underlying model is
$ \mathrm{EMTP}_{2} $, then the entries of $\Theta $ corresponding to edges
of the true graph, and likewise, their estimates would be negative. In
practice, we do not know the underlying graph, but in the case of the Danube
data, there is strong evidence that the true graph contains the flow connection
tree (e.g., \citet{EH2020}). In Figure~\ref{fig:boxplot} we, therefore,
show boxplots of the entries of the empirical precision matrix
$\overline \Theta $, grouped by edges that do (left) and do not (right)
belong to the flow connection tree. We can see a clear difference that
supports the intuition above for underlying $ \mathrm{EMTP}_{2} $ models.

\begin{figure}
	\centering
	\includegraphics[trim=0em 8em 0em 5em,scale=0.4]{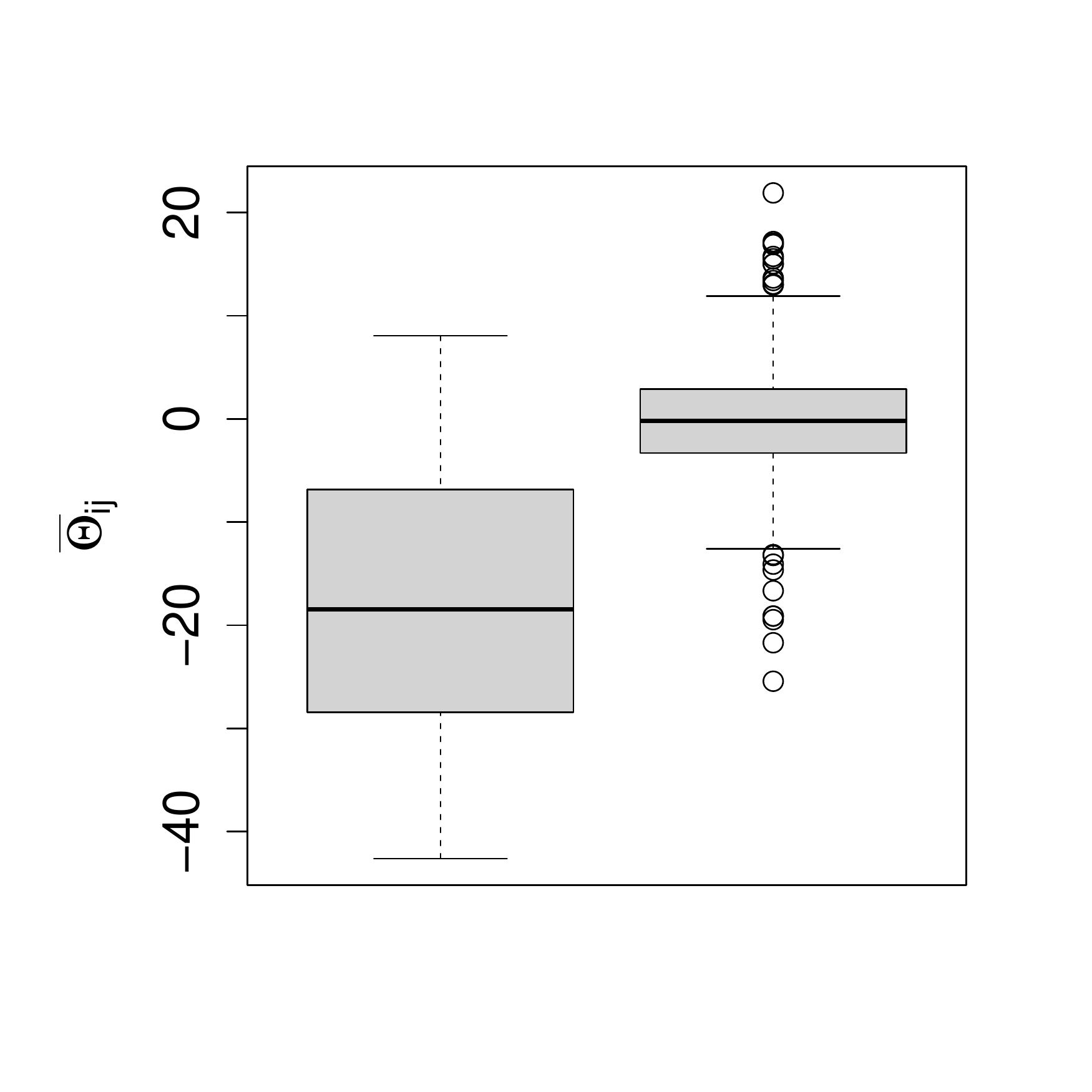}
	\caption{Boxplot for empirical estimates $ \overline{\Theta}_{ij} $ for edges in the flow graph (left) and the non-diagonal remaining entries (right).}\label{fig:boxplot}
\end{figure}

This intuitive reasoning suggests that positive dependence is present in
the Danube data. We, therefore, compute our H\"usler--Reiss estimator under
the $ \mathrm{EMTP}_{2} $ constraint and denote the resulting variogram and
precision matrices by $\widehat \Gamma $ and $\widehat \Theta $, respectively.
To illustrate the regularizing impact of our algorithm, we compare empirical
versions estimates above on the subset $I$ of stations to the corresponding
$ \mathrm{EMTP}_{2} $ estimates
\begin{align*}
	\widehat{\Gamma}_{II}&=\begin{pmatrix} 0.00 & 0.53 &
		\mathbf{0.57} & \mathbf{0.58} & \mathbf{0.60}
		\\
		0.53 & 0.00 & 0.09 & 0.11 & 0.18
		\\
		\mathbf{0.57} & 0.09 & 0.00 & 0.04 & \mathbf{0.17}
		\\
		\mathbf{0.58} & 0.11 & 0.04 & 0.00 & 0.15
		\\
		\mathbf{0.60} & 0.18 & \mathbf{0.17} & 0.15 & 0.00 \end{pmatrix},\\
	& \widehat{
		\Theta}_{II}=\begin{pmatrix} 7.24 & -0.77 & \mathbf{0.00} &
		\mathbf{0.00} & \mathbf{0.00}
		\\
		-0.77 & 14.16 & -8.79 & -0.53 & -2.06
		\\
		\mathbf{0.00} & -8.79 & 32.29 & -23.22 & \mathbf{0.00}
		\\
		\mathbf{0.00} & -0.53 & -23.22 & 29.20 & -3.77
		\\
		\mathbf{0.00} & -2.06 & \mathbf{0.00} & -3.77 & 38.52 \end{pmatrix}.
\end{align*}
In the matrix $ \widehat{\Gamma}_{II} $, we marked in bold the entries
that differ from the empirical version $\overline \Gamma $; note that
$ \widehat{\Gamma}_{35} $ begins to differ only in the third decimal. In
the submatrix of the precision matrix $ \widehat{\Theta}_{II} $, we marked
in bold the entries that have been set to zero by the
$ \mathrm{EMTP}_{2} $ constraint. We observe that these are in correspondence
and that in comparison with $ \overline{\Gamma}_{II} $, only four out
of 10 entries in $ \widehat{\Gamma}_{II} $ have changed. This is a consequence
of the fact that the $ \mathrm{EMTP}_{2} $ solution
$ (\widehat{\Gamma},\widehat{\Theta}) $ must satisfy the KKT conditions
in Theorem~\ref{th:kktdd}. In particular, Condition (iii) of this theorem
imposes zeros in $ \widehat{\Theta} $ exactly where
$ \widehat{\Gamma} $ differs from $ \overline{\Gamma} $. By
\eqref{eq:indep_from_theta} this implies that the H\"usler--Reiss distribution
is an extremal graphical model, as defined in \cite{EH2020}, demonstrating
how $ \mathrm{EMTP}_{2} $ enforces sparsity.
\begin{figure}
	\centering
	\includegraphics[trim=5em 8em 5em 5em,scale=0.49]{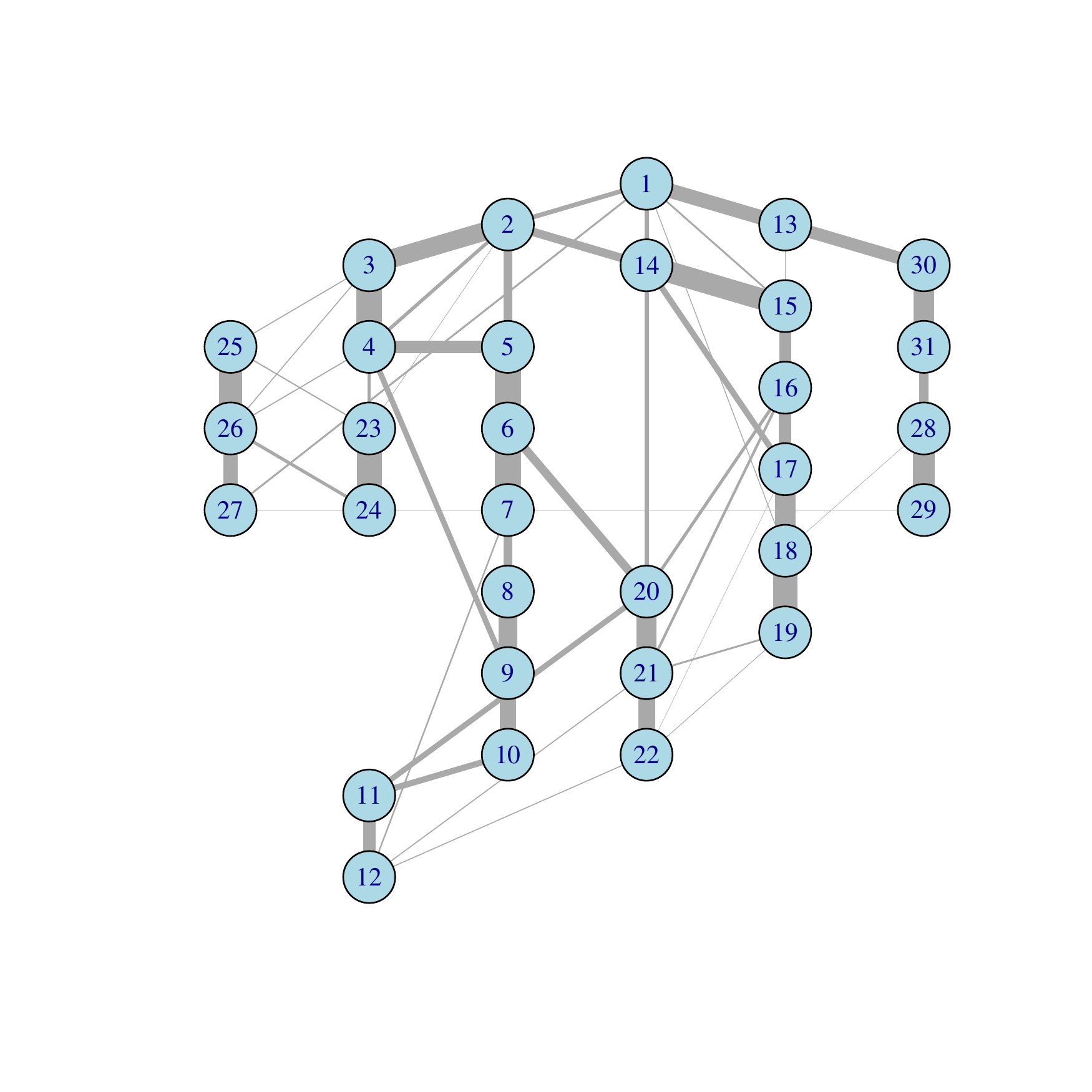}
	\includegraphics[trim=5em 8em 0em 5em,scale=0.49]{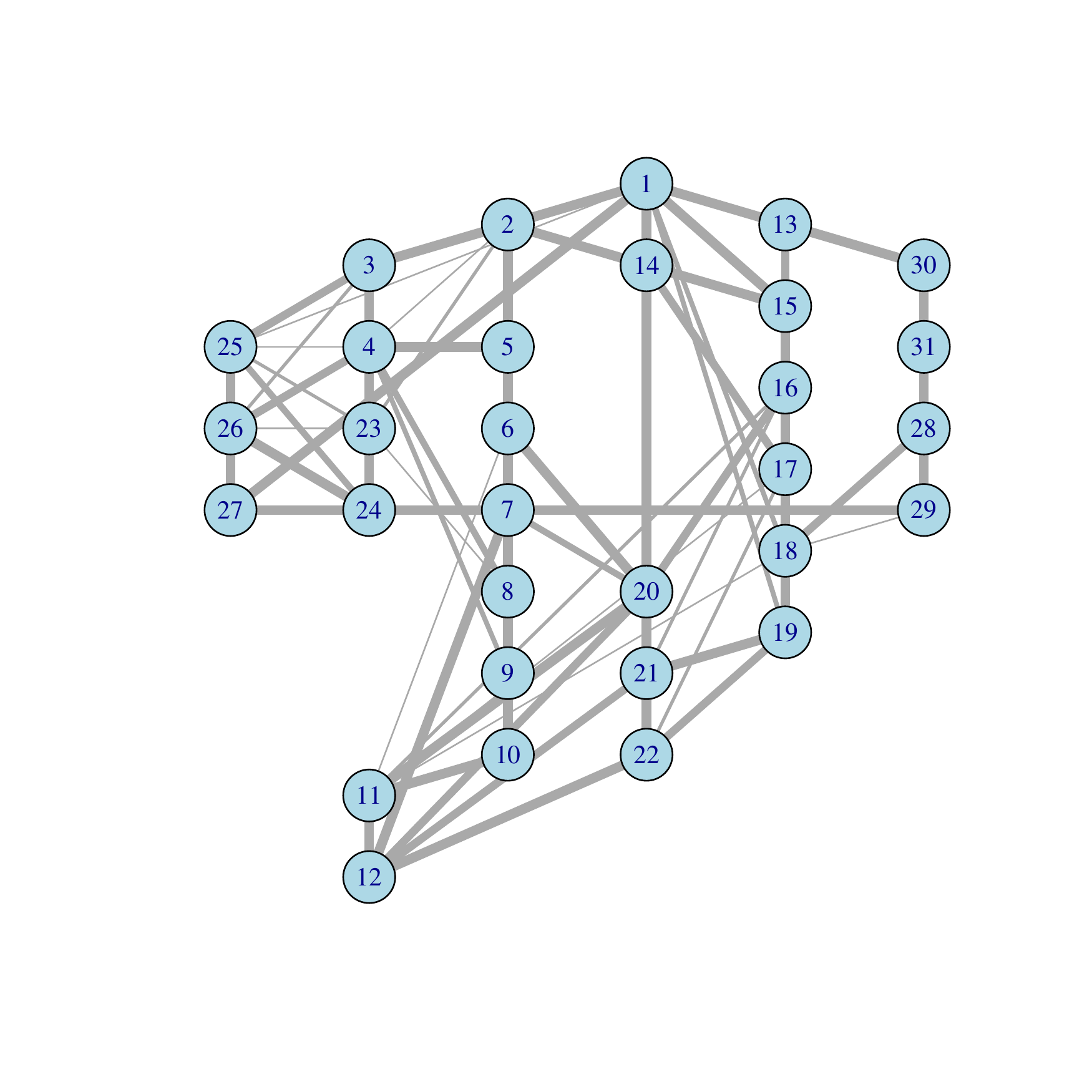}
	\caption{Left: Extremal graphical structure of the fitted \EMTPtwo H\"usler--Reiss distribution; edge thickness is proportional to $ \log(1 - \widehat{\Theta}_{ij}) $. Right: Summary of estimated graphical structures of the fitted \EMTPtwo H\"usler--Reiss distributions for different thresholds $ p\in\{0.7,0.75,0.8,0.85,0.9,0.95\} $;  edge thickness represents the proportion of occurrences of the edge among the estimated graphs.}\label{fig:danube_emtp2}
\end{figure}

The corresponding extremal graph $\widehat G = (V,\widehat E)$ is shown
in the left panel of Figure~\ref{fig:danube_emtp2}. Interestingly, the
$ \mathrm{EMTP}_{2} $ graph contains all physical flow connections, with the
exception of the edges $(25,4)$ and $(20,7)$; see also the geographical
map of the Upper Danube Basin in \citet[Figure~1]{ADE2015}. Most of the
additional connections resemble geographical proximity or similarity, which
may corresponds to positive extremal dependence between such nodes. Theorem~\ref{t:supergraph}
gives us a theoretical foundation to interpret the estimated graph. Indeed,
if the model is $ \mathrm{EMTP}_{2} $, which is a sensible assumption for
the data as argued above, then $\widehat G$ asymptotically contains all
edges that are present in the true underlying graph~$G$.\vspace*{1pt} This means that
if an edge is not present in $\widehat G$, then it cannot be present in
$G$. Since $\widehat G$ is very sparse on this data set, it gives us a
good estimate of the extremal graphical model. In particular, it shows
that many extremal conditional independences are present between locations
that are not neighbors in the flow connection tree.

In order to compare our method to existing approaches, we fit several different
H\"usler--Reiss models to the data, some with graphical structure, and
some without. The first naive approach is to use the combined extremal
variogram $\overline \Gamma $, which corresponds to a trivial, fully connected
graph. As a simple extremal graphical model based on domain knowledge,
we consider a H\"usler--Reiss distribution on the undirected tree given
by the physical flow connection of the river network. As an alternative
tree model, we fit the minimum spanning tree based on
$\overline \Gamma $, which is a consistent estimator of the extremal graph
structure if the true graph is a tree (\citet{EV2020}). As discussed in
\cite{EH2020}, a tree might be too restrictive, and following their methodology,
we fit a sequence of extremal block graph models and choose the best one
according to AIC. \cite{ADE2015} propose a model from spatial extreme value
statistics that heavily relies on domain knowledge of this data set, such
as catchment sizes and distances between stations. We fit their model in
our framework and remark that it has six parameters but corresponds to
a fully connected graph. For the sake of fair comparison, we do not use
censoring in any of the approaches here (cf., \citet{smi1997}). For the
$ \mathrm{EMTP}_{2} $ estimator, censoring could be achieved by using a censored
estimator of $\Gamma $ in the input of Algorithm~\ref{alg:HR}.

The results of the model fits can be found in Table~\ref{fig:danubeAICBIC}.
Among the graphical models, both in terms of AIC and BIC, we observe that
our $ \mathrm{EMTP}_{2} $ performs best. This is remarkable since our method
does not have any tuning parameters and the assumption of
$ \mathrm{EMTP}_{2} $ might seem restrictive. The good performance suggests
that the extremes of this data exhibit strong positive dependence, which
underlines the theoretical findings of this paper. The spatial model of
\cite{ADE2015} performs similarly to our $ \mathrm{EMTP}_{2} $ estimator in
terms of AIC and in terms of BIC, which penalizes the number of model parameters
more strongly; the spatial model is first. We note that this comparison
is flawed since, as opposed to the spatial model, our estimator is completely
data-driven and does not use any domain knowledge. It can, therefore, easily
be applied to general multivariate data where no information of the gauging
stations is available.

\begin{table}
	\centering
	\begin{tabular}{rrrrr}
		\hline
		& twice neg logLH & nb par & AIC & BIC \\ 
		\hline
		empirical variogram & 253.17 & 465 & 1183.17 & 2467.58 \\ 
		flow graph & 1447.85 & 30 & 1507.85 & 1590.72 \\ 
		MST & 1372.58 & 30 & 1432.58 & 1515.45 \\ 
		best block graph MST & 1246.11 & 42 & 1330.11 & 1446.12 \\ 
		Asadi et al. & 1090.35 & 6 & 1102.35 & 1118.92 \\ 
		\EMTPtwo estimator & 1017.00 & 67 & 1151.00 & 1336.07 \\ 
		\hline
	\end{tabular}
	\caption{Results for the different models fitted to the Danube river data set; columns show twice the negative log-likelihood, the number of model parameters and the AIC and BIC values, respectively.}\label{fig:danubeAICBIC}
\end{table}

For a sensitivity analysis with respect to the chosen threshold
$ p $, we study the estimated $ \mathrm{EMTP}_{2} $ graphs for
$ p\in \{0.7,0.75,0.8,0.85,0.9,0.95\} $. The right panel of Figure~\ref{fig:danube_emtp2}
shows a summary by a graph with edge width representing the proportion
of appearances of the edge among all graphs. We observe that most edges
appear in every graph so that the $ \mathrm{EMTP}_{2} $ graph seems to be
stable across different threshold choices.

To summarize, the $ \mathrm{EMTP}_{2} $ estimator allows for a competitive
fit without the choice of tuning parameters and without the need of domain
knowledge. In particular, for high-dimensional applications with potentially
small sample sizes, our estimator is guaranteed to exist and our algorithm
is computationally fast, even for large dimension.

\section{Discussion}
\label{sec:discussion}

In this paper we have studied the role of positive dependence in multivariate
extreme value theory. In particular, the property of
$ \mathrm{EMTP}_{2} $ appears naturally in many statistical models and can
be characterized by Laplacian precision matrices in the important case
of H\"usler--Reiss distributions.

We concentrate on the case of multivariate Pareto distributions, which
describe the multivariate tail under asymptotic dependence; see Section~\ref{s:MPD_prop}.
Our theoretical results rely on the fundamental Theorem~\ref{t:LLCetal}
on the positive dependence of convolutions of a random vectors. In this
paper we mainly used this result to link the probabilistic properties of
a multivariate Pareto distribution to those of its extremal functions;
see Theorem~\ref{t:mtp2}.

Since the assumptions of Theorem~\ref{t:LLCetal} are fairly general, it
can be applied to a much wider range of models of the form
\begin{align}
	\label{random_location} \widetilde{\mathbf Z} = \mathbf X + X_{0}\boldsymbol
	1, 
\end{align}
where $X_{0}$ is a general random variable, also called common factor,
and independent of this, $\mathbf{X}$ is a multivariate random vector.
Such models appear as factor models
(\citet{lawley1962factor,holland1986conditional}) or random location mixtures
in applied probability (\citet{Hashorva2012, kru2018}).

In the framework of extremes, such location mixtures have been proven to
produce versatile tail dependence structures, including both asymptotic
dependence and independence (e.g., \citet{eng2018a}). Intuitively, the
heavier the tail of the common factor $X_{0}$ relative to the tail heaviness
of the components of $\mathbf{X}$, the stronger the dependence of
$\widetilde{\mathbf Z}$ in the extremes.

A future research direction is to extend the theoretical analysis and statistical
methodology of our paper to models for asymptotic independence. Our Theorem~\ref{t:LLCetal}
can be applied to obtain first results. Indeed, as an example, if
$X_{0}$ has a light tail and $\mathbf{X}$ is multivariate Gaussian, then
$\widetilde{\mathbf Z}$ is asymptotically independent (\citet{kru2018}),
and the strength of residual dependence depends on the correlation matrix
of $\mathbf{X}$ (\citet{eng2018a}). If $\mathbf{X}$ is strongly
$ \mathrm{MTP}_{2} $, that is, its precision matrix is a diagonally dominant
M-matrix (see Example~\ref{ex:gauss}), then Theorem~\ref{t:LLCetal} implies
that $\widetilde{\mathbf Z}$ is $ \mathrm{MTP}_{2} $. This theoretical result
could be used to regularize such asymptotically independent models by enforcing
the $ \mathrm{MTP}_{2} $ constraint.

Another popular approach for asymptotic independence is the model of
\cite{HeffernanTawn2004}. Similar to the definition of
$\mathbf{Y}^{k}$ in \eqref{Yk}, this model specifies the multivariate distribution
conditional on one variable being extreme. In this case, even though the
corresponding model has a form similar to \eqref{random_location}, there
is dependence between $X_{0}$ and $\mathbf{X}$. A~different result is,
therefore, needed to characterize positive dependence in these models.

\begin{appendix}
	\section{The algebra of variogram matrices}
	\label{app:Gamma}
	
	\subsection{Algebraic structure}
	\label{secA.1}
	
	Let $\mathbb{S}^{d}$ be the space of real symmetric $d\times d$ matrices
	and $\mathbb{S}^{d}_{0}$ its subspace with zeros on the diagonal. We equip
	$\mathbb{S}^{d}$ with the standard trace inner product
	$\langle A,B\rangle =\operatorname{tr}(AB)=\sum_{i,j}A_{ij}B_{ij}$ and
	$\mathbb{S}^{d}_{0}$ with
	$\llangle A,B\rrangle  =\sum_{i<j}A_{ij}B_{ij}$. For
	$\boldsymbol b\in \mathbb{R}^{d}$ satisfying
	$\boldsymbol b^{T}\boldsymbol 1=1$, we define:
	\begin{itemize}
		\item [(i)] linear subspace of $\mathbb{S}^{d}$:
		$\mathbb U_{\boldsymbol b}=\{A\in \mathbb{S}^{d}:A \boldsymbol b=
		\boldsymbol 0\}$,
		\item [(ii)] projection on $\mathbb{R}^{d}/\boldsymbol 1$:
		$\boldsymbol P_{\boldsymbol b}=I_{d}-\boldsymbol 1\boldsymbol b^{T}$,
		\item [(iii)] linear map:
		$\sigma _{\boldsymbol b}: \mathbb{S}^{d}_{0}\to \mathbb U_{
			\boldsymbol b}$,
		$A\mapsto \boldsymbol P_{\boldsymbol b}(-\frac {A}{2})\boldsymbol P_{
			\boldsymbol b}^{T}$.
	\end{itemize}
	It is useful to note that, for any $\boldsymbol a$, $\boldsymbol b$ such that
	$\boldsymbol a^{T}\boldsymbol 1=\boldsymbol b^{T}\boldsymbol 1=1$,
	\begin{equation}
		\label{eq:Pab} \boldsymbol P_{\boldsymbol a}\boldsymbol P_{\boldsymbol b} =
		\boldsymbol P_{\boldsymbol a}. 
	\end{equation}
	
	The relevant cases for us are when
	$\boldsymbol b=\frac {1}{d}\boldsymbol 1$ and when
	$\boldsymbol b=\boldsymbol e_{k}$ is a canonical unit vector. If
	$\boldsymbol b=\frac {1}{d}\boldsymbol 1$, we omit the subscript writing
	$\mathbb U$ and $\boldsymbol P$. In the special case when
	$\boldsymbol b=\boldsymbol e_{k}$, we write $\boldsymbol P_{k}$ and
	$\mathbb U_{k}$. Note that $\boldsymbol P$ is symmetric and it represents
	the \emph{orthogonal} projection matrix on
	$\mathbb{R}^{d}/\boldsymbol 1$. Also, $\boldsymbol P_{k}$ has rows
	$\boldsymbol e_{i}-\boldsymbol e_{k}$, and in particular, the $k$th row
	is zero. We denote by
	$\overline {\boldsymbol P}_{k}\in \mathbb{R}^{(d-1)\times d}$ the matrix
	obtained from $\boldsymbol P_{k}$ by removing the $k$th row.
	
	Let $\mathbb{S}^{d}_{0}[\Gamma ]$, $\mathbb{S}^{d}_{0}[Q]$ be two copies
	of $\mathbb{S}^{d}_{0}$ with coordinates denoted by $\Gamma _{ij}$ and
	$Q_{ij}$, respectively. Similarly, we let
	$\mathbb U_{\boldsymbol b}[\Sigma ]$,
	$\mathbb U_{\boldsymbol b}[\Theta ]$ be two copies of
	$\mathbb U_{\boldsymbol b}$. Consider the following sequence of transformations:
	\begin{equation*}
		\mathbb{S}^{d}_{0}[\Gamma ]  \overset{\sigma
			_{\boldsymbol b}} {\longrightarrow} \mathbb U_{
			\boldsymbol b}[\Sigma ]
		\overset{\operatorname{ginv}} {\longrightarrow}\mathbb U_{\boldsymbol b}[
		\Theta ] \overset{\sigma _{\boldsymbol b}^{*}} {\longrightarrow}
		\mathbb{S}^{d}_{0}[Q],
	\end{equation*}
	where $\operatorname{ginv}$ stands for the generalized inverse
	$\Sigma \mapsto \Sigma ^{+}$, $\sigma _{\boldsymbol b}$ is a linear map
	defined by
	$\sigma _{\boldsymbol b}(\Gamma )=\boldsymbol P_{\boldsymbol b}(-
	\frac{1}{2}\Gamma )\boldsymbol P_{\boldsymbol b}^{T}$ and
	$\sigma _{\boldsymbol b}^{*}$ denotes the adjoint of the linear map
	$\sigma _{\boldsymbol b}$, that is, the unique transformation that satisfies
	\begin{equation}
		\label{eq:adjoint} \bigl\langle \sigma _{\boldsymbol b}(\Gamma ),\Theta \bigr
		\rangle = \bigllangle \Gamma ,\sigma _{\boldsymbol b}^{*}(\Theta )\bigrrangle
		\quad \text{for all }\Gamma \in \mathbb{S}^{d}_{0}, \Theta
		\in \mathbb U_{\boldsymbol b}. 
	\end{equation}
	\begin{remark}
		\label{rem:app}
		We note that:
		\begin{itemize}
			\item[1.] The map $\sigma _{\boldsymbol b}$ is invertible with the inverse
			$\gamma _{\boldsymbol b}:\mathbb U_{\boldsymbol b}\to \mathbb{S}^{d}_{0}$,
			given by
			\begin{equation*}
				\gamma _{\boldsymbol b}(\Sigma )=(\Sigma _{ii}+\Sigma
				_{jj}-2 \Sigma _{ij})_{i<j}.
			\end{equation*}
			\item[2.] Standard linear algebra gives that the adjoint of the inverse
			$\gamma _{\boldsymbol b}$ is equal to the inverse of the adjoint
			$\sigma _{\boldsymbol b}^{*}$.
			\item[3.] The generalized inverse is a well-defined automorphism on
			$\mathbb U_{\boldsymbol b}$.
			\item[4.] The inner products are preserved in the sense that, for every
			$\Gamma ,Q\in \mathbb{S}^{d}_{0}$ with
			$\Sigma =\sigma _{\boldsymbol b}(\Gamma )$ and
			$\Theta =\gamma ^{*}_{\boldsymbol b}(Q)$, we have that
			\begin{equation*}
				\langle \Sigma ,\Theta \rangle = \bigl\langle \sigma _{\boldsymbol b}( \Gamma ),
				\gamma _{\boldsymbol b}^{*}(Q)\bigr\rangle = \bigllangle \gamma
				_{\boldsymbol b}\bigl(\sigma _{\boldsymbol b}(\Gamma )\bigr),Q \bigrrangle =
				\llangle \Gamma ,Q\rrangle .
			\end{equation*}
		\end{itemize}
	\end{remark}
	
	The adjoint map can be easily computed, and its form is particularly simple
	in the special case when $\boldsymbol b=\frac{1}{d}\boldsymbol 1$ and when
	$\boldsymbol b=\boldsymbol e_{k}$. To avoid confusion, we introduce different
	notation for the coordinates of $\mathbb{S}^{d}_{0}$ and
	$\mathbb U_{\boldsymbol b}$ depending on a particular
	$\boldsymbol b$. We use:
	\begin{itemize}
		\item [(i)] $\Sigma $, $\Theta $ to denote coordinates in $\mathbb U$,
		\item [(ii)] $\widetilde\Sigma ^{(k)}$, $\widetilde\Theta ^{(k)}$ to denote
		coordinates in $\mathbb U_{k}$ and
		\item [(iii)] $\Sigma ^{(k)}$, $\Theta ^{(k)}$ to denote matrices in
		$\mathbb{S}^{d-1}$ obtained from
		$\widetilde\Sigma ^{(k)}$, $\widetilde\Theta ^{(k)}$ by removing the
		$k$th row/column.
	\end{itemize}
	
	\begin{lemma}
		\label{lem:QinTheta}
		The adjoint map
		$\sigma ^{*}:\mathbb U[{\Theta}]\to \mathbb{S}^{d}_{0}[{Q}]$ is given by
		$Q_{ij}=-\Theta _{ij}$ for all $1\leq i<j\leq d$. The adjoint map
		$\sigma ^{*}_{k}:\mathbb U_{k}[{\widetilde\Theta ^{(k)}}]\to
		\mathbb{S}^{d}_{0}[{Q}]$ is given by $Q_{ij}=-\Theta _{ij}^{(k)}$ for all
		$i,j\neq k$ and $Q_{ik}=\sum_{j\neq k} \Theta _{ij}^{(k)}$.
	\end{lemma}
	\begin{proof}
		The adjoint maps are defined by (\ref{eq:adjoint}). We will check this
		condition on the basis of $\mathbb{S}^{d}_{0}$ given by elements of the
		form $B=E_{ij}+E_{ji}$, where $E_{ij}$ denotes the elementary matrix with
		the $ij$th entry equal to one and zero otherwise. Let first
		$\boldsymbol b=\frac{1}{d}\boldsymbol 1$. The right-hand side of (\ref{eq:adjoint})
		becomes $(\sigma ^{*}(\Theta ))_{ij}$. The left-hand side is
		\begin{equation*}
			\biggl\langle \boldsymbol P \biggl(-\frac {B}{2}\biggr)\boldsymbol P,
			\Theta \biggr\rangle = -\frac {1}{2}\operatorname{tr}(B\boldsymbol P\Theta
			\boldsymbol P) = - \frac {1}{2}\operatorname{tr}(B\Theta ) = -\Theta
			_{ij},
		\end{equation*}
		where we used the fact that
		$\boldsymbol P \Theta \boldsymbol P=\Theta $ for all
		$\Theta \in \mathbb U$. The second part of the result follows similar calculations
		and the fact that
		$(\boldsymbol P_{k}^{T}\widetilde\Theta ^{(k)}\boldsymbol P_{k})_{ij}=
		\Theta _{ij}^{(k)}$ if $i,j\neq k$ and
		$(\boldsymbol P_{k}^{T}\widetilde\Theta ^{(k)}\boldsymbol P_{k})_{ik}=-
		\sum_{j\neq k}\Theta _{ij}^{(k)}$.
	\end{proof}
	
	In our paper we start with the variogram matrix
	$\Gamma \in \mathbb{S}^{d}_{0}$. The matrix $\Sigma ^{(k)}$ defined in
	\eqref{eq:gamma2sigma} is a $(d-1)\times (d-1)$ matrix obtained from
	$\sigma _{k}(\Gamma )\in \mathbb U_{k}$ by removing the $k$th row/column.
	The inverse of $\sigma _{k}$ expresses $\Gamma $ in terms of
	$\Sigma ^{(k)}$ as in \eqref{eq:sigma2gamma}. The matrix
	$\Theta =\Sigma ^{+}=(\boldsymbol P(-\frac{1}{2}\Gamma )\boldsymbol P)^{+}$
	is exactly the same matrix that appears in Proposition~\ref{prop:manu}.
	To easily translate between various equivalent representations of the variogram
	matrix $\Gamma $, we define $f_{k}:\mathbb U\to \mathbb U_{k}$ by
	$\Sigma \mapsto \boldsymbol P_{k}\Sigma \boldsymbol P_{k}^{T}$. The following
	results provides the complete picture of the situation.
	\begin{prop}
		\label{prop:cd}
		The adjoint $f^{*}_{k}:\mathbb U_{k}\to \mathbb U$ of $f_{k}$ is defined
		by $\Theta \mapsto \boldsymbol P_{k}^{T}\Theta \boldsymbol P_{k}$. Moreover,
		the following diagram commutes\footnote{By this we mean that composing maps
			along any two directed paths with the same beginning and end results in
			the same function.}
			\[\begin{tikzcd}
		& \S^{d-1}[\Sigma^{(k)}] \arrow[r,"{\rm inv}"] & \S^{d-1}[\Theta^{(k)}] \arrow[d,swap,"\pi^*_{k}"] &\\
		\S^d_0[\Gamma] \arrow[r,"\sigma_k"]  \arrow[dr,swap,"\sigma"] &\quad \mathbb U_k[\widetilde \Sigma^{(k)}]  \arrow[u,"\pi_k"] \arrow[r,"{\rm ginv}"]  &\quad \mathbb U_k[\widetilde \Theta^{(k)}]\hspace{-2mm} \arrow[d,swap,"f^*_k", shift right = 2.5mm]  \arrow[r,"\sigma^*_k"] &\quad \S^d_0[Q]\\
		&\quad \mathbb U[\Sigma]\hspace{-2mm} \arrow[r,"{\rm ginv}"] \arrow[u,swap,"f_k",shift left = 2.6mm] & \quad~~\mathbb U[\Theta]\hspace{2mm} \arrow[ru,swap,"\sigma^*"]& 
	\end{tikzcd},
	\]
		where $\pi _{k}$ drops the $k$th row/column of
		$\widetilde \Sigma ^{(k)}$ and its adjoint $\pi _{k}^{*}$ embeds
		$\Theta ^{(k)}$ in $\mathbb U_{k}$ by adding the zero row/column. All the
		maps, apart from the inversion on the top, are well defined everywhere.
		For the inversion we restrict the map to an open subset where
		$\Sigma ^{(k)}$ is invertible.
	\end{prop}
	\begin{proof}
		To verify the formula for the adjoint $f^{*}_{k}$, we note that, by definition,
		it must satisfy
			\begin{equation*}
			\bigl\langle f_{k}(\Sigma ),\widetilde\Theta ^{(k)}\bigr
			\rangle =\bigl\langle \Sigma ,f^{*}_{k}\bigl(\widetilde\Theta
			^{(k)}\bigr)\bigr\rangle \quad \text{for all }\Sigma \in \mathbb U,
			\widetilde\Theta ^{(k)}\in \mathbb U_{k},
		\end{equation*}
		and the formula follows by basic properties of the matrix trace and the
		fact that
		$\boldsymbol P_{\boldsymbol b}^{T}\widetilde\Theta ^{(k)}
		\boldsymbol P_{\boldsymbol b}\in \mathbb U$.
		
		To verify that the diagram commutes, it is enough to check that that it
		commutes along two side triangles, that is, that
		$\sigma _{k}=f_{k}\sigma $ and
		$\sigma _{k}^{*}=\sigma ^{*}f_{k}^{*}$ and along two central rectangles.
		The bottom rectangle follows by the above calculations. The upper rectangle
		follows by how pseudoinverse works on the space $\mathbb U_{k}$. To see
		that $\sigma _{k}=f_{k}\sigma $, note that, by (\ref{eq:Pab}),
		\begin{equation*}
			f_{k}\bigl(\sigma (\Gamma )\bigr)=\boldsymbol P_{k}
			\boldsymbol P\biggl(- \frac {\Gamma}{2}\biggr)\boldsymbol P \boldsymbol
			P^{T}_{k}=\boldsymbol P_{k}\biggl(-
			\frac {\Gamma}{2}\biggr)\boldsymbol P_{k}^{T}=\sigma
			_{k}(\Gamma ).
		\end{equation*}
		To check that $\sigma _{k}^{*}=\sigma ^{*}f_{k}^{*}$, we use basic properties
		of the adjoint.\vadjust{\goodbreak}
	\end{proof}
	
	Proposition~\ref{prop:cd} gives us another way to verify formula
	\eqref{eq:Theta}.
	
	\begin{lemma}
		\label{lem:Theta}
		Fix $k\in \{1,\ldots ,d\}$ then the form of the mapping $f^{*}_{k}$ implies
		that
		\begin{equation*}
			\Theta _{ij}= 
			\begin{cases} \Theta ^{(k)}_{ij}
				& \text{if }i,j\neq k,
				\\[6pt]
				-\sum_{l\neq k}\Theta ^{(k)}_{il} &
				\text{if } i\neq k, j=k,
				\\[6pt]
				\sum_{i,j\neq k}\Theta ^{(k)}_{ij} &
				\text{if } i=j=k. \end{cases} 
		\end{equation*}
	\end{lemma}
	
	By Lemma~\ref{lem:QinTheta}, $\Theta \in \mathbb U$ is a weighted Laplacian
	matrix with potentially negative weights $Q_{ij}$ for $i\neq j$. The weighted
	matrix-tree theorem used in \eqref{eq:det_spanning_trees} will be useful
	for the next result.
	
	\begin{prop}
		\label{prop:Gamma2Q}
		The mapping $\mathbb{S}^{d}_{0}[\Gamma ]\to \mathbb{S}^{d}_{0}[Q]$, given
		by $\Gamma \mapsto \sigma ^{*}_{k}((\sigma _{k}(\Gamma ))^{+})$, is compactly
		written as
		\begin{equation*}
			Q = \nabla \log\operatorname{Det}\bigl(\sigma _{k}(\Gamma )\bigr).
		\end{equation*}
		Similarly,
		\begin{equation*}
			\Gamma = \nabla \log\operatorname{Det}\bigl(\gamma ^{*}_{k}(Q)
			\bigr).
		\end{equation*}
	\end{prop}
	\begin{proof}
		Let $\overline{\Gamma}$ be a given point in
		$\mathbb{S}^{d}_{0}[\Gamma ]$, and let
		$\widetilde S^{(k)}=\sigma _{k}(\overline{\Gamma})\in \mathbb U_{k}$. Similarly,
		let $Q^{S}=\sigma ^{*}_{k}((\sigma _{k}(\overline{\Gamma}))^{+})$. By
		\cite{holbrook2018differentiating},
		\begin{equation*}
			\nabla _{\Sigma }\bigl(\log\operatorname{Det}(\Sigma )-\bigl\langle \Sigma
			,\bigl({ \widetilde S}^{(k)}\bigr)^{+}\bigr\rangle
			\bigr)\big|_{\Sigma ={\widetilde S}^{(k)}}=0.
		\end{equation*}
		Using the fact that
		$\sigma _{k}:\mathbb{S}^{d}_{0}\to \mathbb U_{k}$ is an invertible linear
		mapping, we get that, equivalently,
		\begin{equation*}
			\nabla _{\Gamma }\bigl(\log\operatorname{Det}\bigl(\sigma _{k}(
			\Gamma )\bigr)-\bigl\langle \sigma _{k}( \Gamma ),\bigl({\widetilde
				S}^{(k)}\bigr)^{+}\bigr\rangle \bigr)\big|_{\Gamma =
				\overline{\Gamma}}=0.
		\end{equation*}
		Since
		$\langle \sigma (\Gamma ),({\widetilde S}^{(k)})^{+}\rangle =\llangle \Gamma ,\sigma ^{*}(({\widetilde S}^{(k)})^{+})\rrangle =\llangle \Gamma ,Q^{S}\rrangle  $, we obtain
		the desired formula.
	\end{proof}
	
	\subsection{Strictly conditionally negative matrices}
	\label{s:EDM}
	
	The variogram matrices are not only assumed to lie in
	$\mathbb{S}^{d}_{0}$ but they are also assumed to be strictly conditionally
	negative definite. We study this additional constraint a bit more in this
	section. Using the notation from Section~\ref{sec:huesler_reiss},
	$\Gamma \in \mathcal C^{d}\subset \mathbb{S}^{d}_{0}$. In this section
	we briefly list relevant results that follow from assuming this extra structure.
	
	\begin{remark}
		\label{rem:edm}
		If $\Gamma $ is a conditionally negative definite matrix, then, by the
		theorem of Schoenberg (\citet{gower,schoenberg}), equivalently, there exist
		vectors $\boldsymbol y_{1},\ldots ,\boldsymbol y_{d}$ in some Euclidean
		space $\mathbb{R}^{p}$ such that
		$\Gamma _{ij}=\|\boldsymbol y_{i}-\boldsymbol y_{j}\|^{2}$. We also call
		such $\Gamma $ a Euclidean distance matrix.
	\end{remark}
	
	\begin{lemma}
		\label{obs:positiveGam}
		If $\Gamma \in \mathcal C^{d}$, then $\Gamma _{ij}>0$ for all
		$i\neq j$.
	\end{lemma}
	\begin{proof}
		Take $\boldsymbol x=\boldsymbol e_{i}-\boldsymbol e_{j}$. By definition,
		$\boldsymbol x^{T}\Gamma \boldsymbol x=-2\Gamma _{ij}$ must be strictly
		negative.
	\end{proof}
	\begin{lemma}
		\label{lem:pdC}
		The cone $\sigma _{\boldsymbol b}(\mathcal C^{d})$ is precisely the set
		of all positive semidefinite matrices in $\mathbb U_{\boldsymbol b}$ of
		the rank $d-1$.\vspace*{1pt} The mapping
		$\pi _{k}:\mathbb U_{k}\to \mathbb{S}^{d-1}$ maps
		$\sigma _{k}(\mathcal C^{d})$ to the positive definite cone
		$\mathbb{S}^{d-1}_{+}$.
	\end{lemma}
	\begin{proof}
		Let $\Gamma \in \mathcal C^{d}$. Since
		$\sigma _{\boldsymbol b}(\Gamma )$ is symmetric, it has real eigenvalues,
		and the eigenvectors are mutually orthogonal. It is clear that
		$\boldsymbol b^{T} \sigma _{\boldsymbol b}(\Gamma )\boldsymbol b=0$, so
		$\boldsymbol b$ is an eigenvector with eigenvalue $0$. We will show that
		all the other eigenvalues must be strictly positive. If
		$\boldsymbol b=\pm \frac{1}{d}\boldsymbol 1$, then
		$\boldsymbol x\perp \boldsymbol b$, and
		$\boldsymbol x\neq \boldsymbol 0$ implies that
		\begin{equation*}
			\boldsymbol x^{T}\sigma (\Gamma ) \boldsymbol x=\boldsymbol
			x^{T} \biggl(- \frac{1}{2}\Gamma \biggr)\boldsymbol x>0
		\end{equation*}
		by the fact that $\Gamma $ is strictly conditionally negative definite.
		This implies that all the remaining eigenvalues of
		$\sigma (\Gamma )$ must be strictly positive. Suppose now that
		$\boldsymbol b\neq \pm \frac{1}{d}\boldsymbol 1$. By Proposition~\ref{prop:cd},
		$\sigma _{\boldsymbol b}(\Gamma )=\boldsymbol P_{\boldsymbol b}
		\sigma (\Gamma )\boldsymbol P_{\boldsymbol b}^{T}$. Since
		$\boldsymbol P_{\boldsymbol b}^{T}\boldsymbol x=\boldsymbol x-(
		\boldsymbol 1^{T}\boldsymbol x)\boldsymbol b$,
		$\boldsymbol P_{\boldsymbol b}^{T}\boldsymbol x\perp \boldsymbol 1$, and
		it follows that
		\begin{equation}
			\label{eq:auxxpx} \boldsymbol x^{T} \sigma _{\boldsymbol b}(\Gamma )
			\boldsymbol x=\bigl( \boldsymbol x-\bigl(\boldsymbol 1^{T}\boldsymbol x
			\bigr)\boldsymbol b\bigr)^{T} \sigma (\Gamma ) \bigl(\boldsymbol x-
			\bigl(\boldsymbol 1^{T}\boldsymbol x\bigr) \boldsymbol b\bigr)\geq 0
		\end{equation}
		with strict inequality if
		$\boldsymbol x-(\boldsymbol 1^{T}\boldsymbol x)\boldsymbol b\neq
		\boldsymbol 0$. However, $\boldsymbol x\perp \boldsymbol b$ and
		$\boldsymbol b^{T}\boldsymbol 1=1$ imply that $\boldsymbol x$ is not parallel
		to $\boldsymbol 1$, and so we must have strict positivity in
		\eqref{eq:auxxpx}.
	\end{proof}
	
	\subsection{Variograms in the positive case}
	\label{sec:povar}
	
	In our study of total positivity for extremes, we showed in Theorem~\ref{t:equivalences_precision_matrices_emtp2}
	that a particularly important case is when $\Theta _{ij}\leq 0$ for all
	$i\neq j$. This is the case when $\Theta $ corresponds to a Laplacian matrix
	on a connected graph with positive weights on each edge. In this case
	$Q_{ij}=-\Theta _{ij}\geq 0$ for all $i\neq j$.
	
	Recall from Remark~\ref{rem:edm} and Lemma~\ref{obs:positiveGam} that if
	$\Gamma \in \mathcal C^{d}$, then $\Gamma $ is always a distance matrix
	in the sense that $\sqrt{\Gamma _{ij}}$ are distances between a finite
	collection distinct points $\boldsymbol y_{i}$ for $i=1,\ldots ,d$ that
	lie in some Euclidean space $\mathbb{R}^{p}$. Let
	$U\in \mathbb{R}^{p\times d}$ be a matrix whose columns are
	$\boldsymbol y_{1},\ldots ,\boldsymbol y_{d}$. Note that translating all
	points $\boldsymbol y_{i}$ will not change mutual distances. We consider
	two important cases:
	
	Case~1: We translate the points by their average so that now
	$\sum_{i} \boldsymbol y_{i}=\boldsymbol 0$. In this case the Gram matrix
	$U^{T} U$ lies in $\mathbb U$, and in fact, it is equal to
	$\Sigma =\sigma (\Gamma )$.

	Case~2: We translate the points to move one of the points to the origin
	so that now $\boldsymbol y_{k}=\boldsymbol 0$ for some
	$k=1,\ldots ,d$. In this case $U^{T} U \in \mathbb U_{k} $, and in fact,
	it is equal to $\tilde\Sigma ^{(k)}=\sigma _{k}(\Gamma )$.

	In what follows, we outline some of interesting results of Miroslav Fiedler;
	see \cite{fiedler98} and also an excellent overview in
	\cite{devriendt2020effective}. Note that if
	$\Gamma \in \mathcal C^{d}$, then $\Sigma $ and
	$\tilde\Sigma ^{(k)}$ have rank $d-1$. It implies that we can assume
	$k=d-1$.
	\begin{lemma}
		\label{lemA.9}
		If $\Theta =\Sigma ^{+}$ is a Laplacian matrix of a weighted graph (with
		nonnegative weights), then $\Sigma =U^{T} U$, where
		$U\in \mathbb{R}^{(d-1)\times d}$ and the columns of $U$ are vertices of
		a simplex, whose polar is hyperacute.
	\end{lemma}
	For the proof, see Lemma~1 in \cite{devriendt2020effective}.
	
	As pointed out in Section~D of \cite{devriendt2020effective}, if
	$\Theta $ is a Laplacian matrix of a graph, then the entries of the corresponding
	matrix $\Gamma $ are the effective resistances, and $\Gamma $ is called
	the resistance matrix. The effective resistance allows the bijection between
	simplices, graphs and Laplacian matrices to be summarized beautifully by
	the following identity (see Theorem~2 in
	\cite{devriendt2020effective}).
	\begin{thm}[Fiedler's identity]
		\label{thmA.10}
		For a weighted graph with Laplacian $\Theta $ and resistance matrix
		$\Gamma $, the following identity holds:
		\begin{equation}
			\label{eq:fiedlerid} -\frac{1}{2} 
			\begin{bmatrix} 0 & \boldsymbol
				1^{T}
				\\
				\boldsymbol 1 & \Gamma \end{bmatrix} 
			= 
			\begin{bmatrix}
				4R^{2} & -2\boldsymbol r^{T}
				\\
				-2\boldsymbol r & \Theta \end{bmatrix} 
			^{-1}, 
		\end{equation}
		where
		$\boldsymbol r=\frac {1}{2}\Theta \xi +\frac {1}{d}\boldsymbol 1$ with
		$\xi =\operatorname{diag}(\Sigma )$, and
		$R=\sqrt{\frac {1}{2}\xi ^{T} (\boldsymbol r+\frac {1}{d}
			\boldsymbol 1)}$.
	\end{thm}
	
	As we noted above, $\sqrt{\Gamma _{ij}}$ are always distances in the sense
	that the map $(i, j) \mapsto \sqrt{\Gamma _{ij}}$ is a metric function.
	However, if $\Theta $ is a Laplacian matrix, then the entries of
	$\Gamma $ are effective resistances. By the next lemma we can conclude
	that, in this special case, the entries of $\Gamma $ form a metric (see
	\cite{klein1993resistance}).
	\begin{lemma}
		\label{lem:eff_res}
		If $\Theta _{ij}\leq 0$ for all $i\neq j$, then the effective resistance
		$(i, j) \mapsto \Gamma _{ij}$ is a metric function.
	\end{lemma}
	The implication in Proposition~\ref{prop:extremal_association} cannot be
	reversed. There are situations when $\Gamma $ is a metric, but
	$\Theta $ is not a Laplacian of a graph. In Proposition~\ref{prop:extremal_association}
	we discussed exact conditions when it happens together with a probabilistic
	interpretation in terms of the association of the extremal function.
	
	The fact that $(i,j)\mapsto \Gamma _{ij}$ is a metric does not give us
	a way to realize this metric as an Euclidean distance metric. The case
	we find particularly interesting is related with tree metrics (see, e.g.,
	\cite{semple2003phylogenetics}). Let $T$ be an undirected tree with
	$d$ leaves labeled with $[d]$. We say that
	$\Gamma \in \mathbb{S}^{d}_{0}$ forms a tree metric over $T$ if there exists
	edge length assignment $\theta _{uv}\geq 0$ for $uv\in T$ such that
	\begin{equation}
		\label{eq:treemet} \Gamma _{ij}=\sum_{uv\in \operatorname{ph}(i,j; T)}
		\theta _{uv}, 
	\end{equation}
	where $\operatorname{ph}(i,j; T)$ denotes the unique path between $i$ and $j$ in
	$T$.
	
	Let $T^{k}$ denote a rooted tree obtained from $T$ by directing all edges
	away from a leaf $k\in [d]$. Note that
	\begin{equation}
		\label{eq:treemet2} \widetilde \Sigma _{ij}^{(k)}=
		\frac{1}{2}(\Gamma _{ik}+\Gamma _{jk}- \Gamma
		_{ij})=\sum_{uv\in \operatorname{ph}(i\wedge j,k;T)}\theta _{uv},
	\end{equation}
	where $i\wedge j$ denotes the most recent common ancestor of $i$ and
	$j$ in the tree $T^{k}$. But this means that the entries of
	$\Sigma ^{(k)}$ lie in the Brownian motion tree model on the tree
	$T^{k}$; see \cite{SUZ20} for more details. We get the following result.
	\begin{prop}
		\label{prop:BMTM}
		The image under $\sigma _{k}$ of the set of all tree metrics over a given
		tree $T$ is equal to the set of covariance matrices of the Brownian motion
		tree model over the rooted tree $T^{k}$.
	\end{prop}
	
	We finish by noting that the observation that the square root of a tree
	metric has an Euclidean embedding (which is a side product of this analysis)
	has been important for understanding some algorithms in phylogenetics
	\cite{layer2017phylogenetic}.
	
	\section{Strong \texorpdfstring{$ \mathrm{MTP}_{2} $}{MTP2} distributions}
	\label{s:smtp2}
	
	In this section we collect some new results on strongly
	$ \mathrm{MTP}_{2} $ distributions that will be later used in Appendix~\ref{app:proofs}
	to prove properties of $ \mathrm{EMTP}_{2} $ distributions. Results of this
	section may be of independent interest. The following lemma offers a useful
	characterization of strong $ \mathrm{MTP}_{2} $ distributions, as defined
	in \eqref{eq:LLC}.
	\begin{lemma}
		\label{lem:LLC2}
		The function $f$ is strongly $ \mathrm{MTP}_{2} $ if and only if for all
		$\mathbf{x},\mathbf{y}\in \mathbb{R}^{d}$, and for all $s,t\geq c$ (for
		some $c\in \mathbb{R}$)
		\begin{align}
			f(\mathbf{x}-s\boldsymbol 1) f(\mathbf{y}-t\boldsymbol 1) \leq f\bigl(
			\mathbf{x}\vee \mathbf{y}-(s\vee t)\boldsymbol 1 \bigr) f\bigl(\mathbf{x} \wedge
			\mathbf{y}-(s\wedge t)\boldsymbol 1\bigr). \label{eq:LLC2} 
		\end{align}
	\end{lemma}
	\begin{proof}
		We first show that the strongly $ \mathrm{MTP}_{2} $ condition
		\eqref{eq:LLC} implies the alternative \eqref{eq:LLC2}. Suppose
		$t\geq s\geq c$. Let $\tilde{\mathbf{x}}=\mathbf{y}-t\boldsymbol 1$,
		$\tilde{\mathbf{y}}=\mathbf{x}-s\boldsymbol 1$ and
		$\alpha =t-s\geq 0$. We have
		$\tilde{\mathbf{x}}+\alpha \boldsymbol 1=\mathbf{y}-s\boldsymbol 1$ and
		$\tilde{\mathbf{y}}-\alpha \boldsymbol 1=\mathbf{x}-t\boldsymbol 1$. By
		\eqref{eq:LLC}
		\begin{align*}
			f(\mathbf{x}-s\boldsymbol 1)f(\mathbf{y}-t\boldsymbol 1)&=f( \tilde{
				\mathbf{x}})f(\tilde{\mathbf{y}})\leq f\bigl((\tilde{\mathbf{x}}+ \alpha
			\boldsymbol 1)\wedge \tilde{\mathbf{y}}\bigr)f\bigl(\tilde{\mathbf{x}} \vee (
			\tilde{\mathbf{y}}-\alpha \boldsymbol 1)\bigr)
			\\
			&=f(\mathbf{x}\wedge \mathbf{y}-s\boldsymbol 1)f(\mathbf{x}\vee \mathbf{y}-t
			\boldsymbol 1),
		\end{align*}
		which is exactly \eqref{eq:LLC2}. If $c\leq t<s$, then we proceed in exactly
		the same way taking $\tilde{\mathbf{x}}=\mathbf{x}-s\boldsymbol 1$,
		$\tilde{\mathbf{y}}=\mathbf{y}-t\boldsymbol 1$ and
		$\alpha =s-t\geq 0$. This proves one implication. The other implication
		is obtained by reversing this argument. Fix $\mathbf{x}$,
		$\mathbf{y}$ and $\alpha \geq 0$. Suppose that \eqref{eq:LLC2} holds, and
		take $\tilde{\mathbf{x}}=\mathbf{x}+s\boldsymbol 1$,
		$\tilde{\mathbf{y}}=\mathbf{y}+t\boldsymbol 1$, $t=c$, $s=c+\alpha $. Then
		by \eqref{eq:LLC2},
		\begin{align*}
			f(\mathbf{x})f(\mathbf{y})&=f(\tilde{\mathbf{x}}-s\boldsymbol 1)f( \tilde{
				\mathbf{y}}-t\boldsymbol 1)\leq f(\tilde{\mathbf{x}}\wedge \tilde{\mathbf{y}}-t
			\boldsymbol 1)f(\tilde{\mathbf{x}}\vee \tilde{\mathbf{y}}-s\boldsymbol 1)
			\\
			&=f\bigl((\mathbf{x}+\alpha \boldsymbol 1)\wedge \mathbf{y}\bigr)f\bigl(
			\mathbf{x} \vee (\mathbf{y}-\alpha \boldsymbol 1)\bigr),
		\end{align*}
		which is exactly \eqref{eq:LLC}.
	\end{proof}
	
	We are now ready to prove Theorem~\ref{t:LLCetal}.
	\begin{proof}[Proof of Theorem~\ref{t:LLCetal}]
		Observe that the density $f$ of $\mathbf{Z}$ satisfies
		\begin{equation}
			\label{eq:jointyy0} f(z_{0},\mathbf{z})=f_{0}(z_{0})f_{\mathbf{X}}(
			\mathbf{z}-z_{0} \boldsymbol 1). 
		\end{equation}
		Proof of statement (1): Using \eqref{eq:jointyy0}, the
		$ \mathrm{MTP}_{2} $ constraint on $f$ is equivalent to
		\begin{equation}
			\label{eq:auxmtp} f_{\mathbf{X}}\bigl(\mathbf{x}\vee \mathbf{y}-(x_{0}
			\vee y_{0}) \boldsymbol 1\bigr)f_{\mathbf{X}}\bigl(\mathbf{x}\wedge
			\mathbf{y}-(x_{0} \wedge y_{0})\boldsymbol 1\bigr) \geq
			f_{\mathbf{X}}(\mathbf{x}-x_{0} \boldsymbol 1)f_{\mathbf{X}}(
			\mathbf{y}-y_{0}\boldsymbol 1) 
		\end{equation}
		for all $(x_{0},\mathbf{x})$, $(y_{0},\mathbf{y})$, where we used the fact
		that
		$f_{0}(x_{0})f_{0}(y_{0})=f_{0}(x_{0}\vee y_{0})f_{0}(x_{0}\wedge y_{0})$
		(with both sides nonzero). But this condition is exactly equivalent to
		$f_{\mathbf{X}}$ being strongly $ \mathrm{MTP}_{2} $ by Lemma~\ref{lem:LLC2}.
		
		Proof of statement (2): We first show the left implication via the equivalent
		characterization of strong $ \mathrm{MTP}_{2} $ from Lemma~\ref{lem:LLC2}.
		Denoting $\bar{\mathbf{x}}=(x_{0},\mathbf{x})$ and
		$\bar{\mathbf{y}}=(y_{0},\mathbf{y})$, we want to show that, for every
		$s\leq t$, it holds that
		\begin{equation*}
			f(\bar{\mathbf{x}}-s\boldsymbol 1)f(\bar{\mathbf{y}}-t\boldsymbol 1) \leq f(
			\bar{\mathbf{x}}\wedge \bar{\mathbf{y}}-s\boldsymbol 1)f( \bar{\mathbf{x}}\vee
			\bar{\mathbf{y}}-t\boldsymbol 1),
		\end{equation*}
		where $f$ as in \eqref{eq:jointyy0}. The left-hand side of this inequality
		is
		\begin{equation*}
			L:= f_{0}(x_{0}-s)f_{0}(y_{0}-t)f_{X}(
			\mathbf{x}-x_{0}\boldsymbol 1)f_{X}( \mathbf{y}-y_{0}
			\boldsymbol 1),
		\end{equation*}
		and the right-hand side is
		\begin{equation*}
			R:=f_{0}(x_{0}\wedge y_{0}-s)f_{0}(x_{0}
			\vee y_{0}-t)f_{X}\bigl(\mathbf{x} \wedge
			\mathbf{y}-(x_{0}\wedge y_{0})\boldsymbol 1
			\bigr)f_{X}\bigl(\mathbf{x} \vee \mathbf{y}-(x_{0}\vee
			y_{0})\boldsymbol 1\bigr).
		\end{equation*}
		In the proof of statement (1), we established that the strongly
		$ \mathrm{MTP}_{2} $ property of $\mathbf{X}$ gives that $\mathbf{Z}$ is
		$ \mathrm{MTP}_{2} $. In other words, the inequality \eqref{eq:auxmtp} holds.
		Given this inequality, to show $L\leq R$ it is certainly enough to show
		that
		\begin{equation*}
			f_{0}(x_{0}-s)f_{0}(y_{0}-t)\leq
			f_{0}(x_{0}\wedge y_{0}-s)f_{0}(x_{0}
			\vee y_{0}-t),
		\end{equation*}
		which holds if $f_{0}$ is strongly $ \mathrm{MTP}_{2} $ by Lemma~\ref{lem:LLC2}.
		For the other direction, note that by statement (1), we have that
		$ \mathbf{X}$ is strongly $ \mathrm{MTP}_{2} $ as $ \mathbf{Z}$ is
		$ \mathrm{MTP}_{2} $. To conclude that $X_{0}$ is strongly
		$ \mathrm{MTP}_{2} $, we will show that strongly $ \mathrm{MTP}_{2} $ distributions
		are closed under marginalization. This will conclude the proof of the second
		statement, as $X_{0}$ is a component of $\mathbf{Z}$.

		Strongly $ \mathrm{MTP}_{2} $ distributions are closed under taking margins:
		Suppose that the random vector $\mathbf{X}=(X_{1},\ldots ,X_{d})$ is strongly
		$ \mathrm{MTP}_{2} $. We will show that $(X_{1},\ldots ,X_{d-1})$ is strongly
		$ \mathrm{MTP}_{2} $. By statement (1) $\mathbf{X}$ is strongly
		$ \mathrm{MTP}_{2} $ if and only if for every $X_{0}$ independent of
		$\mathbf{X}$ and supported on $[c,\infty )$ for some fixed
		$c\in \mathbb{R}$, the vector
		$(X_{0},X_{0}+X_{1},\ldots ,X_{0}+X_{d})$ is $ \mathrm{MTP}_{2} $. By the
		closure property of the $ \mathrm{MTP}_{2} $ distributions, the vector
		$(X_{0},X_{0}+X_{1},\ldots ,X_{0}+X_{d-1})$ is also
		$ \mathrm{MTP}_{2} $ for every such $X_{0}$. Again, using statement (1), this
		is equivalent to $(X_{1},\ldots ,X_{d-1})$ being strongly
		$ \mathrm{MTP}_{2} $. The same argument applies to any other margin.
	\end{proof}
	
	In the proof of Theorem~\ref{t:LLCetal}, we showed that also strong
	$ \mathrm{MTP}_{2} $ is closed under taking margins.
	\begin{prop}
		\label{prop:LLCmargin}
		If $\mathbf{X}$ is strongly $ \mathrm{MTP}_{2} $, then every margin of
		$\mathbf{X}$ is strongly $ \mathrm{MTP}_{2} $.
	\end{prop}
	
	\begin{ex}
		\label{exmp10}
		It is useful to see Theorem~\ref{t:LLCetal} in action in the context of
		Gaussian distributions. If $\Sigma $ is the covariance matrix of
		$\mathbf{X}$ and $v$ is the variance of $X_{0}$, then the covariance of
		$\mathbf{Z}$ has the block form
		\begin{equation*}
			\begin{bmatrix} v & v\boldsymbol 1^{T}
				\\
				v\boldsymbol 1 & \Sigma +v\boldsymbol 1\boldsymbol 1^{T} \end{bmatrix}
		\end{equation*}
		with the inverse
		\begin{equation*}
			\begin{bmatrix} v^{-1}+\boldsymbol 1^{T}\Sigma
				^{-1}\boldsymbol 1 & -\boldsymbol 1^{T} \Sigma ^{-1}
				\\
				-\Sigma ^{-1}\boldsymbol 1 & \Sigma ^{-1} \end{bmatrix}
			.
		\end{equation*}
		It is then clear that this is an M-matrix (equiv. $\mathbf{Y}$ is
		$\mathrm{MTP}_{2}$) if and only if $\Sigma ^{-1}$ is an M-matrix with
		$\boldsymbol 1^{T}\Sigma ^{-1}\geq 0$ (equiv. $\mathbf{X}$ is strongly
		$\mathrm{MTP}_{2}$), which is precisely part 1 of the theorem. For the second
		part, we note that the last matrix above has row sums
		$(v^{-1},0,\ldots ,0)$. If $\mathbf{X}$ is strongly
		$ \mathrm{MTP}_{2} $, then it also forms an M-matrix, and so
		$\mathbf{Z}$ is strongly $ \mathrm{MTP}_{2} $.
	\end{ex}
	
	Another important property that we mentioned in Section~\ref{s:posdep}
	is that univariate distributions are always $ \mathrm{MTP}_{2} $. This result
	is not true for strong $ \mathrm{MTP}_{2} $. A~random vector
	$\mathbf{X}$ is log-concave if its density is log-concave.
	
	\begin{prop}
		\label{p:univ}
		A univariate distribution with density $f:\mathbb{R}\to \mathbb{R}$ is
		strongly $ \mathrm{MTP}_{2} $ if and only if it is log-concave. A~random vector
		$\mathbf{X}=(X_{1},\ldots ,X_{d})$ with independent components is strongly
		$ \mathrm{MTP}_{2} $ if and only if each $X_{i}$ is log-concave.
	\end{prop}
	\begin{proof}
		Note that \eqref{eq:LLC} with $d=1$ becomes nontrivial, only if
		$0\leq \alpha \leq y-x$ in which case if gives
		$f(y-\alpha )f(x+\alpha )-f(x)f(y)\geq 0$. Let
		$F:\mathbb{R}\to \mathbb{R}\cup \{-\infty \}$ be defined by
		$F(x)=-\log f(x)$. Denoting $\alpha =\lambda (y-x)$ for
		$\lambda \in [0,1]$, we get
		\begin{align*}
			& \bigl(F\bigl((1-\lambda )x+\lambda y\bigr)-(1-\lambda )F(x)-\lambda F(y)
			\bigr)
			\\
			&\quad {}+ \bigl(F\bigl(\lambda x+(1-\lambda ) y\bigr)-\lambda F(x)-(1-\lambda )
			F(y) \bigr)\leq 0.
		\end{align*}
		Taking $\lambda =\frac{1}{2}$, we conclude midpoint convexity of $F$, which
		is equivalent to concavity as $x<y$ are arbitrary. On the other hand, convexity
		of $F$ trivially implies the above inequality, which is equivalent to the
		strongly $ \mathrm{MTP}_{2} $ inequality.
		
		For the second statement, let
		$F(\mathbf{x})=-\log f_{X}(\mathbf{x})$ so that
		$F(\mathbf{x})=\sum_{i} F_{i}(x_{i})$, where $F_{i}=-\log f_{i}$. We want
		to show that for each $\mathbf{x},\mathbf{y}\in \mathbb{R}^{d}$,
		$\alpha \geq 0$
		\begin{equation*}
			\sum_{i} \bigl(F_{i}(x_{i})+F_{i}(y_{i})-F_{i}
			\bigl(x_{i}\vee (y_{i}- \alpha )\bigr)-F_{i}
			\bigl((x_{i}+\alpha )\wedge y_{i}\bigr) \bigr)\geq 0
		\end{equation*}
		if and only if each summand is nonnegative. The left implication is obvious.
		But the right implication is also clear using the insights of the proof
		of the univariate case. Simply take $\mathbf{x}$, $\mathbf{y}$ such that
		$x_{i}>y_{i}$ for all $i\neq k$. The corresponding summands are zero and
		so necessarily
		$F_{k}(x_{k})+F_{k}(y_{k})-F_{k}(x_{k}\vee (y_{k}-\alpha ))-F_{i}((x_{k}+
		\alpha )\wedge y_{k})\geq 0$.
	\end{proof}
	
	\subsection{Log-concave tree processes}
	\label{secB.1}
	
	Theorem~\ref{t:LLCetal} and Proposition~\ref{p:univ} give a natural way
	to construct multivariate strongly $ \mathrm{MTP}_{2} $ distributions with
	log-concave distributions. If $X_{1},\ldots ,X_{d}$ are univariate log-concave
	and independent, then
	\begin{equation}
		\label{eq:simplepath} (X_{1},X_{1}+X_{2},X_{1}+X_{2}+X_{3},
		\ldots ,X_{1}+\cdots +X_{d}) 
	\end{equation}
	is strongly $ \mathrm{MTP}_{2} $ and log-concave. We now provide a generalization
	of this construction.
	
	Let $T=(V,E)$ be an undirected tree with vertex set
	$V=\{1,\ldots ,d\}$. A~rooted tree $T^{k}$ is a tree obtained from
	$T$ by choosing a vertex $k\in V$, called the root, and directing all edges
	away from $k$. For any two nodes $i$, $j$ in an undirected tree $ T $, we
	denote by $\operatorname{ph}(ij;T)$ the set of edges on the (unique) path between
	$i$ and $j$ in this tree. Equivalently, for a rooted tree $ T^{k} $, let
	$\operatorname{ph}(ij;T^{k})$ be the set of directed edges on the (unique) path
	from $ i $ to $ j $ in $ T^{k} $.
	\begin{defi}
		\label{defn3}
		For a given rooted tree $T^{k}$ with vertices $V=\{1,\ldots ,d\}$, let
		$X_{1},\ldots ,X_{d}$ be a collection of independent random variables.
		Let $\mathbf{Z}=(Z_{1},\ldots ,Z_{d})$ be defined by
		\begin{equation}
			\label{eq:sumpath} Z_{i} = X_{k} +\sum
			_{uv\in \operatorname{ph}(ki;T^{k})} X_{v}, 
		\end{equation}
		where $uv$ denotes a directed edge $u\to v$. Then we say that
		$\mathbf{Z}$ follows an additive process on $T^{k}$. If all $X_{i}$ are
		log-concave, then we call such process a log-concave process on a tree.
	\end{defi}
	For example, if $d=3$, the vector in \eqref{eq:simplepath} forms an additive
	process on the tree $1\to 2\to 3$. Rerooting this tree at $2$ results in
	an additive process $(X_{1}+X_{2},X_{2},X_{2}+X_{3})$.
	
	\begin{prop}
		\label{prop:treecon}
		If $\mathbf{Z}$ follows a log-concave tree process, then it has a strongly
		$ \mathrm{MTP}_{2} $ and log-concave distribution.
	\end{prop}
	\begin{proof}
		Since $X_{k}$ is log-concave, then by Theorem~\ref{t:LLCetal} the vector
		$ \mathbf{Z}$ is strongly $ \mathrm{MTP}_{2} $ as long as
		$(Z_{i}-Z_{k})_{i\neq k}$ is strongly $ \mathrm{MTP}_{2} $. This vector can
		be split into independent components indexed by the children of $k$ in
		$T^{k}$. In each of the components, we apply the same argument recursively.
		The fact that concatenating independent strongly $ \mathrm{MTP}_{2} $ vectors
		gives a strongly $ \mathrm{MTP}_{2} $ vector is clear.
	\end{proof}
	A special case of the construction in Proposition~\ref{prop:treecon} is
	when $\mathbf{X}$ is independent zero-mean Gaussian. In this case the set
	of marginal distributions over the leaves of $T$ is called the Brownian
	motion tree model (\citet{felsenstein1973maximum}); see, for example, Section~2
	in \cite{SUZ20} for the structural equation representation, as in
	\eqref{eq:sumpath}. Since the distribution of $\mathbf{Z}$ is strongly
	$ \mathrm{MTP}_{2} $, this property is preserved in the margin by Proposition~\ref{prop:LLCmargin}.
	We recover a well-known fact that the inverse covariance matrix in a Brownian
	motion tree model is always a diagonally dominant M-matrix; see, for example,
	\cite{dellacherie2014inverse}, where the covariance matrices in the Brownian
	motion tree model are called simply tree matrices. In extreme value theory,
	we encounter such construction in the context of the extremal tree models
	(see Section~\ref{s:tree_models}).
	
	\section{Auxiliary results and proofs}
	\label{app:proofs}
	
	\subsection{The exponent measure \texorpdfstring{$\Lambda$}{Lambda}}
	\label{exp_measure}
	
	We use here the notation of Section~\ref{s:MPD_prop}.
	
	In order to describe the extremal dependence structure, the assumption
	of multivariate regular variation is widely used (\citet{res2008}). Formally,
	it is equivalent to the existence of the limit
	\begin{align}
		\label{mrv} \lim_{u\to \infty} u\bigl[1-\mathbb{P}(\mathbf{X}\le
		\mathbf{z}+u \boldsymbol 1)\bigr]=\Lambda (\mathbf{z}) 
	\end{align}
	for all
	$\mathbf z \in \mathcal E=[-\infty ,\infty )^{d}\setminus \{(-\infty ,
	\ldots ,-\infty )\}$. The exponent measure $ \Lambda $ is a Radon measure
	on $ \mathcal{E} $, and $\Lambda (\mathbf z)$ is the short-hand notation
	for
	$\Lambda (\mathcal E \setminus [(-\infty ,\ldots ,-\infty ),
	\mathbf{z}])$. The fact that $\Lambda $ arises as a limit in \eqref{mrv}
	implies a homogeneity property
	$\Lambda (\mathbf{z}+t\boldsymbol 1) = t^{-1} \Lambda (\mathbf{z})$, for
	any $t > 0$. If we assume that $\Lambda $ possesses a positive Lebesgue
	density $\lambda $, then it satisfies
	$\lambda (\mathbf{y}+t\boldsymbol 1)=t^{-1}\lambda (\mathbf{y}) $ for any
	$t>0$ and $ \mathbf{y}\in \mathbb{R}^{d}$. The $I$th marginal
	$\lambda _{I}$ of $ \lambda $ is defined for any nonempty
	$ I\subset [d]:=\{1,\ldots ,d\} $, as usual by integrating out all components
	in $[d] \setminus I$.
	
	The relation of the exponent measure to the multivariate Pareto distribution
	$\mathbf{Y}$ is the following:
	\begin{align}
		\label{MPD_conv_2} \mathbb{P}(\mathbf{Y}\le \mathbf{z})= \frac{\Lambda (\mathbf{z}\wedge \boldsymbol 0)-\Lambda (\mathbf{z})}{\Lambda (\boldsymbol 0)}, \quad
		\mathbf{z}\in \mathcal L. 
	\end{align}
	From this it follows that the density $f_{\mathbf{Y}}$ of
	$\mathbf{Y}$ satisfies
	$f_{\mathbf{Y}}(\mathbf{y})=\lambda (\mathbf{y})/\Lambda (
	\boldsymbol 0)$. Similarly, the random vector $ \mathbf{Y}^{k} $ has Lebesgue
	density $\lambda $ supported on the product space $ \mathcal{L}^{k}$.
	
	\subsection{Proof of Theorem~\protect\ref{t:mtp2}}
	\label{sec100.2}
	
	We first prove the last statement that the condition holds for one
	$k$ if and only if it holds for all $k$. Note that $ \mathbf{Y}^{k}$ has
	density $\lambda (\mathbf{y}) $ and is supported on the product space
	$ \mathcal{L}^{k} $. Assume that $ \mathbf{Y}^{k} $ is
	$ \mathrm{MTP}_{2} $, that is, that $ \lambda (\mathbf{y}) $ satisfies
	\eqref{eq:log-supermodular} on $ \mathcal{L}^{k} $. For any
	$ k' \in [d] $, there exists some $ t>0 $ such that, for any
	$ \mathbf{x},\mathbf{y}\in \mathcal{L}^{k'} $, it holds that
	$ \mathbf{x}+t\boldsymbol 1, \mathbf{y}+t\boldsymbol 1, (\mathbf{x}
	\wedge \mathbf{y})+t\boldsymbol 1, (\mathbf{x}\vee \mathbf{y})+t
	\boldsymbol 1 \in \mathcal{L}^{k} $. Hence, it follows from the homogeneity
	of $ \lambda $ that
	\begin{align*}
		\lambda (\mathbf{x}\vee \mathbf{y}) \lambda (\mathbf{x}\wedge \mathbf{y})&=
		t^{2} \lambda \bigl((\mathbf{x}\vee \mathbf{y})+t \boldsymbol 1 \bigr)
		\lambda \bigl((\mathbf{x}\wedge \mathbf{y})+t\boldsymbol 1\bigr)
		\\
		&\ge t^{2}\lambda (\mathbf{x}+t\boldsymbol 1) \lambda (\mathbf{y}+t
		\boldsymbol 1)
		\\
		&= \lambda (\mathbf{x}) \lambda (\mathbf{y}).
	\end{align*}
	
	Now, to prove the first statement, we see that it is enough to check this
	condition on one $k$, so without loss of generality take $k=d$. Using
	\eqref{Yk} we have $ Y^{d}_{d}=E$ and $ Y^{d}_{i}=E+W^{d}_{i}$ for
	$i=1,\ldots ,d-1$. Since $E$ is exponentially distributed, we can use Theorem~\ref{t:LLCetal}
	to conclude that $\boldsymbol Y^{d}$ is $ \mathrm{MTP}_{2} $ if and only if
	the vector $\mathbf{W}^{d}_{\setminus d}$ is strongly
	$ \mathrm{MTP}_{2} $.
	
	\subsection{Proof of Proposition~\protect\ref{prop:marginex}}
	\label{sec100.3}
	
	Let $ \mathbf{Y}$ be $ \mathrm{EMTP}_{2} $. This means that for all
	$ k\in V $, the vector $ \mathbf{Y}^{k} $ is $ \mathrm{MTP}_{2} $. Let
	$ \mathbf{Y}_{I} $ be the marginal of $ \mathbf{Y}$ for some
	$ I\subset V $. It holds that
	\begin{equation*}
		\mathbf{Y}_{I} | Y_{k}>0\stackrel{d} {=}\bigl(
		\mathbf{Y}^{k}\bigr)_{I} \quad \text{for all } k\in I;
	\end{equation*}
	see \cite{EH2020}. As $ \mathrm{MTP}_{2} $ is closed under taking margins,
	the proposition follows.
	
	\subsection{Proof of Theorem~\protect\ref{t:equivalences_precision_matrices_emtp2}}
	\label{sec100.4}
	
	By Theorem~\ref{t:mtp2}, $\mathbf{Y}$ is $ \mathrm{EMTP}_{2} $ if and only
	if each $ \mathbf{W}^{k}$ is strongly $ \mathrm{MTP}_{2} $. By Example~\ref{ex:gauss},
	equivalently, each $\Theta ^{(k)}$ is a diagonally dominant M-matrix. This
	establishes that $\mathbf{Y}$ being $ \mathrm{EMTP}_{2} $ is equivalent to
	(ii).
	
	On the other hand, $\Theta $ being a Laplacian of a connected graph with
	positive edge weights is equivalent with the seemingly simpler condition
	(i). The fact that $\Theta \in \mathbb U^{d}_{+}$ implies (i) is clear.
	For the other direction, note that, by Proposition~\ref{prop:manu}, the
	row sums of $\Theta $ are zero and $\Theta _{ij}\leq 0$ for all
	$i\neq j$. It follows that $\Theta $ is a Laplacian matrix of a graph weighted
	with $Q_{ij}=-\Theta _{ij}\geq 0$. Since $\Gamma $ is strictly conditionally
	negative definite, again by Proposition~\ref{prop:manu},
	$\operatorname{rank}(\Theta )=d-1$. It then follows thaat
	$\det (\Theta ^{(k)})>0$. By \eqref{eq:det_spanning_trees} we conclude
	that at least one of the tree terms is strictly positive, proving that
	the underlying graph is connected, that is,
	$\Theta \in \mathbb{U}^{d}_{+}$.
	
	The proof will be then concluded if we establish equivalence between conditions
	(i)--(iv). Note first that all these conditions hold simultaneously for
	all bivariate H\"usler--Reiss distributions; compare Example~\ref{ex:bivariate_HR_mtp2}.
	Thus, we assume $d\geq 3$. In all cases we heavily rely on Lemma~\ref{lem:Theta}.
	In particular, if (i) holds, then the formula
	$\Theta ^{(k)}_{ij}=\Theta _{ij}$ for $i,j\neq k$ implies that
	$\Theta ^{(k)}$ is an M-matrix and the formula
	$\sum_{l\neq k}\Theta ^{(k)}_{il}=-\Theta _{ik}$ implies that each
	$\Theta ^{(k)}$ is also diagonally dominant. In other words, (i) implies
	all (ii), (iii) and (iv).
	
	Since (iv) is weaker than (ii), it remains to show that both (iii) and
	(iv) imply (i).
	
	(iii) $\Rightarrow $ (i): To show that $\Theta _{ij}\leq 0$, we take any
	$k\neq i,j$ (there will be at least one as $d\geq 3$) and use again the
	formula $\Theta _{ij}=\Theta ^{(k)}_{ij}\leq 0$.
	
	(iv) $\Rightarrow $ (i): Suppose $\Theta ^{(k)}$ is a diagonally dominant
	M-matrix. By the same argument as above, this is enough to conclude
	$\Theta _{ij}\leq 0$ for all $i,j\neq k$. Similarly, as above,
	$\Theta _{ik}=-\sum_{l\neq k}\Theta ^{(k)}_{il}\leq 0$ because
	$\Theta ^{(k)}$ is diagonally dominant.
	
	\subsection{Proof of Proposition~\protect\ref{prop:extremal_association}}
	\label{sec100.5}
	
	For a H\"usler--Reiss random vector $ \mathbf{Y}$, it holds that
	$ \mathbf{W}^{k} $ is associated if and only if $ \Sigma ^{(k)} $ is nonnegative.
	Since $\mathbf{W}^{k}_{\setminus k}$ is Gaussian with covariance matrix
	$\Sigma ^{(k)}$, this is equivalent with $\Sigma ^{(k)}\geq 0$ by
	\cite{Pitt1982}. Now, the result follows from the fact that
	\begin{align*}
		\Sigma _{ij}^{(k)}\ge 0 &\quad \Longleftrightarrow\quad \Sigma
		_{ii}^{(k)}+ \Sigma _{jj}^{(k)} \ge \Sigma
		_{ii}^{(k)}+\Sigma _{jj}^{(k)}-2\Sigma
		_{ij}^{(k)} \quad \Longleftrightarrow \quad \Gamma
		_{ik}+\Gamma _{jk}\ge \Gamma _{ij}
	\end{align*}
	for all $ i,j,k\in [d] $.
	
	\subsection{Proof of Proposition~\protect\ref{prop:indep_gen}}
	\label{sec100.6}
	
	By Theorem~\ref{t:mtp2}, $\mathbf{Y}$ is $ \mathrm{EMTP}_{2} $ if and only
	if $\mathbf{W}^{d}_{\setminus d}=(U_{i}-U_{d})_{i\neq d}$ is strongly
	$ \mathrm{MTP}_{2} $. Let $X_{0}=-U_{d}$ and $X_{i}=U_{i}$ for
	$i=1,\ldots ,d-1$. Denoting $\mathbf{X}=(X_{1},\ldots ,X_{d-1})$, the fact
	that $\mathbf{X}$ is strongly $ \mathrm{MTP}_{2} $ follows from Proposition~\ref{p:univ},
	as the distribution of each $X_{i}$ is log-concave. By assumption the distribution
	of $-X_{0}$ is also log-concave, and so the distribution of $X_{0}$ must
	be log-concave. It follows from Theorem~\ref{t:LLCetal} that the vector
	$(X_{0},X_{0}+X_{1},\ldots ,X_{0}+X_{d-1})$ is strongly
	$ \mathrm{MTP}_{2} $. Since the strong $ \mathrm{MTP}_{2} $ property is closed
	under taking margins (Proposition~\ref{prop:LLCmargin}), we conclude that
	\begin{equation*}
		(X_{0}+X_{1},\ldots ,X_{0}+X_{d-1}) = (
		U_{1}- U_{d},\ldots , U_{d-1}- U_{d})
	\end{equation*}
	is strongly $ \mathrm{MTP}_{2} $, proving that $\mathbf{Y}$ is
	$ \mathrm{EMTP}_{2} $.
	
	\subsection{Proof of Theorem~\protect\ref{thm:bivariate_emtp2}}
	\label{sec100.7}
	
	By Theorem~\ref{t:mtp2}, $ \mathrm{EMTP}_{2} $ is equivalent to
	$ W^{1}_{2}$ being strongly $ \mathrm{MTP}_{2} $. But this is a univariate
	random variable, so equivalently, its density must be log-concave by Proposition~\ref{p:univ}.
	
	\subsection{Proof of Proposition~\protect\ref{prop:trees_mtp2}}
	\label{sec100.8}
	
	By Theorem~\ref{t:mtp2}, we need to check whether
	$ \mathbf{W}^{k}_{\setminus k} $ is strongly $ \mathrm{MTP}_{2} $. By
	\eqref{eq:tree_representation} it follows that
	\begin{equation*}
		Y^{k}_{i}=E + \sum_{e\in \operatorname{ph}(ki;T^{k})}
		W_{e},
	\end{equation*}
	where $ W_{e},  e\in T^{k} $ are independent. As $E$ is log-concave, by
	Proposition~\ref{prop:treecon} we have that $\mathbf{Y}^{k}$ is strongly
	$ \mathrm{MTP}_{2} $ if and only if $ W_{e} $ is log-concave for each
	$ e\in E $. This is equivalent to $ \mathbf{W}^{k}_{\setminus k} $ being
	strongly $ \mathrm{MTP}_{2} $ by Theorem~\ref{t:LLCetal}(2). From Example~\ref{ex:bivariate_HR_mtp2}
	it then follows that H\"usler--Reiss tree models are always
	$ \mathrm{EMTP}_{2} $.
	
	\subsection{Proof of Theorem~\protect\ref{t:faithful}}
	\label{pr:faithful}
	
	Since by definition $ \mathbf{Y}\sim \mathbb P_{\mathbf Y} $ satisfies
	the pairwise Markov property with respect to its pairwise independence
	graph $ G_{e}(\mathbb P_{\mathbf Y}) $ and since it has a positive and
	continuous density, it also satisfies the global Markov property. Indeed,
	since the density of $\mathbf{Y}^{k}$ is proportional to the density of
	$\mathbf{Y}$, it is also positive and continuous.
	\citet[Chapter~3]{Lauritzen96} shows the equivalence of the pairwise
	and the global Markov property in this case for $\mathbf{Y}^{k}$, and this
	then propagates to the corresponding extremal conditional independence
	properties by Definition~\ref{def:extremal_indep}. Therefore,
	$ \mathbf{Y}$ satisfies the global Markov property with respect to
	$ G_{e}(\mathbb P_{\mathbf Y}) $.
	
	Assume disjoint $ A,B,C\subset V $ such that $ C $ does not separate
	$ A $ from $ B $ in $ G_{e}(\mathbb P_{\mathbf Y}) $. We need to show that
	that
	$ \mathbf{Y}_{A} \not \perp _{e} \mathbf{Y}_{B} | \mathbf{Y}_{C} $ to
	conclude that $ \mathbb P_{\mathbf Y} $ is extremal faithful to
	$ G_{e}(\mathbb P_{\mathbf Y}) $. To see this, let
	$ (k,l)\in E(G_{e}(\mathbb P_{\mathbf Y})) $ with $ k,l\notin C $, which
	means that
	$ Y_{k} \not \perp _{e} Y_{l} | \mathbf{Y}_{V\setminus kl} $. As
	$\mathbf{Y}$ is $ \mathrm{EMTP}_{2} $, upward-stability (see Theorem~\ref{thm:us_st_comp})
	implies that $ Y_{k} \not \perp _{e} Y_{l} | \mathbf{Y}_{C} $. As
	$ C $ does not separate $ A $ from $ B $, there is a path in
	$ G_{e}(\mathbb P_{\mathbf Y}) $ from some $ i\in A $ to some
	$ j\in B $ that does not intersect $ C $. It holds that
	$ Y_{s} \not \perp _{e} Y_{t} | \mathbf{Y}_{C} $ for any edge
	$ (s,t) $ on the path from $ i $ to $ j $. As $ \mathbf{Y}$ is
	$ \mathrm{EMTP}_{2} $, it satisfies singleton-transitivity (see Theorem~\ref{thm:us_st_comp})
	such that for edges $(s,t)$, $(t,u)$ on the path from $i$ to $j$ it follows
	that
	$ Y_{s} \not \perp _{e} Y_{u} | \mathbf{Y}_{C} \vee Y_{s} \not
	\perp _{e} Y_{u} | \mathbf{Y}_{tC} $. Now, using again upward stability,
	we obtain that $ Y_{s} \not \perp _{e} Y_{u} | \mathbf{Y}_{tC} $ implies
	$ Y_{s} \not \perp _{e} Y_{u} | \mathbf{Y}_{C} $. This gives that
	$ Y_{i} \not \perp _{e} Y_{j} | \mathbf{Y}_{C} $. With Remark~\ref{rem:singletonCI}
	the theorem follows.
	
	\subsection{Proof of Lemma~\protect\ref{lem:tolapl}}
	\label{sec100.10}
	
	We use the notation of Appendix~\ref{app:Gamma}, where
	$\tilde S^{(k)}$ is the embedding of $S^{(k)}\in \mathbb{S}^{d-1}$ in
	$\mathbb{S}^{d}$ by adding to it a zero row/column (as the $k$th row/column).
	Similarly, $\tilde\Theta ^{(k)}$ is the pseudoinverse of
	$\tilde\Sigma ^{(k)}$. Note that
	$\langle S^{(k)},\Theta ^{(k)}\rangle =\langle \tilde S^{(k)},
	\tilde \Theta ^{(k)}\rangle $, which is equal to
	$\llangle \overline{\Gamma},Q\rrangle  $ simply because
	both are equal to
	$\llangle \overline{\Gamma},Q\rrangle  $ by Remark~\ref{rem:app}(4).
	This gives
	\begin{equation*}
		\bigl\langle S^{(k)},\Theta ^{(k)}\bigr\rangle =\langle S,\Theta
		\rangle = \llangle \overline \Gamma ,Q\rrangle .
	\end{equation*}
	Moreover, by \eqref{eq:det_spanning_trees}
	\begin{equation*}
		\log \det \Theta ^{(k)} = \log \biggl(\sum_{T\in \mathcal T}
		\prod_{ij
			\in T} Q_{ij}\biggr) = \log
		\operatorname{Det} \Theta -\log (d).
	\end{equation*}
	This proves both \eqref{eq:surrogatelike2} and \eqref{eq:likeinQ}.
\end{appendix}

\subsection*{Acknowledgments}
	The authors would like to thank the two anonymous referees, an Associate
	Editor and the Editor for their constructive comments that strongly improved
	the quality of this paper.

	Sebastian Engelke and Frank R\"ottger were supported by the Swiss National
	Science Foundation (Grant 186858). Piotr Zwiernik acknowledges the support of the Natural Sciences and Engineering Research Council of Canada (NSERC), grant RGPIN-2023-03481.

\bibliographystyle{Chicago}
\bibliography{bibliography.bib}
\end{document}